\documentclass{amsart}
\usepackage[utf8]{inputenc}
\usepackage{amsthm}
\usepackage{amsmath}
\usepackage{tikz}
\usepackage{comment}
\usetikzlibrary{arrows}

\usepackage{enumerate}   
\usepackage{amssymb}
\usepackage{fancyhdr} 
\usepackage[left=3.5cm,right=3.5cm,top=3.5cm,bottom=3.5cm]{geometry}
\usepackage{xcolor}
\newcommand{\norm}[1]{\left\lVert#1\right\rVert}
\usepackage{graphicx}
\usepackage{ esint }
\usepackage[dvipsnames]{xcolor}

\newtheorem{theorem}{Theorem}[section]

\newtheorem{corollary}[theorem]{Corollary}
\newtheorem{lemma}[theorem]{Lemma}
\newtheorem{definition}[theorem]{Definition}
\newtheorem{example}[theorem]{Example}

\theoremstyle{remark}
\newtheorem{question}[theorem]{Question}
\newtheorem*{remark}{Remark}

\newcommand{\veps}{\varepsilon}
\newcommand{\CC}{\mathbb C}
\newcommand{\RR}{\mathbb R}
\newcommand{\DD}{\mathbb D}

\newcommand{\TT}{\mathbb T}
\newcommand{\NN}{\mathbb N}
\newcommand{\ZZ}{\mathbb Z}

\newcommand{\ovdd}{\overline{\DD}^2}
\newcommand{\ovdt}{\overline{\DD}^3}
\newcommand{\ovD}{\overline{\DD}^d}
\newcommand{\oDD}{\overline{\DD}}
\newcommand{\m}[1]{\vert #1 \vert}
\newcommand{\TTI}{\TT^{\m{I}}}
\newcommand{\tk}{\theta_k}
\newcommand{\rk}{\rho_k}
\newcommand{\tl}{\theta_l}

\newcommand{\vb}[1]{V_{\beta_#1}}
\newcommand{\sed}{S(\eta,\ovd)}
\newcommand{\eio}{e^{i\theta}}
\newcommand{\dss}{\displaystyle\sum}
\newcommand{\Ab}{A_\beta^2(\DD^3 )}
\newcommand{\tu}{\theta_1}
\newcommand{\td}{\theta_2}
\newcommand{\tr}{\theta_3}
\newcommand{\I}[1]{\lbrace 1,\dots,#1\rbrace}
\newcommand{\Si}{S_I(\eta,\ovd)}
\newcommand{\Abd}{A^2_\beta(\DD^d)}
\newcommand{\Abdi}[1]{A^2_{\beta_{#1}}(\DD^d)}
\newcommand{\wik}{\omega_{I,k}(\ovd)}
\newcommand{\wikxi}{\omega_{I,k,\xi}(\ovd)}
\newcommand{\dprod}{\displaystyle\prod}
\newcommand{\oii}{(0,1)^{\m{I}}}
\newcommand{\Ei}{E_I(\xi,\eta,\ovd)}

\def\Imm{\mathrm{Im}}
\def\Ree{\mathrm{Re}}
\def\ovd{\overline{\delta}}

\usepackage{url}
\usepackage[hidelinks]{hyperref}
\hypersetup{
	colorlinks=true,
	linkcolor=cyan,
	filecolor=cyan,
	citecolor =cyan,      
	urlcolor=cyan,
}
\author{Frédéric Bayart}
\address{Laboratoire de Math\'ematiques Blaise Pascal UMR 6620 CNRS, Universit\'e Clermont Auvergne, Campus universitaire des C\'ezeaux, 3 place Vasarely, 63178 Aubi\`ere Cedex, France.}
\email{frederic.bayart@uca.fr}
\author{Anne Dorval}
\address{Laboratoire de Math\'ematiques Blaise Pascal UMR 6620 CNRS, Universit\'e Clermont Auvergne, Campus universitaire des C\'ezeaux, 3 place Vasarely, 63178 Aubi\`ere Cedex, France.}
\email{anne.dorval@uca.fr}

\thanks{The authors are partially supported by the grant ANR-24-CE40-0892-01 of the French National Research
		Agency ANR}

\title[Composition Operators]{Composition Operators on weighted Bergman spaces of the polydisc}
\date{}
\begin{document}
\begin{abstract}
We study composition operators between weighted Bergman spaces of the polydisc induced by smooth symbols. We prove a general characterization of continuity which is local around the contact points of the symbol and which, under a natural non-degeneracy assumption, only involves its behaviour  on the polytorus.  We deduce from this several consequences about the automatic continuity of the induced operator. We study in depth the case of the tridisc and exhibit several examples showing that a characterization of continuity using only derivatives seems impossible.
\end{abstract}
\maketitle

\section{Introduction}

Let $\mathcal U$ be a domain in $\CC^d$, let $X$ be a Banach space of holomorphic functions on $\mathcal U$ and let $\phi:\mathcal U\to\mathcal U$
be holomorphic (we will write $\phi\in\mathcal O(\mathcal U,\mathcal U)$). The composition operator with symbol $\phi$ is defined by $C_\phi(f)=f\circ\phi,$ $f\in X$. The first question to solve when studying composition operators is that of continuity: for which symbols $\phi$ does $C_\phi$ induce a bounded composition operator on $X$? When $X$ is the Hardy space or a weighted Bergman space of the unit disc $\DD,$ the answer is easy: by the Littlewood
subordination principle, this is always the case. 

The situation breaks down dramatically on the polydisc $\DD^d,$ $d\geq 2,$ where there are simple examples of symbols $\phi$ such that $C_\phi$ does not map $H^2(\DD^d)$ or any Bergman space $A_\beta^2(\DD^d)$ into itself. 
The papers \cite{BAYPOLY,KSZ08,SZ06b,SZ06a} contain some general results as well as a detailed study of the case of affine symbols. 

It turns out that studying the continuity of composition operators on Hardy and Bergman spaces of the bidisc and the tridisc, with symbols regular up to the boundary, already becomes an issue: on the bidisc, $C_\phi$ is bounded on $A^2_\beta(\DD^2)$ (or on $H^2(\DD^2)$) if and only if $d\phi(\xi)$ is
invertible at every $\xi\in\TT^2$ with $\phi(\xi)\in\TT^2$ \cite{BAYPOLY,Ko22}; on the tridisc, the answer for
$A^2(\DD^3)$ still only involves first-order derivatives \cite{Ko22}, the answer for $H^2(\DD^3)$ requires
second-order ones \cite{Baytridisc}, whereas the answer for general weighted Bergman spaces is still unknown.
We also mention the recent papers \cite{Bes25b,Bes26,Bes25a,CCS25,KL25} on composition operators on the bidisc.


In this paper we investigate composition operators $C_\phi:A^2_{\beta_1}(\DD^d)\to A^2_{\beta_2}(\DD^d)$ for general weights $\beta_1,\beta_2\geq-1$, dimension $d\geq 2$
and smooth symbols. Our first main result, Theorem \ref{result_gen},
gives a necessary and sufficient condition for $C_\phi$ to map boundedly $A^2_{\beta_1}(\DD^d)$ into $A^2_{\beta_2}(\DD^d).$ Although this new characterization does not look simpler than the usual Carleson measure condition, it will imply several general statements (see below)  and under a natural non-degeneracy assumption it will lead to a criterion which only involves the boundary values of $\phi$ and which will be far more tractable in practice.

Theorem \ref{result_gen} has several corollaries which go into opposite conditions.
\begin{itemize}
\item \emph{Stability} (Theorem \ref{thm:stability}): if $C_\phi$ is bounded on $A^2_{\beta_1}(\DD^d)$, it is bounded on
$A^2_{\beta}(\DD^d)$ for all $\beta\geq\beta_1$.
\item \emph{Automatic continuity} (Theorem \ref{thm:automaticcontinuity}): $C_\phi$ always maps $A^2_\beta(\DD^d)$ into
$A^2_{d_\phi(\beta+2)-2}(\DD^d)$, where $d_\phi$ is the maximal number of coordinates of $\phi$ which can touch the
polycircle simultaneously; this improves \cite[Theorem A]{SZ06b} and is optimal.
\item \emph{No characterization by derivatives} (Theorem \ref{thm:nthderivative}): in contrast with \cite{Baytridisc,Ko22}, even
on the tridisc, the knowledge of all derivatives up to any fixed order at the contact points does not determine boundedness on $A^2_\beta(\DD^3)$. For all $n\geq 1$, there exist $\beta\in(-1,0),$ $\phi,\psi\in\mathcal O(\DD^3,\DD^3)\cap \mathcal C^n(\overline{\DD}^3)$ with the same contact points, the same derivatives up to order $n$ at the contact points, such that 
$C_\phi$ is continuous on $A^2_{\beta}(\DD^3)$ and $C_\psi$ is not continuous on $A^2_\beta(\DD^3)$. 
\item \emph{Two different weights require two different characterizations} (Theorem \ref{thm:weights}):  for any $\beta_1,\beta_2\in(-1,0)$ with $\beta_1<\beta_2,$ there exists $\phi\in\mathcal O(\DD^3,\DD^3)\cap \mathcal C^1(\overline{\DD}^3)$ such that $C_\phi$ is continuous on $A^2_{\beta_2}(\DD^3)$ and $C_\phi$ is not continuous on $A^2_{\beta_1}(\DD^3).$
\end{itemize}



Hence it seems impossible to characterize in full generality the continuity of $C_\phi$ on $A_\beta^2(\DD^3),$ $\beta\in(-1,0)$. Nevertheless, we do in Section \ref{sec:tridisc} an in-depth study to understand what can happen when we have 
a full knowledge of the second-order derivative. This allows us to give plenty of examples in Section \ref{sec:examples}. In Section \ref{sec:bidisc}, we come back 
to composition operators on the bidisc and we are interested in the following problem: let $\beta_1>-1$; for which values of $\beta_2\geq\beta_1$ does there exist $\phi\in \mathcal O(\DD^2,\DD^2)\cap\mathcal C^1(\overline{\DD}^2)$ such that $C_\phi$ maps $A^2_{\beta_1}(\DD^2)$ into $A_{\beta}^2(\DD^2)$ if and only if $\beta\geq \beta_2.$ We show that this is the case if and only if $\beta_2\in[\beta_1,\beta_1+1/2]$ or $\beta_2=2\beta_1+2.$ This answers a question 
of \cite{KL25}.

\medskip

\noindent {\bf About this paper.} The first version of this paper was mostly written in 2025 and was posted on arXiv on January 2026. At the end of September 2026, we were informed almost simultaneously by A. Beslikas and by an anonymous referee (we thank both of them) that there was a mistake in the statement of what is now Corollary \ref{result_gen_restricted}, which was formulated without the restricted variables property. This leads us to introduce some new sets $E_I(\xi,\eta,\ovd)$ (see \eqref{eq:EIxieta} below) and to prove Theorem \ref{result_gen} which is a generalized version of Corollary \ref{result_gen_restricted}, without the restricted variables property. As almost all of our examples had the restricted variables property, a large part of the paper is unchanged.

\medskip

\noindent {\bf AI statement.} The first version of the paper was written without any use of AI. For
the revised version, we used AI tools to assist with language revision and proof checking. All the strategies and proofs were obtained by the authors who take full responsibility for the final version of the paper.


\section{Some general results}
\subsection{Notations}Let us start by introducing some notation that will be used throughout this paper. The unit vector $(1,\dots,1)$ will be denoted by $e$. 
For $\theta=(\theta_1,\dots,\theta_d)\in\RR^d,$ we will write $e^{i\theta}$ for $(e^{i\theta_1},\dots,e^{i\theta_d}).$ Any $z\in\overline{\DD}^d$ such that $z_k\neq 0$ for all $k$ will be uniquely written $((1-\rho_k)e^{i\theta_k})$
with $\theta_k\in[-\pi,\pi)$ and $\rho_k\in[0,1).$ 

For  $f : \overline{\DD}^d \rightarrow A$, $\theta \in \RR^d$ and $\rho\in[0,1]^d$, we write $f(\theta)$ for $f(e^{i\tu},\dots,e^{i\theta_d})$ and $f(\rho,\theta)$ for $f(z)$ with $z=((1-\rho_k)e^{i\theta_k}).$
If $f$ maps $A$ into $\CC^d,$ the map $f_I$, where $I=\{i_1,\dots,i_p\}\subset \{1,\dots,d\}$, will denote $(f_{i_1},\dots,f_{i_p}):A\to\CC^p.$ For $P = \{k_1,\dots,k_p\} \subset \I{d}$ and $z \in \ovD$, we denote $z_P = (z_{k_1},\dots,z_{k_p})$. 
For $f,g : A \rightarrow \RR$, we shall write ${f\lesssim g} $ provided there exists $C>0$ such that, for all $x \in A,\ f(x) \leq Cg(x).$ When we write $f\simeq g,$ we mean $f\lesssim g$ and $g\lesssim f$.

The Lebesgue measure on $\RR^d$ will be denoted $\lambda_d$ and the normalized surface measure on $\TT^d$ will be denoted $\sigma_d$. Let $dA$ denote the normalized area measure on $\DD$. For $\beta >-1$, we put $dA_\beta(z) = (\beta+1)(1-\m{z}^2)^\beta dA(z)$. On the polydisc $\DD^d$, we define 
$$dV_\beta(z) = dA_\beta(z_1)\cdots dA_\beta(z_d), \hspace{5mm} z=(z_1,\dots,z_d) \in \DD^d.$$
For $\beta>-1,$ the Bergman space $\Abd$ is the function space defined by 
$$ \Abd = \left\lbrace f \in \mathcal{O}(\DD^d,\CC):\ \norm{f}_2 = \int_{\DD^d}\m{f(z)}^2dV_\beta(z)<+\infty \right\rbrace,$$ and the Hardy space $H^2(\DD^d)$ is defined by 
$$ H^2(\DD^d) = \left\lbrace f \in \mathcal{O}(\DD^d,\CC):\  \norm{f} = \sup\limits_{0<r<1}\int_{\TT^d}\m{f(re^{i\theta})}^2d\sigma_d(e^{i\theta}) <+\infty \right\rbrace.$$ 
For simplicity, we will sometimes write $A^2_{-1}(\DD^d)$ for $H^2(\DD^d)$
and $V_{-1}$ for $\sigma_d$.

For $\eta \in \TT^d$ and $\ovd \in (0,2)^d$, the Carleson box is the set $${\sed = \lbrace z \in \DD^d :\ \m{z_j-\eta_j}<\delta_j,\ j=1,\dots,d\rbrace}.$$ If $I\subset \{1,\dots,d\}$, $S_I(\eta,\overline\delta)$
will mean $\lbrace z \in \DD^d :\ \m{z_j-\eta_j}<\delta_j,\ j\in I\rbrace$.

\subsection{Geometric lemmas}
In this subsection, we collect several lemmas which will be used throughout the proofs. We first begin with an easy observation. Let $\varphi\in \mathcal O(\DD^d,\CC)\cap \mathcal C(\overline{\DD}^d).$ Then for all $P\subset\{1,\dots,d\}$ and for all $z\in \TT^{|P|},$ the map $\varphi(z,\cdot)$ is holomorphic on $\DD^{d-|P|}.$ Indeed, it is the uniform limit on $\overline{\DD}^{d-|P|}$ of $(\varphi((1-1/n)z,\cdot))_n$ and we apply the Weierstrass theorem. 

This simple observation is already important to prove the following version of the Julia-Carathéodory theorem. Let us first introduce the following notation: let $\varphi\in \mathcal O(\DD^d,\DD)\cap \mathcal C^1(\ovD)$ and let $\xi\in\ovD$ be such that $\varphi(\xi)\in\TT.$ We set 
\begin{equation}\label{eq:pphixi}
P_{\varphi,\xi}=\left\{k\in\{1,\dots,d\}:\ \frac{\partial \varphi}{\partial z_k}(\xi)\neq 0\right\}.
\end{equation}

\begin{lemma}\label{JC}
Let $\varphi\in\mathcal O( {\DD^d},{\DD})\cap\mathcal C^1(\overline{\DD}^d)$ and assume that $\varphi(e) = 1$. Let $P=P_{\varphi,e}$. 
\begin{enumerate}[(a)]
\item for all $k\in P,$ $\frac{\partial\varphi}{\partial z_k}(e)>0$;
\item for all $w\in \oDD^{|P^c|}$, $\varphi(e_P,w)=1$;
\item for all $w\in \oDD^{|P^c|}$, for all $k\in P,$ $\frac{\partial \varphi}{\partial z_k}(e_P,w)=\frac{\partial \varphi}{\partial z_k}(e)$.
\item if moreover $\varphi\in\mathcal C^2(\ovD)$, there exist $F:\ovD\to \CC$ and $C>0$ such that for all $z\in\ovD,$
$$\varphi(z)=1+\sum_{k\in P} \frac{\partial\varphi}{\partial z_k}(e)(z_k-1)+F(z),\ |F(z)|\leq C\sum_{k\in P}|z_k-1|^2.$$
\end{enumerate}
\end{lemma}
\begin{proof}
(a) follows from the Julia-Carathéodory theorem (in one variable, see \cite[Theorem 1.33]{LQ26}) applied to the holomorphic maps $z_k\mapsto \varphi(z_k,e_{\{k\}^c})$ which map $\DD$ into $\DD$ (as it is not constant for $k\in P$).
For the proof of (b), we use a variant of the proof of \cite[Lemma 2.7]{BAYPOLY}. Let $\psi(w)=\varphi(e_P,w)$ that belongs to $\mathcal O(\DD^{|P^c|},\oDD)\cap \mathcal C^1(\oDD^{|P^c|}).$ We aim to show that $\psi$ is constant. By the maximum principle, it suffices to show that $\psi(0)=1.$ Define now $g(z)=\psi(z e_{|P^c|})$ which belongs to $\mathcal O(\DD,\oDD)\cap \mathcal C^1(\oDD)$. Then $g(1)=1$ and
$$g'(1)=\sum_{k\in P^c} \frac{\partial\varphi}{\partial z_k}(e)=0.$$
By Julia-Carathéodory theorem again, $g$ cannot map $\DD$ into $\DD$ otherwise $g'(1)>0$, hence by the maximum principle, this implies that $g$ is constant on $\DD$, therefore that $\psi(0)=g(0)=1.$

Regarding (c), we set, for $k\in P,$ $\psi_k=\frac{\partial\varphi}{\partial z_k}(e_P,\cdot)$ and $h_k=\Im m(\psi_k)$. The function $\psi_k$ belongs to the polydisc algebra $A(\DD^{|P^c|})$ . By (b) and the Julia-Carathéodory theorem applied as in (a) to the function $\tilde\varphi(z)=\varphi(z_P,\xi z_{P^c})$,
depending if $\varphi$ is constant or not, $\psi_k(\xi)\geq 0$ for all $\xi\in\TT^{|P^c|}$. Hence the function $h_k$ is a harmonic function on $\DD^{|P^c|}$, continuous on $\oDD^{|P^c|}$, which is equal to $0$ on the polytorus $\TT^{|P^c|}$. By the Poisson formula, it is equal to $0$. The holomorphic function $\psi_k$ taking only real values, it has to be constant.

Finally, for $z\in\ovD,$ let $z'=(e_P,z_{P^c})$ so that by (b), (c) and the Taylor formula
\begin{align*}
    \varphi(z)-\varphi(e)&=\varphi(z)-\varphi(z')\\
    &=\sum_{k\in P}\frac{\partial\varphi}{\partial z_k}(e)(z_k-1)+\int_0^1 (1-t)D^2\varphi(z'+t(z-z'))(z-z',z-z')dt.
\end{align*}
We get (d) since $z_k-z'_k=0$ if $k\notin P.$
\end{proof}

\begin{remark}\label{rem:nonzerogradient}
In particular, (b) of the previous lemma implies that $P_{\varphi,e}\neq\varnothing$, namely $\nabla \varphi(e)\neq 0.$
\end{remark}

We shall also need the following geometric lemma, called the parametrized Morse lemma (see \cite[Section 4.44]{BG92} for instance).
\begin{lemma}\label{Morse}
Let $f : \RR^n \times \RR^{d-n} \rightarrow  \RR$ be a $\mathcal{C}^3-$function such that $f(0) = 0$, $d_yf(0) = 0 $ and $d^2_yf(0)$ is non-singular. There exist a neighbourhood $U$ of $0$ in $\RR^d $, a diffeomorphism $\Gamma  : U \rightarrow \RR^d$, $(x,y) \mapsto (x,\gamma(x,y))$ and a $\mathcal{C}^1-$map $h : \mathbb{R}^n \rightarrow \RR$ such that for any $(x,y) \in U,$
$$f(x,y) = \sum\limits_{j=1}^{d-n}\pm\gamma_j(x,y)^2+h(x).$$
\end{lemma}

We will also need to estimate the volume of certain sets to characterize  continuity. This is done in the following lemmas:

\begin{lemma}\label{equiv}
Let $\beta >-1$. For all measurable subsets $A$ of $[0,2\pi)^d$ and for all $\ovd \in (0,1)^d,$
$$ 2^{-d}\prod_{k=1}^d \delta_k^{\beta+1}\lambda_d(A) \leq \pi^d V_\beta(\lbrace ((1-\rk)e^{i\tk})_{k=1,\dots,d}: 0\leq \rho_k\leq \delta_k,\ \theta \in A \rbrace)\leq 2^{d\beta}\prod_{k=1}^d \delta_k^{\beta+1}\lambda_d(A).$$
\end{lemma}
\begin{proof}
First, let us consider 
\begin{align*}
    \displaystyle\int_{1-\delta}^1(\beta+1)(1-r^2)^\beta rdr &= \left[ \dfrac{-(1-r^2)^{\beta+1}}{2}\right]_{1-\delta}^1 = 2^\beta\delta^{\beta+1}\left(1-\frac{\delta}{2}\right)^{\beta+1}.
\end{align*}
We conclude using Fubini's theorem. 
\end{proof}

\begin{remark} In particular the constants involved are independent of $\beta$ provided $\beta\in(-1,0]$ (see also \cite[Lemma 3.6]{Baytridisc}).
\end{remark}

\begin{lemma}\label{techniq}
Let $O$ be a bounded open subset of $\mathbb{R}^2$. Then there exists $C>0$ such that for all $a \in \mathbb{R}$ and all $\delta\in(0,e^{-1})$, we have 
$$\lambda_2(\lbrace (x,y)\in O:\ \m{x^2-y^2 - a} < \delta\rbrace)\leq C\delta\log(1/\delta).$$
\end{lemma}

\begin{proof}
We can assume that $a>0$ and $O = [-M/2,M/2]^2$ for some $M>1$.  Let $(x,y) \in O$ and set $u=x-y$ and $v=x+y$. Then, $(u,v)\in[-M,M]^2$ and
$$\m{x^2-y^2-a} <\delta \Leftrightarrow \m{uv-a}<\delta.$$
In particular, if the set $\{(x,y)\in O:\ |x^2-y^2-a|<\delta\}$ 
is nonempty, then $a$ has to be smaller than $M^2+\delta,$
an assumption that we now make. By symmetry, one only has to estimate the measure of 
$$\mathcal E=\{(u,v)\in[0,M]\times[-M,M]:\ |uv-a|<\delta\}.$$
Assume first that $a\geq\delta$ which implies $\mathcal E\subset[0,+\infty)^2.$ Using $0\leq v\leq M$ and $(a-\delta)/u\leq v\leq (a+\delta)/u,$ we find
that $\mathcal E\subset\mathcal E_1\cup\mathcal E_2$ with 
\begin{align*}
\mathcal E_1&=\left\{(u,v)\in\RR^2:\ \frac{a-\delta}M\leq u\leq \frac{a+\delta}M,\ 0\leq v\leq M\right\}\\
\mathcal E_2&=\left\{(u,v)\in\RR^2:\ \frac{a+\delta}M\leq u\leq M+2,\ \frac{a-\delta}u\leq v\leq\frac{a+\delta}u\right\}
\end{align*}
(we have taken here $M+2$ instead of $M$ to be sure that $M+2>(a+\delta)/M$).
Therefore, 
\begin{align*}
    \lambda_2(\mathcal E)&\leq 2\delta+\int_{(a+\delta)/M}^{M+2}\frac{2\delta}udu\\
    &\leq 2\delta+2\delta\log(M+2)-2\delta\log(a+\delta)+2\delta\log(M)\\
    &\leq C\delta\log(1/\delta).
\end{align*}
On the other hand, if $a<\delta,$ then we write $\mathcal E\subset \mathcal E'_1\cup\mathcal E'_2\cup\mathcal E'_3$ with
\begin{align*}
\mathcal E'_1&=\left\{(u,v)\in\RR^2:\ 0\leq u\leq \frac{\delta-a}M,\ -M\leq v\leq M\right\}\\
\mathcal E'_2&=\left\{(u,v)\in\RR^2:\ \frac{\delta-a}M\leq u\leq \frac{a+\delta}M,\ \frac{a-\delta}u\leq v\leq M\right\}\\
\mathcal E'_3&=\left\{(u,v)\in\RR^2:\ \frac{a+\delta}M\leq u\leq M+2,\ \frac{a-\delta}u\leq v\leq \frac{a+\delta}u\right\}
\end{align*}
so that 
\begin{align*}
\lambda_2(\mathcal E)&\leq 2(\delta-a)+\int_{\frac{\delta-a}M}^{\frac{\delta+a}M}\left(M+\frac{\delta-a}u\right)du+2\delta\log(M+2)-2\delta\log(a+\delta)+2\delta\log(M)\\
&\leq 2\delta+(\delta-a)\log\left(\frac{\delta+a}{\delta-a}\right)+4\delta\log(M+2)-2\delta\log(a+\delta)\\
&\leq C\delta\log(1/\delta)
\end{align*}
since $x\mapsto x\log(x)$ is decreasing on $(0,e^{-1})$ and $(\delta-a)\log(\delta+a)\leq 0.$
\end{proof}

\subsection{How to prove continuity - Carleson measures}
It is well known that the continuity of composition operators on Hardy and Bergman spaces of the polydisc is linked to volume estimates of the preimage of Carleson sets. As a consequence of \cite[Theorem 2.5]{Jaf91}, for $\beta_1,\beta_2>-1$, we know that $C_\phi$ maps boundedly $A^2_{\beta_1}(\DD^d)$ into $A^2_{\beta_2}(\DD^d)$ if and only if
\begin{equation}\label{eq:carleson}
\exists C>0,\ \forall \eta\in\TT^d,\ \forall \ovd\in(0,1)^d,\ V_{\beta_2}(\phi^{-1}(S(\eta,\ovd)))\leq C V_{\beta_1}(S(\eta,\ovd)).
\end{equation}
This means that $V_{\beta_2}\circ\phi^{-1}$ is a Carleson measure for $A_{\beta_1}^2(\DD^d).$ Recall that
$$V_{\beta_1}(S(\eta,\ovd))\simeq \prod_{i=1}^d \delta_i^{2+\beta_1}.$$
This characterization extends to $\beta_2=-1$ (that is, replacing $A_{\beta_2}(\DD^d)$ by $H^2(\DD^d)$), with the measure $\sigma_d\circ\phi^{-1}$ instead of
$V_{\beta_2}\circ\phi^{-1}.$ It becomes more complicated for $\beta_1=-1$
since Carleson measures on Hardy spaces of the polydisc are difficult objects. The condition 
$$V_{\beta_2}(\phi^{-1}(S(\eta,\ovd)))\leq C\prod_{i=1}^d \delta_i$$
is still necessary but not sufficient. To prove continuity, a strategy first devised in \cite{BAYPOLY} is to prove that $C_\phi$ maps continuously $A^2_{\beta_1}(\DD^d)$ into $A^2_{\beta_2}(\DD^d)$ with $\|C_\phi\|_{\mathcal L(A_{\beta_1}^2,A_{\beta_2}^2)}\leq C$ with $C$ independent of $\beta_1$ and to let $\beta_1$ to $-1$. For this we will need the following result (see \cite[Proposition 9.3]{BAYPOLY}): 

\begin{lemma}\label{lem:independentconstant}
There exists a constant $C(d)$ depending only on the dimension $d$ such that, for all $\phi\in\mathcal O(\DD^d,\DD^d)$, all $\beta_1\in(-1,0]$ and  all $\beta_2>-1,$ if \eqref{eq:carleson} is satisfied, then $\|C_\phi\|_{\mathcal L(A_{\beta_1}^2,A^2_{\beta_2})}\leq C(d)\cdot C.$
\end{lemma}

\subsection{First order variables at contact points}

In our study, variables appearing at first order in the expansion of $\phi$ at the contact points will play a particular role. We thus need some definitions to localize these variables and to quantify their importance. The first one already appears in \eqref{eq:pphixi}.

\begin{definition}\label{def:pomega}
\begin{enumerate}[(a)]
\item Let $\varphi \in \mathcal{O}(\DD^d,\DD) \cap \mathcal{C}^1(\ovD)$ and $\xi\in\oDD^d$ such that $\varphi(\xi)\in\TT.$ We define
$$P_{\varphi,\xi} = \left\lbrace k \in \I{d} : \dfrac{\partial\varphi(\xi)}{\partial z_k} \neq 0 \right\rbrace.$$
\item Let $\phi\in\mathcal O(\DD^d,\DD^d)\cap\mathcal C^1(\ovD)$, 
$I\subset\{1,\dots,d\}$ and $\xi\in\oDD^d$ such that $\phi_I(\xi)\in \TTI$. For $j\in I,$ we simply write $P_{j,\xi}=P_{\phi_j,\xi}$ and we set $P_{I,\xi}=\bigcup\limits_{j\in I  } P_{j,\xi}.  $
\item  Let $\phi\in\mathcal O(\DD^d,\DD^d)\cap\mathcal C^1(\ovD)$, 
$I\subset\{1,\dots,d\}$ and $\xi\in\oDD^d$ such that $\phi_I(\xi)\in \TTI$, let  $\ovd \in (0,1)^{|I|}$ and let $k\in\{1,\dots,d\}$. We set 
$$ \omega_{I,k,\xi}(\ovd) = \left\lbrace\begin{array}{ll} 
1 & \text{provided } k \notin P_{I,\xi}=\bigcup\limits_{j\in I  } P_{j,\xi} \\
\min\lbrace \delta_j,\ j\in I, \ k\in P_{j,\xi} \rbrace      & \text{otherwise.}  \\
\end{array}\right.$$
\end{enumerate}
 \end{definition}



In most cases, it will be sufficient to consider $\xi\in\TT^d$ in the previous definitions. For $\xi\in\oDD^d,$ set $\hat\xi=(\xi_1/|\xi_1|,\dots,\xi_d/|\xi_d|)$ where $\xi_k/|\xi_k|=1$ if $\xi_k=0.$ If $\varphi\in\mathcal O(\DD^d,\DD)\cap\mathcal C^1(\ovD)$, 
and $\xi\in\oDD^d$ are such that $\varphi(\xi)\in \TT,$
then $\varphi(\hat\xi)=\varphi(\xi):$ if $|\xi_1|\neq 1$, the function $z\mapsto \varphi(z,\xi_{\{1\}^c})$ is holomorphic hence constant by the maximum principle. Moreover, $|\xi_k|=1$ if $k\in P_{\varphi,\xi}$ and
$P_{\varphi,\xi}=P_{\varphi,\hat\xi}$ by Lemma \ref{JC}.

\medskip

The values of the $\wikxi$, $k=1,\dots,d,$ are rather intricate.
To help digest their definition, we prove the following lemma, which already points out their importance in the study of continuity of composition operators.

\begin{lemma}\label{lem:wikcontinu}
Let $\phi\in\mathcal O(\DD^d,\DD^d)  \cap  \mathcal{C}^1(\ovD)$ and let $\beta_1,\beta_2\geq  -1$. Assume that $C_\phi$ maps boundedly $A_{\beta_1}^2(\DD^d)$
into $A_{\beta_2}^2(\DD^d)$. Then there exists $D>0$ such that, for all $I\subset\{1,\dots,d\}$ and all $\xi\in\ovD$ such that $\phi_I(\xi)\in\TTI$, 
for all $\ovd\in\oii$,
$$\prod_{k=1}^d \wikxi^{2+\beta_2}\leq D\prod_{j\in I}\delta_j^{2+\beta_1}.$$
\end{lemma}
\begin{proof}
Taking into account the remarks after Definition \ref{def:pomega}, we may assume $\xi\in\TT^d$ and even $\xi=e$ and $\phi_I(e)=e$. We first show that for all $\ovd\in(0,1)^{|I|},$ 
$$E:=\{((1-\rho_k)e^{i\theta_k}):\ 0\leq \rho_k\leq \wikxi,\ |\theta_k|\leq\wikxi,\ k\in P_{I,\xi}\}\subset \phi_I^{-1}(S_I(\phi_I(e),C_1\ovd))$$
with $C_1=\max\limits_{j\in\{1,\dots,d\}}(2 \|\nabla \phi_j\|_1)$ where 
$$\|\nabla \phi_j\|_1=\sup_{w\in\ovD}\sum_{l=1}^d \left|\frac{\partial \phi_j}{\partial z_l} (w)\right|.$$
Indeed, for $j\in I,$ 
\begin{align*}
    |\phi_j(z)-\phi_j(e)|&\leq 
    |\phi_j(z_{P_{j,e}},z_{P_{j,e}^c})-\phi_j(e_{P_{j,e}},z_{P_{j,e}^c})|+|\phi_j(e_{P_{j,e}},z_{P_{j,e}^c})-\phi_j(e)|\\
    &\leq \|\nabla \phi_j\|_1\max_{k\in P_{j,e}}|z_k-1|\\
    &\leq 2 \|\nabla\phi_j\|_1\delta_j
\end{align*}
by the mean value inequality and Lemma \ref{JC} (b). 
We extend $\overline\delta$ to $(0,2]^d$ by setting $\delta_j=2$ if $j\notin I$, so that, assuming without loss of generality $C_1\geq 1$,
$$E\subset\phi^{-1}(S(\phi(e),C_1\ovd)).$$
Since $C_\phi$ maps boundedly $A_{\beta_1}^2(\DD^d)$ into $A_{\beta_2}^2(\DD^d),$
we know that there exists $C_2>0$ (independent of $I,$ $\xi$ and $\ovd$) such that 
$$V_{\beta_2}(\phi^{-1}(S(\phi(e),C_1\ovd)))\leq C_2 \prod_{j\in I}\delta_j^{2+\beta_1}.$$
We get the result since Lemma \ref{equiv} for $\beta_2>-1$ (the case $\beta_2=-1$ being trivial) implies 
$$V_{\beta_2}(E)\gtrsim \left(\prod_{k\in P_{I,\xi}}\wikxi\right)^{2+\beta_2}=\left(\prod_{k=1}^d\wikxi\right)^{2+\beta_2}$$
(recall that $\wikxi=1$ when $k\notin P_{I,\xi}$).
\end{proof}

 We will also use the following purely combinatorial result to estimate their product.

\begin{lemma}\label{lem:produitwk3}
Let $\phi\in\mathcal O(\DD^d,\DD^d)\cap \mathcal C^1(\ovD),$ $I\subset\{1,\dots,d\}$ and $\xi\in\oDD^d$ such that $\phi_I(\xi)\in\TT^{|I|}$.
For $J\subset I,$ we set $P_{J,\xi}=\bigcup\limits_{j\in J}P_{j,\xi}$ and
$\gamma_{I,\xi}=\min\limits_{\varnothing\neq J\subset I}\frac{|P_{J,\xi}|}{|J|}.$
Then for all $\ovd\in\oii,$
$$\prod_{k=1}^d \omega_{I,k,\xi}(\ovd)\leq \prod_{j\in I}\delta_j^{\gamma_{I,\xi}}.$$
\end{lemma}

The proof of Lemma \ref{lem:produitwk3} will be based on the following easy lemma.
\begin{lemma}
    Let $q\geq 1$ and 
    $$K=\{(x_1,\dots,x_q)\in [0,+\infty)^q:\ x_1\geq x_2\geq\cdots\geq x_q\geq 0,\ x_1+\cdots+x_q=1\}.$$ Let $n_1,\dots,n_q\in\mathbb Z_+$ and let $F$ be defined on $K$ by $F(x)=\dss_{j=1}^q n_jx_j.$ Then 
    $$\min_K F\geq \min_{1\leq k\leq q}\frac{n_1+\cdots+n_k}{k}.$$
\end{lemma}
\begin{proof}
A linear map attains its minimum on a polytope only at an extremal point of the polytope. The extremal points of $K$ are the points $(1/k,\dots,1/k,0,\dots,0)$.
\end{proof}

\begin{proof}
    Since the proof is purely combinatorial and does not depend on $\xi$, we shall omit the subscript $\xi.$ Let $\ovd\in \oii$. There exists a 
    one-to-one map $\sigma:\{1,\dots,|I|\}\to I$ such that $\delta_{\sigma(1)}\leq \cdots\leq \delta_{\sigma(|I|)}.$ 

For $k\in P_I$, let us set
$$j(k)=\min\big\{j\in\{1,\dots,|I|\}:\ k\in P_{\sigma(j)}\big\},$$
which is well defined since $k\in P_I=\bigcup\limits_{j=1}^{|I|}P_{\sigma(j)}$. By definition of $\omega_{I,k}$
and since $\delta_{\sigma(1)}\leq\cdots\leq\delta_{\sigma(|I|)}$, we have
$$\omega_{I,k}(\ovd)=\min\big\{\delta_{\sigma(j)}:\ 1\leq j\leq |I|,\ k\in P_{\sigma(j)}\big\}=\delta_{\sigma(j(k))}.$$
Moreover, for $j\in\{1,\dots,|I|\}$, $j(k)=j$ if and only if
$k\in P_{\sigma(j)}\backslash\big(P_{\sigma(1)}\cup\cdots\cup P_{\sigma(j-1)}\big)$. In other words, the sets
$$A_j=P_{\sigma(j)}\backslash\big(P_{\sigma(1)}\cup\cdots\cup P_{\sigma(j-1)}\big),\quad 1\leq j\leq |I|,$$
are pairwise disjoint, their union is $P_I$, and $\omega_{I,k}(\ovd)=\delta_{\sigma(j)}$ for all $k\in A_j$.
Grouping the factors according to the value of $j(k)$, we get
$$\prod_{k\in P_I}\omega_{I,k}(\ovd)=\prod_{j=1}^{|I|}\prod_{k\in A_j}\delta_{\sigma(j)}
=\prod_{j=1}^{|I|}\delta_{\sigma(j)}^{|A_j|}
=\prod_{j=1}^{|I|}\delta_{\sigma(j)}^{\left|P_{\sigma(j)}\backslash(P_{\sigma(1)}\cup\cdots\cup P_{\sigma(j-1)})\right|}.$$

We set, for $j\in\{1,\dots,|I|\}$, $x_j=-\log(\delta_{\sigma(j)})$ and $n_j=|P_{\sigma(j)}\backslash (P_{\sigma(1)}\cup\cdots\cup P_{\sigma(j-1)})|.$
Taking the logarithm and using the previous lemma with homogeneity, we find
\begin{align*}
\log \prod_{k\in P_I}\omega_{I,k}(\ovd)&=-\sum_{j=1}^{|I|}n_j x_j\\
&\leq \left(-\sum_{j=1}^{|I|}x_j\right)\min_{1\leq k\leq|I|}\frac{n_1+\cdots+n_k}k.
\end{align*}
Since the sets $P_{\sigma(j)}\backslash (P_{\sigma(1)}\cup\cdots\cup P_{\sigma(j-1)})$ are pairwise disjoint, $n_1+\cdots+n_k=|P_{\{\sigma(1),\dots,\sigma(k)\}}|$ so that
$$\prod_{k\in P_I}\omega_{I,k}(\ovd)\leq \prod_{j\in I}\delta_j^{\min\limits_{1\leq k\leq |I|}\frac{|P_{\{\sigma(1),\dots,\sigma(k)\}}|}k}\leq \prod_{j\in I}\delta_j^{\min\limits_{\varnothing\neq J\subset I}\frac{|P_J|}{|J|}}. $$
\end{proof}

We shall use two particular cases of this lemma. First of all, since $|P_J|/|J|\geq 1/d$, we always have for all $\ovd\in(0,1)^d$
\begin{equation} \label{eq:wikcontinuite}
    \dprod_{k=1}^d\wikxi^d = \dprod_{k \in P_{I,\xi} } \wikxi^d \leq \dprod_{j \in I} \delta_j.
\end{equation}
Next, we also observe that if 
$\alpha + \gamma_{I,\xi}(1+\beta_2) \geq 2+\beta_1$, then for all $\ovd\in(0,1)^d$
\begin{equation}\label{eq:wikalpha}
\left(\displaystyle\prod_{j\in I}\delta_j\right)^{\alpha}\leq \frac{
\displaystyle\prod_{j\in I}\delta_j^{2+\beta_1}}
{\displaystyle\prod_{k=1}^d \wikxi^{1+\beta_2}}.
\end{equation}

\subsection{Modified pre-images of Carleson windows}

Our new characterization of continuity will depend on variants of preimages of Carleson windows, which only involve the values of $\phi$ on the polycircle with respect to the variables appearing at first order in $\phi$ at the contact points, the remaining (“hidden”) variables being allowed to range over the whole polydisc.

Let us introduce these relevant sets for our characterization. We consider $\phi\in\mathcal O(\DD^d,\DD^d)\cap \mathcal C^1(\ovD)$, $I\subset\{1,\dots,d\}$ and $\xi\in\oDD^d$ such that $\phi_I(\xi)\in\TT^{|I|}$. For $j\in I$ and $z\in\ovD$ such that $z_k\neq 0$ if $k\in P_{j,\xi}$, we set 
$\pi_{j,\xi}(z)=w$ where $w_k=z_k/|z_k|$ if $k\in P_{j,\xi}$, $w_k=z_k$ otherwise (namely we normalize each coordinate which appears at first order in the development of $\phi$ at $\xi$). We then set, for $\eta\in\TTI$ and $\ovd\in(0,1)^{|I|}$, 

\begin{align}E_I(\xi,\eta,\ovd)=\Big\{z=\big((1-\rho_k)e^{i\theta_k}):&\ z_k\neq 0\textrm{ if }k\in P_{I,\xi},\ 0\leq \rho_k\leq \omega_{I,k,\xi}(\ovd)\textrm{ if }k\in P_{I,\xi},\label{eq:EIxieta}\\
&\ |\phi_j(\pi_{j,\xi}(z))-\eta_j|<\delta_j,\ j\in I\Big\}. \nonumber
\end{align}

The next lemma is the main step towards the proof of our characterization. It says that our Carleson windows are comparable to the sets we have just introduced.

\begin{lemma}\label{lem:doubleinclusion}
Let $\phi\in\mathcal O(\DD^d,\DD^d)\cap \mathcal C^1(\ovD),$ $I\subset\{1,\dots,d\}$ and $\xi\in\oDD^d$ such that $\phi_I(\xi)\in\TT^{|I|}$. There exists $D>0$ and $\mathcal U$ a neighborhood of $\xi$ in $\ovD$ such that, for all $\eta\in\TTI,$ for all $\ovd\in\oii$, 
\begin{enumerate}[(a)]
\item $E_I(\xi,\eta,\ovd)\subset \phi_I^{-1}(S_I(\eta,D\ovd))$;
\item $\phi_I^{-1}(S_I(\eta,\ovd))\cap\mathcal U\subset E_I(\xi,\eta,D\ovd)$.
\end{enumerate}
\end{lemma}
\begin{proof}
\noindent (a)
Let $z=((1-\rho_k)e^{i\theta_k})$ belonging to $E_I(\xi,\eta,\ovd)$.
The mean value inequality yields for any $j\in I,$
$$|\phi_j(z)-\phi_j(\pi_{j,\xi}(z))|\leq \|\nabla \phi_j\|_1 \sup_{k\in P_{j,\xi}}|z_k-e^{i\theta_k}|$$
(we can take the supremum over $P_{j,\xi}$ since the other coordinates are unchanged). 
Therefore,
$$|\phi_j(z)-\phi_j(\pi_{j,\xi}(z))|\leq \|\nabla \phi\|_1 \sup_{k\in P_{j,\xi}}\omega_{I,k,\xi}(\ovd)\leq \|\nabla \phi\|_1 \delta_j.$$
The conclusion of a) now follows from the triangle inequality, with $D=1+\|\nabla \phi\|_1$.

\smallskip

\noindent (b) Without loss of generality, we assume that $\hat\xi=e$ and $\phi_I(\xi)=\phi_I(\hat \xi)=e_I.$ We know that $|\xi_k|=1$ if $k\in P_{I,\xi}$. Therefore, one can find an open neighborhood $\mathcal U_1$ of $\xi$ such that 
for all $z\in\mathcal U_1$ and all $k\in P_{I,\xi}$, $z_k\neq 0$. Let $j\in I$ and $z\in\mathcal U_1$ be such that $|\phi_j(z)-\eta_j|<\delta_j$. We write $z=((1-\rho_k)e^{i\theta_k})$ and set $P=P_{j,\xi}$ and, for $t\in[0,1],$
$$h(t)=|\phi_j((1-t\rho_P)e^{i\theta_P},z_{P^c})|^2$$
(namely, we are looking at $|\phi_j|^2$ along the segment joining $z$ to $\pi_{j,\xi}(z)$). By the fundamental theorem of calculus, 
\begin{align*}
|\phi_j(z)|^2-|\phi_j(\pi_{j,\xi}(z))|^2&=\int_0^1 h'(t)dt\\
&=-\sum_{k\in P}2\rho_k \underbrace{\int_0^1  \Re e\left(e^{i\theta_k} \left(\overline{\phi_j} \frac{\partial \phi_j}{\partial z_k}\right)((1-t\rho_P)e^{i\theta_P},z_{P^c})\right)dt}_{\displaystyle :=F_{j,k,\xi}(z)}
\end{align*}
Now, for all $w\in\overline{\DD}^{|P^c|},$ by Lemma \ref{JC},
\begin{align*}
F_{j,k,\xi}(e_P,w)&=\int_0^1  \frac{\partial \phi_j}{\partial z_k}(e_P,w)dt\\
&=\frac{\partial \phi_j}{\partial z_k}(e)>0
\end{align*}
Moreover, $F_{j,k,\xi}$ is uniformly continuous on $\{z\in\ovD:\ |z_k-\xi_k|\le\frac12,\ k\in P_{I,\xi}\}$. Therefore one can find an open subset $\mathcal U_2$ of $\ovD$ of the form
$$\mathcal U_2=\left\{w\in\ovD:\ |w_k-\xi_k|<r,\ k\in P_{I,\xi} \right\}$$
and $c>0$ such that, for all $z\in\mathcal U=\mathcal U_1\cap\mathcal U_2$, for all $j\in I$ and all $k\in P_{j,\xi}$ 
$F_{j,k,\xi}(z)\geq c.$ For these $z$, since $|\phi_j(z)|^2\geq 1-2\delta_j$ and $|\phi_j(\pi_{j,\xi}(z))|^2\leq 1,$ one gets 
$$-2c\sum_{k\in P}\rho_k\geq -2\delta_j$$
which yields $\rho_k\leq \delta_j/c$ for all $k\in P_{j,\xi}$. Taking the infimum over $j\in I,$
we get, for all $k\in P_I,$ $\rho_k\leq \omega_{I,k,\xi}(\ovd)/c.$
Moreover, we also have 
\begin{align*}
|\phi_j(\pi_{j,\xi}(z))-\eta_j|&\leq |\phi_j(\pi_{j,\xi}(z))-\phi_j(z)|+|\phi_j(z)-\eta_j|\\
&\leq \|\nabla \phi\|_1 \max_{k\in P_{j,\xi}}|z_k-e^{i\theta_k}|+\delta_j\\
&\leq \left(1+\frac{\|\nabla \phi\|_1}{c}\right)\delta_j.
\end{align*}
This yields the result.
\end{proof}

The $V_\beta$-volume of the sets $E_I(\xi,\eta,\ovd)$ does not really seem easier to compute than the volume of the preimages of Carleson windows. However, a simple Fubini's argument will be very useful for our purposes.

Again, we consider $\phi\in\mathcal O(\DD^d,\DD^d)\cap \mathcal C^1(\ovD)$, $I\subset\{1,\dots,d\}$ and $\xi\in\oDD^d$ such that $\phi_I(\xi)\in\TT^{|I|}$. For simplicity, we shall write $P$ for $P_{I,\xi}$ and $\omega_k$ for $\omega_{I,k,\xi}$. We then set, for $\eta\in\TTI,$ $\ovd\in(0,1)^{|I|}$ and $\rho\in(0,1)^{|P|}$
$$F_I(\xi,\eta,\ovd,\rho)=\left\{(\theta,w)\in[0,2\pi)^{|P|}\times \oDD^{|P^c|}:\ |\phi_j(\pi_{j,\xi}((1-\rho)e^{i\theta},w))-\eta_j|<\delta_j,\ j\in I\right\}.$$

By Fubini's theorem and polar integration, for all $\beta>-1$,
\begin{equation}\label{eq:volume}
V_{\beta}(E_I(\xi,\eta,\ovd))=\int_{\prod\limits_{k\in P}[0,\omega_k(\ovd)]}
(\lambda_{|P|}\otimes V_{\beta})(F_I(\xi,\eta,\ovd,\rho))\prod_{k\in P}
W_\beta(\rho_k)\,d\rho_k,
\end{equation}
where
$$W_\beta(\rho)=\frac{\beta+1}{\pi}\big(\rho(2-\rho)\big)^{\beta}(1-\rho).$$

This allows us to prove the following lemma.

\begin{lemma}\label{lem:changeofbeta}
Let $\beta'>\beta>-1$. Let $\phi\in\mathcal O(\DD^d,\DD^d)\cap \mathcal C^1(\ovD)$, $I\subset\{1,\dots,d\}$ and $\xi\in\oDD^d$ such that $\phi_I(\xi)\in\TT^{|I|}$. Let also $\eta\in\TTI,$ $\ovd\in(0,1)^{|I|}$. Then 
$$V_{\beta'}(E_I(\xi,\eta,\ovd))\le\Big(\frac{\beta'+1}{\beta+1}\Big)^{d}2^{|P_{I,\xi}|(\beta'-\beta)} \prod_{k\in P_{I,\xi}}\wikxi^{\beta'-\beta}\,
V_{\beta}(E_I(\xi,\eta,\ovd)).$$
\end{lemma}
\begin{proof}
Again we write $P$ for $P_{I,\xi}$ and $\omega_k$ for $\omega_{I,k,\xi}$. Since $\beta'>\beta$, for $0\le\rho_k\le\omega_k(\ovd)$,
$$W_{\beta'}(\rho_k)=\frac{\beta'+1}{\beta+1}\big(\rho_k(2-\rho_k)\big)^{\beta'-\beta}W_{\beta}(\rho_k)
\le \frac{\beta'+1}{\beta+1}\,2^{\beta'-\beta}\,\omega_k(\ovd)^{\beta'-\beta}\,W_{\beta}(\rho_k),$$
and, since $(1-|w|^2)^{\beta'-\beta}\le 1$ on $\DD^{|P^c|}$,
$$(\lambda_{|P|}\otimes V_{\beta'})(F_I(\xi,\eta,\ovd,\rho))\le \left(\frac{\beta'+1}{\beta+1}\right)^{|P^c|}(\lambda_{|P|}\otimes V_{\beta})(F_I(\xi,\eta,\ovd,\rho)).$$
Therefore \eqref{eq:volume} and an integration over $\dprod_{k\in P}[0,\omega_k(\ovd)]$ gives the result.
\end{proof}

\color{black}

\section{A new characterization of continuity and applications}

\subsection{Characterization}

We are now ready to prove our first main result. 
\begin{theorem}\label{result_gen}
Let $\beta_1,\beta_2\geq -1$ with $(\beta_1,\beta_2)\neq (-1,-1)$. Let $\phi \in \mathcal{O}(\DD^d,\DD^d) \cap \mathcal{C}^1(\ovD)$. Then, the following assertions are equivalent:  
\begin{enumerate}[(a)]
\item ${C_\phi : A_{\beta_1}^2(\DD^d) \rightarrow A_{\beta_2}^2(\DD^d)}$ is bounded. \\
\item For all $I \subset \I{d}$ and all $\xi\in\ovD$ such that $\phi_I(\xi)\in\TTI$,  there exists $C>0$ such that for all $\eta \in \TTI$ and all $\ovd \in (0,1)^{\m{I}},$
$$V_{\beta_2}(E_I(\xi,\eta,\ovd))\leq C \prod\limits_{j\in I}\delta_j^{2+\beta_1}.$$
\item For all $I \subset \I{d}$ and all $\xi\in\ovD$ such that $\phi_I(\xi)\in\TTI$ and $I$ is maximal with respect to this property, there exist $\mathcal U$ a neighborhood of $\xi$ in $\ovD$ and $C>0$ such that for all $\eta \in \TTI$ and all $\ovd \in (0,1)^{\m{I}},$
\begin{equation}\label{eq:carlesonbord} 
V_{\beta_2}(E_I(\xi,\eta,\ovd)\cap\mathcal U)\leq C \prod\limits_{j\in I}\delta_j^{2+\beta_1}.
\end{equation}
\end{enumerate}
\end{theorem}

\begin{proof}

We first observe that $(b) \Rightarrow (c)$ is immediate. Let us show that $(a)\implies (b)$.  Let  $I \subset \I{d}$ and $\xi\in\oDD^d$ be such that $\phi_I(\xi)\in\TTI.$ Let also $\eta\in\TTI$ and $\ovd\in(0,1)^{|I|}.$ By Lemma \ref{lem:doubleinclusion}, there exists $D\geq 1$ such that 
$$E_I(\xi,\eta,\ovd) \subset \phi_I^{-1}(S_I(\eta,D\ovd)).$$
Let us now extend $\ovd$ to $(0,2]^d$ by setting $\delta_j=2$ if $j\notin I$ and $\eta$ to $\TT^d$ by taking $\eta_j=1$ if $j\notin I.$
Then 
$$\phi_I^{-1}(S_I(\eta,D \ovd))= \phi^{-1}(S(\eta,D\ovd)).$$
We now use the necessary condition for continuity of composition operators using Carleson measures (which is valid even if $\beta_1=-1$) to get
\begin{align*} V_{\beta_2}(E_I(\xi,\eta,\ovd) )
&\leq V_{\beta_2}(\phi_I^{-1}(S_I(\eta,D \ovd)))\\
&\lesssim V_{\beta_2}(\phi^{-1}(S(\eta,D\ovd )))\\
&\lesssim V_{\beta_1}(S(\eta,D\ovd))\\
&\lesssim \prod_{j\in I}\delta_j^{2+\beta_1}
\end{align*}
where the involved constants do not depend on $\eta$ and $\ovd.$ Observe that this implication remains true if $\beta_1=\beta_2=-1.$

\smallskip

We split the proof of $(c)\implies (a)$ in two cases, according to whether $A_{\beta_1}^2$ is, or is not, the Hardy space. \\ 

\noindent \textit{Case 1: $C_\phi : A^2_{\beta_1}(\DD^d) \rightarrow A^2_{\beta_2}(\DD^d) $ with $\beta_1>-1$. } \\

The proof will be divided into two steps. We first show that \eqref{eq:carlesonbord} implies a local Carleson type condition and then we conclude with a covering argument. \\

\smallskip

\noindent {\bf Fact.}
Let $\xi\in\ovD$ and $I\subset\{1,\dots,d\}$ be such that $\phi_I(\xi)\in\TTI$ and $I$ is maximal with respect to this property. Then there exist a neighbourhood $\mathcal V$ of $\xi$ in $\overline{\DD}^d$ and $D:=D(\phi,I,\xi,\beta_1,\beta_2)$ such that, for all $\eta\in \TTI$ and all $\ovd\in(0,1)^{|I|},$ 
$$V_{\beta_2}(\phi_I^{-1}(\Si)\cap \mathcal V) \leq D V_{\beta_1}(\Si)$$
where $D$ does not depend on $\beta_1$ (resp. $\beta_2$) 
if $\beta_1\in(-1,0]$ (resp. $\beta_2\in (-1,0]$).

\smallskip

\noindent {\it Proof of the fact.}
By b) of Lemma \ref{lem:doubleinclusion}, there exists $\mathcal{U}_1$ a neighbourhood of $\xi$ and $D>0$ such that   $$ \phi^{-1}_I(S_I(\eta,D^{-1}\ovd)) \cap \mathcal{U}_1 \subset \Ei. $$
Let $\mathcal{U}$ be the neighbourhood of $\xi \in \ovD$ given by $(c)$ of Theorem \ref{result_gen}. Denoting $\mathcal{V} = \mathcal{U} \cap \mathcal{U}_1$ then $\mathcal{V} $ is a neighbourhood of $\xi$ in $\ovD$ such that 
$$\phi^{-1}_I(S_I(\eta,\ovd)) \cap \mathcal{V} \subset E_I(\xi,\eta,D \ovd) \cap \mathcal{U}.$$
Thus, 
\begin{align*}
V_{\beta_2}(\phi_I^{-1}(S_I(\eta,\ovd))\cap \mathcal V)
&\leq \vb{2}(E_I(\xi,\eta,D\ovd)\cap\mathcal U)\\
&\leq C\cdot C_{\beta_2} \prod_{j\in I} \delta_j^{2+\beta_1}\\
&\leq C \cdot C_{\beta_1}\cdot C_{\beta_2} V_{\beta_1}(S_I(\eta,\ovd)).
\end{align*}
That $C_{\beta_i}$ 
may be chosen independently of $\beta_i$ 
when $\beta_i$ belongs to $(-1,0]$ follows from \cite[Lemma 8.1]{BAYPOLY}. \hfill $\qed$

\smallskip

We are now ready for the second half of $(c)\implies (a),$ the covering argument.
We introduce $$A_1 = \lbrace \xi \in \overline{\DD}^d:\  \phi(\xi) \in \TT^d\}. $$ 
Let $\xi\in A_1.$ By the fact, there exist $\mathcal U_{1,\xi}$ a neighbourhood of $\xi$ in $\overline{\DD}^d$ and $C_{1,\xi}>0$ such that  for all $\eta \in \TT^d,$ and all $\ovd\in (0,1)^d$,
$$V_{\beta_2}(\phi^{-1}(S(\eta,\ovd)) \cap \mathcal U_{1,\xi}) \leq C_{1,\xi}(\beta_1,\beta_2)V_{\beta_1}(S(\eta,\ovd)).$$
Then, the sets $(\mathcal U_{1,\xi})_\xi$ form an open covering of $A_1$ which is compact so we can extract a finite covering $(\mathcal U_{1,k})_{k=1,\dots,s_1}$. For all $\eta \in \TT^d$ and all  $\ovd \in (0,2)^d, $ we get 
$$\vb{2}\left(\phi^{-1}(S(\eta,\ovd)) \cap \bigcup\limits_{k=1}^{s_1}\mathcal U_{1,k}\right) \leq \sum\limits_{k=1}^{s_1}\vb{2}(\phi^{-1}(S(\eta,\ovd)) \cap \mathcal U_{1,k}) \leq \left(\sum\limits_{k=1}^{s_1}C_{1,k}\right)\vb{1}(S(\eta,\ovd)), $$ with $C_{1,k}$ independent of $\beta_i$ if $\beta_i \in (-1,0].$ 
For notational convenience, we set $D_{1,k}=C_{1,k}.$

Then there exists $\varepsilon_1 \in (0,1)$ such that for all $\xi \in B_1 := \ovD \setminus \bigcup\limits_{k=1}^{s_1}\mathcal U_{1,k}, $ we have $\m{\phi_j(\xi)}< 1-\varepsilon_1$ for some $j \in \lbrace 1,\dots,d\rbrace$. We introduce 
$$A_2 = \lbrace \xi \in B_1:\  \exists I \subset \lbrace 1 ,\dots,d \rbrace,\  \m{I}  = d-1,\ \phi_I(\xi)\in \TT^{d-1} \rbrace.$$
Let $\xi \in A_2.$  There exists $I \subset \lbrace 1 ,\dots,d \rbrace$ with  $\m{I}  = d-1$ such that $ \phi_I(\xi)\in \TT^{d-1}$ and $\m{\phi_j(\xi)} < 1 - \varepsilon_1$ for the single  $j\in\{1,\dots,d\}$
which does not belong to $I.$ So, applying again the fact, we get that there exist $\mathcal U_{2,\xi}$ a neighbourhood of $\xi$ in $\ovD$ and $C_{2,\xi}>0$ such that, for all $\eta \in \TTI,$ for all $\ovd \in (0,1)^{\m{I}}$ and for all $z\in \mathcal U_{2,\xi},$ we have 
\begin{align*}
&V_{\beta_2}(\phi_I^{-1}(\Si) \cap \mathcal U_{2,\xi}) \leq  C_{2,\xi}V_{\beta_1}(\Si)\textrm{ and }
|\phi_j(z)|<1-\veps_1.
\end{align*}
Let now $\ovd\in(0,1)^{|I|}.$
Then $\phi_j^{-1}(S(\eta,\ovd)) \cap \mathcal U_{2,\xi} = \varnothing$ 
if $\delta_j<\veps_1.$ If $\delta_j\geq \veps_1$, then
\begin{align*}
V_{\beta_2}(\phi^{-1}(\sed) \cap \mathcal U_{2,\xi}) \leq V_{\beta_2}(\phi_I^{-1}(\Si) \cap \mathcal U_{2,\xi}) &\leq C_{2,\xi}V_{\beta_1}(\Si) \\
&\leq \dfrac{ C_{2,\xi}}{\delta_j^{2+\beta_1}}\vb{1}(\sed)\\
&\leq  \dfrac{C_{2,\xi}}{\varepsilon_1^{2+\beta_1}}\vb{1}(\sed) \\
& = D_{2,\xi}\vb{1}(\sed).
\end{align*}
As $\dfrac{1}{\varepsilon_1^{2+\beta_1}}\leq \dfrac{1}{\varepsilon_1^2}$ when $\beta_1 \in (-1,0],$ we get that $D_{2,\xi}$ is independent of $\beta_i $ if $\beta_i \in (-1,0].$ Moreover, the sets $(\mathcal U_{2,\xi})_\xi$ form an open covering of $A_2$ which is compact so we can extract a finite covering $(\mathcal U_{2,k})_{k=1,\dots,s_2}$ and we have 
$$V_{\beta_2}\left(\phi^{-1}(\sed) \cap \bigcup\limits_{k=1}^{s_2} \mathcal U_{2,k}\right) \leq \sum_{k=1}^{s_2}\vb{2}(\phi^{-1}(\sed)\cap \mathcal U_{2,k}) \leq \left(\sum_{k=1}^{s_2}D_{2,k}\right)\vb{1}(\sed).$$ 
Then there exists $\varepsilon_2 \in (0 ,\varepsilon_1)$ such that for all $\xi \in B_2 := B_1 \setminus \bigcup\limits_{k=1}^{s_2}\mathcal U_{2,k}$, there exist two distinct $j_1, j_2 \in \lbrace 1,\dots,d\rbrace$ such that $\m{\phi_{j_1}(\xi)} < 1-\varepsilon_2$  and $\m{\phi_{j_2}(\xi)} < 1-\varepsilon_2$.
We repeat the reasoning with 
$$A_3 =  \lbrace \xi \in B_2:\ \exists I \subset \I{d},\ \m{I} = d-2,\ \phi_I(\xi) \in \TT^{d-2} \rbrace$$ and we define $B_3, A_4, \dots , B_{d-1}, A_d$. At the last step, there exists $\varepsilon_d \in (0,\varepsilon_{d-1})$ such that for all $\xi \in B_d := B_{d-1}\setminus \bigcup\limits_{k=1}^{s_d}\mathcal U_{d,k}$, we have $\m{\phi_j(\xi)} < 1-\varepsilon_d$ for all $j\in \lbrace 1,\dots,d\rbrace.$\\
Let $\ovd \in (0,2)^d$. Assume that at least one of the $\delta_j$ is less than $\varepsilon_d$, then $\phi^{-1}(\sed)\cap B_d = \varnothing $ so that $$\phi^{-1}(\sed) =\phi^{-1}(\sed) \cap \bigcup\limits_{l,k}\mathcal U_{l,k}. $$ 
Hence
\begin{align*}
    \vb{2}(\phi^{-1}(\sed)) \leq \sum_{l} \vb{2}\left(\phi^{-1}(\sed) \cap \bigcup\limits_{k}\mathcal U_{l,k} \right) &\leq \left(\sum_{k,l}D_{l,k}\right)\vb{1}(\sed)\\
    & = D \vb{1}(\sed).
\end{align*}
with $D>0$ independent of $\beta_i$ if $\beta_i \in (-1,0]$.
Multiplying eventually $D$ with some power of $\veps_d,$ as above, this remains true if all $\delta_j$ are bigger than $\varepsilon_d.$ In conclusion, we have shown that there exists $D(\beta_1,\beta_2)>0$  such that for all $\eta \in \TT^d$ and all $\ovd \in (0,2)^d$, we have $$\vb{2}(\phi^{-1}(\sed)) \leq D(\beta_1,\beta_2)\vb{1}(\sed).$$ 
Therefore $C_\phi$ maps continuously $A_{\beta_1}^2(\DD^d)$ into $A^2_{\beta_2}(\DD^d)$ and $\norm{C_\phi} \leq F(C,\beta_1,\beta_2)$ with $F(C,\beta_1,\beta_2)$ independent of $\beta_i $ if $\beta_i \in (-1,0]$ (see Lemma \ref{lem:independentconstant}).

\smallskip

\noindent \textit{Case 2: $C_\phi : H^2(\DD^d) \rightarrow A^2_{\beta_2}(\DD^d)$ with $\beta_2>-1$.}  

\smallskip

Let us consider $\beta_1\in(-1,0]$ and let us set $\beta'_2 = \beta_2 + d(1+\beta_1).$ We are going to prove that $C_\phi : A_{\beta_1}^2(\DD^d) \rightarrow A_{\beta_2'}^2(\DD^d) $ is bounded with $\norm{C_\phi} \leq C$, where $C>0$ is independent of $\beta_1 \in (-1,0]$. By assumption, we know that for all $\xi \in \ovD$, all $I \subset \I{d}$ such that  $\phi_I(\xi)\in \TTI$ and $I$ is maximal, there exist $\mathcal U $ a neighbourhood of $\xi$ in $\ovD$ and $C>0$ such that for all $\eta \in \TTI$ and all $\ovd \in (0,1)^{\m{I}},$
$$\vb{2}(\Ei \cap \mathcal U) \leq C\dprod_{j\in I}\delta_j. $$
We want to show that $\vb{2'}(\Ei \cap \mathcal U) \leq C\dprod_{j\in I}\delta_j^{2+\beta_1}. $ By Lemma \ref{lem:changeofbeta}, for $\beta_2' > \beta_2>-1, $ we have 
\begin{align*}
 V_{\beta_2'}(E_I(\xi,\eta,\ovd)\cap \mathcal{U})&\le\Big(\frac{\beta_2'+1}{\beta_2+1}\Big)^{d}2^{|P_{I,\xi}|(\beta_2'-\beta_2)} \prod_{k\in P_{I,\xi}}\wikxi^{\beta_2'-\beta_2}\,
V_{\beta_2}(E_I(\xi,\eta,\ovd)\cap \mathcal{U}) \\
&\le C\Big(\frac{\beta_2'+1}{\beta_2+1}\Big)^{d}2^{|P_{I,\xi}|d(\beta_1+1)} \prod_{k\in P_{I,\xi}}\wikxi^{d(\beta_1+1)}\,\dprod_{j\in I}\delta_j \\
& \le C\Big(\frac{\beta_2'+1}{\beta_2+1}\Big)^{d}2^{|P_{I,\xi}|d(\beta_1+1)} \prod_{j\in I}\delta_j^{\beta_1+1}\dprod_{j\in I} \delta_j \text{ by \eqref{eq:wikcontinuite}}\\
& \le C_{\beta_2,\beta_1} \dprod_{j\in I}\delta_j^{2+\beta_1},  
 \end{align*}
 with $C_{\beta_2,\beta_1} =C\Big(\frac{\beta_2'+1}{\beta_2+1}\Big)^{d}2^{|P_{I,\xi}|d(\beta_1+1)} \le C\Big(\frac{\beta_2+d+1}{\beta_2+1}\Big)^{d}2^{|P_{I,\xi}|d} =C_{\beta_2} $ when $\beta_1\in (-1,0].$ Hence, for all $\xi \in \ovD$, all $I \subset \I{d}$ such that  $\phi_I(\xi)\in \TTI$ and $I$ is maximal, there exist $\mathcal U$ a neighbourhood of $\xi$ in $\ovD$ and $C>0$ such that for all $\eta \in \TTI$ and all $\ovd \in (0,1)^{\m{I}},$
$$V_{\beta_2'}(E_I(\xi,\eta,\ovd)\cap \mathcal{U}) \le C_{\beta_2}\prod_{j\in I} \delta_j^{2+\beta_1}.$$
By Case 1 of the present theorem, we find that $C_\phi :  A_{\beta_1}^2(\DD^d) \rightarrow  A_{\beta_2'}^2(\DD^d)$ is bounded with $\norm{C_\phi} \leq D$ where $D$ is independent from $\beta_1 \in (-1,0]$ . Letting $\beta_1\to-1$, we find that $C_\phi : H^2(\DD^d) \rightarrow  A_{\beta_2}^2(\DD^d)$ is bounded since $\|f\|_{A_{\beta_1}^2}$ goes to $\|f\|_{H^2}$ and $\|f\|_{A_{\beta'_2}^2}$
 goes to $\|f\|_{A_{\beta_2}^2}$ for all polynomials $f$ as $\beta_1\to -1.$
\end{proof}

\begin{remark}
    We do not know whether Theorem \ref{result_gen} $(c)\implies (a)$ remains true for $\beta_1=\beta_2=-1$. It does under the restricted variables property, see Corollary \ref{result_gen_restricted}.
\end{remark}

\color{black}

\subsection{Applications - automatic continuity}
We now develop some applications of Theorem \ref{result_gen} which allow to give several results of automatic continuity of composition operators.
We start with a stability result about the continuity of composition operators:

\begin{theorem}\label{thm:stability}
Let $\phi \in \mathcal{O}(\DD^d,\DD^d)\cap \mathcal{C}^1(\ovD)$ and $\beta_1,\beta_2> -1$. Assume that $C_\phi$ maps boundedly $A_{\beta_1}^2(\DD^d)$ into $A_{\beta_2}^2(\DD^d)$. Then for all $\beta'_1\geq \beta_1,$ $C_\phi$ maps boundedly $A_{\beta_1'}^2(\DD^d)$ into $A_{\beta_2'}^2(\DD^d)$ where 
$$\beta_2'=\frac{\beta_1'(\beta_2+2)+2(\beta_2-\beta_1)}{\beta_1+2}.$$
In particular, if $C_\phi$ is bounded on $A_{\beta_1}^2(\DD^d),$ then it is bounded on $A_{\beta_1'}^2(\DD^d).$
\end{theorem}

\begin{proof}
Let $I\subset\{1,\dots,d\}$ and $\xi\in\oDD^d$ such that $\phi_I(\xi)\in\TTI$.
 We know that there exists $C>0$ such that, for all $\eta\in\TT^{|I|}$ and all $\ovd\in\oii$, 
$$V_{\beta_2}(E_I(\xi,\eta,\ovd))\leq C \prod\limits_{j\in I}\delta_j^{2+\beta_1}.$$
By Lemma \ref{lem:changeofbeta}, 
\begin{align*}
 V_{\beta_2'}(E_I(\xi,\eta,\ovd))&\le C(\beta_2,\beta_2')  \prod_{k\in P_{I,\xi}}\wikxi^{\beta_2'-\beta_2}V_{\beta_2}(E_I(\xi,\eta,\ovd)) \\
 &\leq C(\beta_2,\beta_2') \prod_{k\in P_{I,\xi}}\wikxi^{\beta'_2-\beta_2}\prod\limits_{j\in I}\delta_j^{2+\beta_1}
\end{align*}
(the value of $C(\beta_2,\beta'_2)$ has changed from line to line).

Hence, to apply Theorem \ref{result_gen}, we only have to show that
$$\prod_{k=1}^d \omega_{I,k,\xi}(\ovd)^{\beta'_2-\beta_2}  \leq C \prod_{j\in I} \delta_j^{\beta'_1-\beta_1}$$
for some $C$ which does not depend on $I,\xi,\eta,\ovd$. This is easy: if $D,$ $a,$ $b$, $x$, $y$, $x',$ $y'$ are positive real numbers such that $a^x\leq D b^y,$ then $a^{x'}\leq D^{x'/x}b^{y'}$ as soon as $y'=yx'/x.$ We apply this inequality starting from Lemma \ref{lem:wikcontinu} with $x=2+\beta_2$, $y=2+\beta_1$, $x'=\beta'_2-\beta_2$ and $y'=\beta'_1-\beta_1$. We immediately get the relation for $\beta_2'$ given in the statement of the theorem.
\end{proof}

Part of the previous result was announced in \cite[Corollary 6]{Jaf90}, without any assumption of regularity at the boundary of $\phi.$ However, the proof of \cite{Jaf90} is not correct, since it uses Theorem C of \cite{Jaf91} with a measure which does not satisfy the assumptions of this last theorem.

\smallskip

In the particular case of the tridisc, Theorem \ref{thm:stability} can be used in conjunction with Theorem $10$ of \cite{Ko22} to lead to the following result. 

\begin{theorem}\label{thm:weighttridisc}
Let $\phi \in \mathcal{O}(\DD^3,\DD^3) \cap \mathcal{C}^2(\ovdt). $ The following propositions are equivalent: 
\begin{enumerate}
\item $C_\phi $ is bounded on $A^2(\DD^3)$.
\item For all $\beta \geq 0,$ $C_\phi $ is bounded on $A^2_\beta(\DD^3)$.
\item There exists $\beta \geq 0$ such that  $C_\phi $ is bounded on $A^2_\beta(\DD^3)$.
\item For any $\xi\in\TT^3$ such that $\phi(\xi)\in\TT^3,$ $d\phi(\xi)$ is invertible and for any $1\leq i_1<i_2\leq 3$ and any $\xi\in\TT^3$ such that $(\phi_{i_1}(\xi),\phi_{i_2}(\xi))\in\TT^2,$ either
\begin{enumerate}
    \item $\nabla\phi_{i_1}(\xi),\nabla\phi_{i_2}(\xi)$ are linearly independent, or
    \item $\frac{\partial \phi_{i_k}}{\partial z_j}(\xi)\neq 0$ for $j=1,2,3$ and $k=1,2.$
\end{enumerate}
\end{enumerate}
\end{theorem}

 \begin{proof}
1) $\Rightarrow$ 2) This is exactly the content of Theorem \ref{thm:stability}.\\
2) $\Rightarrow$ 3) Obvious.\\
3) $\Rightarrow$ 4) To prove this implication, we just need to check if what was done in \cite{Ko22} can be extended to $\beta \geq 0. $  Let $\beta \geq 0 $  and assume that $C_\phi $ is bounded on $A^2_\beta(\DD^3).$ 
Let $\xi\in\TT^3.$ If $\phi(\xi)\in\TT^3,$ the necessary condition comes from \cite{KSZ08}. If $\phi_I(\xi)\in \TT^2$ for some $I\subset\{1,2,3\}$ with $|I|=2$, and if the necessary condition is not satisfied, the work of \cite{Ko22} shows that one can assume that there exist $C>0$, $\eta\in\TT^3$ and $\alpha_1,\alpha_2\in\mathbb C$ such that, for all $\delta\in(0,1)$, setting $\ovd=(\delta,\delta,1)$, 
$$\{z\in\DD^3:\ |\alpha_1(z_1-1)+\alpha_2(z_2-1)|<C\delta\}\subset \phi^{-1}(S(\eta,\ovd)).$$
Now, $V_{\beta}(\{z\in\DD^3:\ |\alpha_1(z_1-1)+\alpha_2(z_2-1)|<C\delta\})\gtrsim \delta^{2\beta+7/2}$ by \cite[Proposition 5]{Ko22} whereas $V_\beta(S(\eta,\ovd))\simeq \delta^{2\beta+4}.$ Hence, $C_\phi$ cannot be bounded on $A_\beta^2(\DD^3)$.\\
4) $\Rightarrow$ 1) This is Theorem 10 of \cite{Ko22}.
 \end{proof}   

In \cite[Theorem A]{SZ06b}, it is shown that if $\phi\in \mathcal O(\DD^d,\DD^d)$, then $C_\phi$ always defines a bounded composition operator from $A^2_{\beta}(\DD^d)$ into $A^2_{d(\beta+2)-2}(\DD^d)$. As another application of Theorem \ref{result_gen}, we give an enhancement of this result where the target space depends on the number of coordinates that can touch simultaneously the polycircle. 

\begin{theorem}\label{thm:automaticcontinuity}
    Let $\phi\in\mathcal O(\DD^d,\DD^d)\cap\mathcal C^1(\ovD)$
    such that $\phi(\ovD)\cap\partial\DD^d\neq\varnothing$, let $\beta> -1$
    and let $d_\phi=\max\{|I|:\ I\subset\{1,\dots,d\},\ \phi_I(\ovD)\cap \TT^{|I|}\neq\varnothing\}.$ Then $C_\phi$ maps boundedly $A_\beta^2(\DD^d)$ into $A^2_{d_\phi(\beta+2)-2}(\DD^d)$ and this is optimal: for all $q\in\{1,\dots,d\}$, there exists $\phi\in\mathcal O(\DD^d,\DD^d)\cap\mathcal C^1(\ovD)$
    such that $d_\phi=q$ and $C_\phi$ does not map $A_\beta^2(\DD^d)$ into $A_{\beta'}^2(\DD^d)$ for any $\beta'<d_\phi(\beta+2)-2$.
\end{theorem}

The proof will need the following easy lemma.
\begin{lemma}\label{lem:meanvalue}
Let $\varphi\in \mathcal O(\DD^d,\DD) \cap \mathcal C^1(\ovD)$ and let $\xi\in\ovD$ be such that $\varphi(\xi)\in\TT$. For all $P\subset\{1,\dots,d\}$ with $P_{\varphi,\xi}\subset P,$ there exist $C>0$ and $\mathcal U$ a neighborhood of $\xi$ in $\ovD$ such that,  for all $w\in \oDD^{|P^c|}$, for all $\rho\in[0,1)^{|P|}$, for all $\eta\in\mathbb C$, for all $\delta>0$, 
$$\lambda_{|P|}\left(\left\{\theta\in[-\pi,\pi)^{|P|}:\ ((1-\rho)e^{i\theta},w)\in\mathcal U,\ |\varphi(\pi_{\varphi,\xi}((1-\rho)e^{i\theta},w))-\eta|<\delta\right\}\right)\leq C\delta$$
where $\pi_{\varphi,\xi}(z)=z'$ with $z'_k=z_k/|z_k|$ if $k\in P_{\varphi,\xi}$, $z'_k=z_k$ otherwise.
\end{lemma}
\begin{proof}
We may and shall assume that $\hat\xi=e$, that $\varphi(\xi)=1$ and that $1\in P_{\varphi,e}$. 
By Lemma \ref{JC} (c), for all $w\in \DD^{|P^c|}$, 
$$\frac{\partial\varphi}{\partial z_1}(e_P,w)=\frac{\partial \varphi}{\partial z_1}(e)>0.$$
Therefore, there exist $c>0$ and $\mathcal U$ a neighborhood of $\xi$ in $\ovD$  which may be written
$$\mathcal U=\{z:\ |\theta_k|<r,\ \rho_k<r,\ k\in P_{\varphi,\xi},\ |z_k-\xi_k|<r,\ k\notin P_{\varphi,\xi}\}$$
such that, for all $\theta\in[-\pi,\pi)^{|P|}$, for all $\rho\in[0,1)^{|P|}$ and all $w\in \oDD^{|P^c|}$ satisfying $((1-\rho)e^{i\theta},w)\in\mathcal U,$ 
$$\Re e\left(e^{i\theta_1}\frac{\partial\varphi}{\partial z_1}(\pi_{\varphi,\xi}((1-\rho)e^{i\theta},w))\right)\geq c.$$
Let $Q=P\backslash\{1\}$ and $\tilde\theta\in[-\pi,\pi)^{|Q|}$. Let $\theta_1,\theta'_1\in\mathbb R$ and set $\theta=(\theta_1,\tilde\theta)$, $\theta'=(\theta_1',\tilde\theta)$. Provided $((1-\rho)e^{i\theta},w)$ and $((1-\rho)e^{i\theta'},w)$ belong to $\mathcal U$ (for fixed $\rho$, $\tilde\theta$ and $w$, the set of $\theta_1$ such that $((1-\rho)e^{i(\theta_1,\tilde\theta)},w)\in\mathcal U$ is the interval $(-r,r)$, so that the mean value inequality applies to $h$ between $\theta_1$ and $\theta_1'$), 
$$|\varphi(\pi_{\varphi,\xi}((1-\rho)e^{i\theta},w))-\varphi(\pi_{\varphi,\xi}((1-\rho)e^{i\theta'},w))|\geq c|\theta_1-\theta'_1|.$$
Indeed,
$$\varphi(\pi_{\varphi,\xi}((1-\rho)e^{i\theta},w))-\varphi(\pi_{\varphi,\xi}((1-\rho)e^{i\theta'},w))=\varphi(e^{i\theta_1},\tilde w)-\varphi(e^{i\theta_1'},\tilde w)$$
for some $\tilde w\in \oDD^{d-1}$ and the function $h:t\mapsto \varphi(e^{it},\tilde w)$ satisfies 
$$h'(t)=ie^{it}\frac{\partial\varphi}{\partial z_1}(e^{it},\tilde w)$$
so that $\Im m(h'(t))\geq c$ provided $(e^{it},\tilde w)\in\mathcal U$.

Therefore, $\tilde \theta$, $\rho$ and $w$ being fixed, 
$$\left\{\theta_1\in[-\pi,\pi):\ ((1-\rho)e^{i(\theta_1,\tilde\theta)},w)\in\mathcal U,\ |\varphi(\pi_{\varphi,\xi}((1-\rho)e^{i(\theta_1,\tilde\theta)},w))-\eta|<\delta\right\}$$
is contained in an interval of length $2\delta/c$. The lemma then follows from an application of Fubini's theorem.
\end{proof}

\begin{proof}[Proof of Theorem \ref{thm:automaticcontinuity}]
We show that Condition (c) of Theorem \ref{result_gen} is satisfied. Let $\xi\in\ovD$ and $I\subset\{1,\dots,d\}$ maximal be such that $\phi_I(\xi)\in\TTI.$ Set $P=P_{I,\xi}$. Let $\mathcal U,C>0$ be given by Lemma \ref{lem:meanvalue} and working for each $\phi_j,$ $j\in I.$ Let also $\eta\in\TTI,$ $\ovd\in(0,1)^{|I|}$, $w\in \oDD^{|P^c|}$, $\rho\in[0,1)^{|P|}$.

Then
\begin{align*}
\lambda_{|P|}\left(\left\{\theta\in[-\pi,\pi)^{|P|}:\ ((1-\rho)e^{i\theta},w)\in\mathcal U,\ |\phi_j(\pi_{j,\xi}((1-\rho)e^{i\theta},w))-\eta_j|<\delta_j,\ \forall j\in I\right\}\right) \\
\leq \min_{j\in I}\lambda_{|P|}\left(\left\{\theta\in[-\pi,\pi)^{|P|}:\ ((1-\rho)e^{i\theta},w)\in\mathcal U,\ |\phi_j(\pi_{j,\xi}((1-\rho)e^{i\theta},w))-\eta_j|<\delta_j\right\}\right)\\
\leq C \min_{j\in I}\delta_j.
\end{align*}

Let us also set
$$F_I^{\mathcal U}(\xi,\eta,\ovd,\rho)=\left\{(\theta,w)\in F_I(\xi,\eta,\ovd,\rho):\ ((1-\rho)e^{i\theta},w)\in\mathcal U\right\},$$
so that the sections of $E_I(\xi,\eta,\ovd)\cap\mathcal U$ at fixed $\rho$ are exactly the sets
$F_I^{\mathcal U}(\xi,\eta,\ovd,\rho)$. The previous inequality says that, for all
$w\in\oDD^{|P^c|}$, the section $\{\theta:\ (\theta,w)\in F_I^{\mathcal U}(\xi,\eta,\ovd,\rho)\}$
has $\lambda_{|P|}$-measure at most $C\min\limits_{j\in I}\delta_j$. Since $V_{d_\phi(\beta+2)-2}$ is a
probability measure on $\oDD^{|P^c|}$, Fubini's theorem yields
$$\lambda_{|P|}\otimes V_{d_\phi(\beta+2)-2}\big(F_I^{\mathcal U}(\xi,\eta,\ovd,\rho)\big)\leq C\min_{j\in I}\delta_j .$$

Now,
$$\min_{k\in P} \omega_{I,k,\xi}(\ovd)=\min_{j\in I}\delta_j$$
and let $k_0\in P$ such that $\omega_{I,k_0,\xi}(\ovd)=\min\limits_{j\in I}\delta_j.$
It follows from \eqref{eq:volume} that
\begin{align*}V_{d_\phi(\beta+2)-2}(E_I(\xi,\eta,\ovd)\cap\mathcal U)&\lesssim \min_{j\in I}\delta_j \int_{\prod\limits_{k\in P}[0,\wikxi]}\prod_{k\in P}W_{d_\phi(\beta+2)-2}(\rho_k)d\rho_k\\
&\lesssim \min_{j\in I}\delta_j \int_{[0,\omega_{I,k_0,\xi}]} W_{d_\phi(\beta+2)-2}(\rho)d\rho\\
& \lesssim \min_{j\in I}\delta_j  \left(\min_{j\in I}\delta_j\right)^{d_\phi(\beta+2)-1}\\
&\lesssim \prod_{j\in I} \delta_j^{\frac{d_\phi(\beta+2)}{|I|}}\\
&\lesssim \prod_{j\in I} \delta_j^{2+\beta}
\end{align*}

since $|I|\leq d_\phi$.
This implies that $C_\phi$ maps boundedly $A_\beta^2(\DD^d)$ into $A^2_{d_\phi(\beta+2)-2}(\DD^d).$

Conversely, consider the function $\phi \in \mathcal{O}(\DD^d,\DD^d)\cap \mathcal C^1(\overline \DD^d) $ defined by $\phi(z) = (z_1,\dots, z_1,0,\dots, 0)$ where $z_1$ is repeated $q$ times, so that $d_\phi=q.$ Let $\xi=\eta=e,$ $\ovd=(\delta,\dots,\delta)$ for $\delta\in(0,1)$, $I=\{1,\dots,q\}$. In that case, it is very easy to describe $E_I(\xi,\eta,\ovd):$
$$E_I(\xi,\eta,\ovd)=\left\{z=((1-\rho)e^{i\theta})\in\ovD: 0\leq \rho_1\leq\delta,\ |e^{i\theta_1}-1|<\delta\right\}.$$
Therefore, for all $\beta'>-1,$ $V_{\beta'}(E_I(\xi,\eta,\ovd))$ behaves like $\delta^{2+\beta'}$, whereas $\dprod_{j\in I}\delta_j^{2+\beta}=\delta^{q(2+\beta)}$. Hence, if $C_\phi$ maps $A^2_\beta(\DD^d)$ into $A^2_{\beta'}(\DD^d)$, it is necessary that $\beta'\geq q(\beta+2)-2$ by (b) of Theorem \ref{result_gen}.
\end{proof}

\color{black}

\subsection{Symbols with the restricted variables property}

It can be difficult to apply Theorem \ref{result_gen} in practice because
the sets $E_I(\xi,\eta,\ovd)$ do not seem to be more tractable than the preimages of Carleson windows. We now develop some particular cases which will be useful to get practical examples. 

\begin{definition}
Let $\phi\in\mathcal O(\DD^d,\DD^d)\cap\mathcal C^1(\ovD)$, $\xi\in\ovD$ and $I\subset\{1,\dots,d\}$ such that $\phi_I(\xi)\in\TT^{|I|}$. We say that $\phi$ has the
{\bf restricted variables property at $\xi$ with respect to $I$} if, for all $j\in I$, there exists $\psi_{j,\xi}\in\mathcal O(\DD^{|P_{j,\xi}|},\DD)\cap\mathcal C^1(\overline \DD^{|P_{j,\xi}|})$ such that for all $z\in\DD^d$, $\phi_j(z)=\psi_{j,\xi}(z_{P_{j,\xi}}).$ 
We say that $\phi$ has the {\bf restricted variables property} if it has the restricted variables property at $\xi$ with respect to $\{j\}$ for all $\xi\in\ovD$ and all $j\in\{1,\dots,d\}$ such that $\phi_j(\xi)\in\TT$ (or equivalently if it has the restricted variables property at $\xi$ with respect to $I$ for all $\xi\in\ovD$ and all $I\subset\{1,\dots,d\}$ with $\phi_I(\xi)\in\TT^{|I|}$).
\end{definition}

Let $j\in\{1,\dots,d\}$ and $\xi\in\ovD$ such that $\phi$ has the restricted variables property at $\xi$ with respect to $\{j\}$. Then for all $\xi'\in\ovD$ such that $\phi_j(\xi')\in\TT,$ we have $P_{j,\xi'}\subset P_{j,\xi}$. In particular, if $\phi$ has the restricted variables property (as a global property), $P_{j,\xi}$ does not depend on $\xi$ and we shall denote it by $P_j$. Likewise, for $I\subset\{1,\dots,d\}$ and $\xi\in\ovD$ with $\phi_I(\xi)\in\TT^{|I|}$, $P_{I,\xi}$ and $\wikxi$ do not depend on $\xi$ and we shall omit the subscript $\xi$ in that context.

The following immediate remark will be useful several times: if $P_{j,\xi}=\{1,\dots,d\}$ for all $j\in I,$ then $\phi$ admits the restricted variables property at $\xi$ with respect to $I.$

\smallskip

It will be interesting to work with maps satisfying the restricted variables property
because it will be easier to estimate $V_\beta(E_{I}(\xi,\eta,\ovd))$
for any $\beta\geq -1$.
Indeed, if $\phi$ has the restricted variables property at $\xi$ with respect to $I$ and if $z\in\DD^d$ is such that $z_k\neq 0$ for all $k$, for all $j\in I$, then 
$$\phi_j(\pi_{j,\xi}(z))=\phi_j(z_1/|z_1|,\dots,z_d/|z_d|),$$
since $\phi_j$ does not depend on the variables that are unchanged by $\pi_{j,\xi}$. Therefore, the sets $E_I(\xi,\eta,\ovd)$, introduced in \eqref{eq:EIxieta} have the simpler form:
$$E_I(\xi,\eta,\ovd)= \lbrace ((1-\rk)e^{i\tk}) : 0\leq \rk \leq \wikxi,\ k \in P_{I,\xi} \text { and }  |\phi_j(e^{i\theta})-\eta_j|<\delta_j,\ \forall j\in I\rbrace. $$
This easily leads to the following lemma.
\begin{lemma}\label{lem:comparaisontwoCarlesonsets}
Let $\phi\in\mathcal O(\DD^d,\DD^d)\cap\mathcal C^1(\ovD)$, let $\xi\in\ovD$ and $I\subset\{1,\dots,d\}$ be such that
$\phi_I(\xi)\in\TT^{|I|}$ and $\phi$ has the restricted variables property at $\xi$ with respect to $I$. Let
$\theta_0\in[-\pi,\pi)^d$ be such that $\hat\xi=e^{i\theta_0}$ and let $\mathcal O$ be a neighbourhood of $\theta_0$ in
$\RR^d$. Then there exist a neighbourhood $\mathcal U$ of $\xi$ in $\ovD$ and $C>0$ such that, for all $\beta\geq-1$,
all $\eta\in\TT^{|I|}$ and all $\ovd\in(0,1)^{|I|}$,
\begin{equation}\label{eq:restricted1}
V_\beta(E_I(\xi,\eta,\ovd)\cap\mathcal U)\leq C\,2^{d(1+\beta)}\prod_{k=1}^d\omega_{I,k,\xi}(\ovd)^{1+\beta}
\lambda_d\big(\{\theta\in\mathcal O:\ |\phi_j(e^{i\theta})-\eta_j|<\delta_j,\ j\in I\}\big).
\end{equation}
If $\mathcal O=[-\pi,\pi)^d$, one may take $\mathcal U=\ovD$. Conversely,
\begin{equation}\label{eq:restricted2}
V_\beta(E_I(\xi,\eta,\ovd))\geq\frac{1}{(2\pi)^d}\prod_{k=1}^d\omega_{I,k,\xi}(\ovd)^{1+\beta}
\lambda_d\big(\{\theta\in[-\pi,\pi)^d:\ |\phi_j(e^{i\theta})-\eta_j|<\delta_j,\ j\in I\}\big).
\end{equation}
\end{lemma}
\begin{proof}
    Write $P=P_{I,\xi}$ and assume without loss of generality that $\theta_0=0$. Let $a\in(0,\pi)$ be such that
$B=(-a,a)^d\subset\mathcal O$ and set
$$\mathcal U=\{z\in\ovD:\ z_k=r_ke^{i\theta_k}\neq0\textrm{ and }\theta_k\in(-a,a)\textrm{ for all }k\in P\},$$
which is a neighbourhood of $\xi$ since $|\xi_k|=1$ for $k\in P$. Let
$A=\{\theta:\ |\phi_j(e^{i\theta})-\eta_j|<\delta_j,\ j\in I\}$. By the restricted variables property at $\xi$,
$\phi_j(\pi_{j,\xi}(z))=\phi_j(e^{i\theta})$ for $j\in I$, and $A$ only depends on $\theta_P$. Hence, setting $ B'=(-a,a)^{P}\times[-\pi,\pi)^{P^c}$
$$E_I(\xi,\eta,\ovd)\cap\mathcal U=\big\{((1-\rho_k)e^{i\theta_k}):\ 0\leq\rho_k\leq\omega_{I,k,\xi}(\ovd),\ k\in P,
\ \theta\in A\cap B'\big\}.$$
    Using Fubini's theorem and polar integration as when we obtained \eqref{eq:volume} and using
$$\int_{1-\omega}^1 2{(\beta+1)}r(1-r^2)^{\beta}dr=(1-(1-\omega)^2)^{\beta+1}=(\omega(2-\omega))^{\beta+1}$$
we get
$$V_\beta(E_I(\xi,\eta,\ovd)\cap\mathcal U)=\prod_{k\in P}\big(\omega_{I,k,\xi}(\ovd)(2-\omega_{I,k,\xi}(\ovd))\big)^{1+\beta}
\,\frac{\lambda_d(A\cap B')}{(2\pi)^d}.$$
Since $A$ only depends on $\theta_P$, $\lambda_d(A\cap B')=(\pi/a)^{|P^c|}\lambda_d(A\cap B)\leq(\pi/a)^{d}\lambda_d(A\cap\mathcal O)$,
and \eqref{eq:restricted1} follows from $2-\omega\leq 2$, with $C=(2\pi)^{-d}(\pi/a)^d$. When $\mathcal O=[-\pi,\pi)^d$,
the same computation with $\mathcal U=\ovD$ and $B'=[-\pi,\pi)^d$ works, and \eqref{eq:restricted2} follows from
$2-\omega\geq1$. For $\beta=-1$, all these formulas remain true with $V_{-1}=\sigma_d$.
\end{proof}

Hence, Theorem \ref{result_gen} admits the following corollary.

\begin{corollary}\label{result_gen_restricted}
Let $\beta_1, \beta_2\geq-1$.  Let $\phi \in \mathcal{O}(\DD^d,\DD^d) \cap \mathcal{C}^1(\ovD)$ with the restricted variables property. Then, the following assertions are equivalent:  
\begin{enumerate}[(a)]
\item ${C_\phi : A_{\beta_1}^2(\DD^d) \rightarrow A_{\beta_2}^2(\DD^d)}$ is bounded. \\
\item For all $I\subset\{1,\dots,d\}$ such that $\phi_I(\TT^d)\cap\TT^{|I|}\neq\varnothing$, there exists $C>0$ such that,
for all $\eta\in\TTI$ and all $\ovd\in(0,1)^{|I|}$,
\begin{equation*}\label{eq:carlesonbordrestricted1}
\lambda_d\big(\{\theta\in[-\pi,\pi)^d:\ |\phi_j(\theta)-\eta_j|<\delta_j,\ j\in I\}\big)\leq C\,
\frac{\dprod_{j\in I}\delta_j^{2+\beta_1}}{\dprod_{k=1}^d\omega_{I,k}(\ovd)^{1+\beta_2}}.
\end{equation*}
\item For all $\theta_0 \in \RR^d$, all $I \subset \I{d}$ such that  $\phi_I(\theta_0)\in \TTI$ and $I$ is maximal with respect to this property, there exist a neighborhood $O$ of $\theta_0$ in $\RR^d$ and $C>0$ such that for all $\eta \in \TTI$ and all $\ovd \in (0,1)^{\m{I}},$
\begin{equation*} \label{eq:carlesonbordrestricted2}
\lambda_d(\lbrace \theta \in O : \m{\phi_j(\theta)-\eta_j}<\delta_j,\ j\in I \rbrace) \leq C \dfrac{\prod\limits_{j\in I}\delta_j^{2+\beta_1}}{\prod\limits_{k=1}^d\omega_{I,k}(\ovd)^{1+\beta_2}}. 
\end{equation*}
\end{enumerate}
\end{corollary} 
\begin{proof}
    That (b) of Corollary \ref{result_gen_restricted} is equivalent to (b) of Theorem \ref{result_gen} immediately follows from \eqref{eq:restricted1}
    and \eqref{eq:restricted2}. 

    That (c) of Corollary \ref{result_gen_restricted} implies (c) of Theorem \ref{result_gen} is slightly more delicate. Indeed, pick $\xi\in\ovD$ and $I\subset\{1,\dots,d\}$ such that $\phi_I(\xi)\in\TTI$ and $I$ is maximal with respect to this property. Let $\theta_0\in\RR^d$ such that $\hat\xi=e^{i\theta_0}$. We know that $\phi_I(\theta_0)\in\TTI$ but it is not true that $I$ is always maximal with respect to this property (think of $\phi(z)=(z_1,z_2)$, $\xi=(1,0)$, $I=\{1\}$ whereas $\hat\xi=(1,1)$ and $\phi(\hat \xi)\in\TT^2$ so that $I$ is not maximal with respect to $\hat\xi$). We nevertheless claim that property (c) remains true for this $\theta_0$ and this $I$.  Let $\eta\in\TT^{|I|}$ and $\ovd\in(0,1)^{|I|}$. We consider $I'\supset I$ such that $\phi_{I'}(\theta_0)\in\TT^{|I'|}$ and $I'$ is maximal with respect to this property. Let $O'$ and $C'$ be given by (c)
for $\theta_0$ and $I'$. 

We extend $\ovd$ and $\eta$ to $I'$ by setting
$\eta'_j=\eta_j$, $\delta'_j=\delta_j$ for $j\in I$, $\eta'_i=\phi_i(e^{i\theta_0})\in\TT$ and  $\delta'_i=1/2$ for $i\in I'\backslash I$. 
By continuity, there exists a neighborhood $O\subset O'$ of $\theta_0$ such that
$|\phi_i(e^{i\theta})-\eta'_i|<1/2$ for all $\theta\in O$ and all $i\in I'\backslash I$.

 Then
$$\{\theta\in O:\ |\phi_j(e^{i\theta})-\eta_j|<\delta_j,\ j\in I\}
=\{\theta\in O:\ |\phi_i(e^{i\theta})-\eta'_i|<\delta'_i,\ i\in I'\},$$
so that, by (c),
$$\lambda_d\big(\{\theta\in O:\ |\phi_j(e^{i\theta})-\eta_j|<\delta_j,\ j\in I\}\big)\leq
C'\,\frac{2^{-|I'\backslash I|(2+\beta_1)}\dprod_{j\in I}\delta_j^{2+\beta_1}}{\dprod_{k=1}^d\omega_{I',k}(\ovd')^{1+\beta_2}}
\leq
C'\,\frac{\dprod_{j\in I}\delta_j^{2+\beta_1}}{\dprod_{k=1}^d\omega_{I',k}(\ovd')^{1+\beta_2}}.$$
Now, for all $k\in\{1,\dots,d\}$, $\omega_{I',k}(\ovd')$ is the minimum of $\omega_{I,k}(\ovd)$ and
of the numbers $\delta'_i=1/2$ for $i\in I'\backslash I$ with $k\in P_i$.
Since $\omega_{I,k}(\ovd)\leq 1$, we get $\omega_{I',k}(\ovd')\geq \frac 12\omega_{I,k}(\ovd)$, hence
$$\prod_{k=1}^d\omega_{I',k}(\ovd')^{1+\beta_2}\geq 2^{-d(1+\beta_2)}\prod_{k=1}^d\omega_{I,k}(\ovd)^{1+\beta_2}.$$
This yields the claim with $C=C'\,2^{d(1+\beta_2)}$ (observe that $C$ adds no extra dependence on $\beta_2$ provided $\beta_2\in[-1,0)$). 
We then deduce from Lemma \ref{lem:comparaisontwoCarlesonsets} the existence of $\mathcal U$ and $C''>0$ such that
$$V_{\beta_2}\big(E_I(\xi,\eta,\ovd)\cap\mathcal U\big)\leq C''\prod_{j\in I}\delta_j^{2+\beta_1}.$$

It remains to prove that $(c)\implies (a)$ when the domain and the target space are both $H^2(\DD^d)$, the case left open in Theorem \ref{result_gen}.
Let $\beta_1\in(-1,-1+1/d)$ and set $\beta_2=d(1+\beta_1)-1$ so that $\beta_2\in(-1,0)$ and $\beta_2\to -1$ as $\beta_1\to -1$. Keeping the same notations as above,
\begin{align*}
 V_{\beta_2}(E_I(\xi,\eta,\ovd)\cap\mathcal U)& \leq C 2^{d(\beta_2+1)}\prod_{k\in P_I}\wik^{1+\beta_2} \prod_{j\in I}\delta_j  \textrm{ by \eqref{eq:restricted1}}\\
 &\leq C 2^{d}\prod_{j\in I}
 \delta_j^{(1+\beta_2)/d}\prod_{j\in I}\delta_j \textrm{ by \eqref{eq:wikcontinuite}}\\
 &\leq  C 2^{d} \prod_{j\in I}
 \delta_j^{2+\beta_1}
\end{align*}
by our choice of $\beta_2$. This implies that there exists some constant $D,$ which does not depend on $\beta_1,\beta_2$ since they both belong to $(-1,0)$, such that 
$\|C_\phi\|_{A_{\beta_1}^2(\DD^d)\to A_{\beta_2}^2(\DD^d)}\leq D$. Letting $\beta_1$ to $-1$ yields the result.
\end{proof}

The implication $(a)\implies (b)$ of the previous corollary does not require the global version of the restricted variables property. Hence, we also obtain:

\begin{corollary}\label{cor:necconditionrestricted}
Let $\beta_1, \beta_2\geq-1$.  Let $\phi \in \mathcal{O}(\DD^d,\DD^d) \cap \mathcal{C}^1(\ovD)$ be such that
 ${C_\phi : A_{\beta_1}^2(\DD^d) \rightarrow A_{\beta_2}^2(\DD^d)}$ is bounded. 
Then for all $I\subset\{1,\dots,d\}$ and all $\xi\in\ovD$ such that $\phi_I(\xi)\in\TT^{|I|}$ and $\phi$ has the restricted variables property at $\xi$ with respect to $I$, there exists $C>0$ such that,
for all $\eta\in\TTI$ and all $\ovd\in(0,1)^{|I|}$,
\begin{equation}\label{eq:carlesonbordrestrictedxi}
\lambda_d\big(\{\theta\in[-\pi,\pi)^d:\ |\phi_j(\theta)-\eta_j|<\delta_j,\ j\in I\}\big)\leq C\,
\frac{\dprod_{j\in I}\delta_j^{2+\beta_1}}{\dprod_{k=1}^d\omega_{I,k,\xi}(\ovd)^{1+\beta_2}}.
\end{equation}
\end{corollary}

These considerations lead to necessary and sufficient conditions for boundedness where the quantities $\omega_{I,k,\xi}$ are not involved.

\begin{corollary}\label{cont_restricted}
Let $ \beta_1, \beta_2 \geq -1. $ Let $\phi\in\mathcal O( \DD^d, \DD^d)\cap\mathcal C^1(\ovD) $ such that $\phi$ has the restricted variables property. 
Assume that for all $\theta_0 \in \RR^d$ and all $I \subset \lbrace 1, \dots, d\rbrace $ such that $\phi_I(e^{i\theta_0}) \in \TTI$ and $I$ is maximal with respect to this property, there exist $C>0$, $\alpha\in\RR$ and $O$ a neighbourhood of $\theta_0$ in $\RR^d$ such that for all $\eta \in \TT^{\m{I}}$ and all $\ovd \in (0,1)^{\m{I}}$, 
 \begin{equation*}
\lambda_d(\lbrace \theta \in O : \m{\phi_j(\theta)- \eta_j } < \delta_j,\ j \in I\rbrace) \leq C \left(\displaystyle\prod_{j\in I}\delta_j\right)^\alpha, 
\end{equation*}
with $\alpha + \gamma_I(1+\beta_2) \geq 2+\beta_1$. Then $C_\phi : \Abdi{1}\rightarrow \Abdi{2} $ is bounded.  
\end{corollary}

\begin{proof}
This immediately follows from \eqref{eq:wikalpha} and an application of Corollary \ref{result_gen_restricted}.
\end{proof}

\begin{corollary} \label{discont_restricted}
Let $\beta_1, \beta_2\geq -1.$ Let $\phi \in\mathcal O( \DD^d,\DD^d)\cap \mathcal{C}^1(\ovD)$, $\xi\in\ovD$ and $I\subset\{1,\dots,d\}$ be such that $\phi$ has the restricted variables property at $\xi$ with respect to $I.$  Assume that there exist $c>0$ and $\eta \in \TTI$ such that for all $\delta \in (0,1),$ 
$$\lambda_d(\lbrace \theta \in [-\pi,\pi)^d : \m{\phi_j(\theta)-\eta_j}<\delta,\  j \in I\rbrace) \geq c f(\delta)$$
where 
$$\lim_{\delta\to 0}\frac{f(\delta)}{\delta^{|I|(2+\beta_1)-|P_{I,\xi}|(1+\beta_2)}}=+\infty.$$
Then $C_\phi : A^2_{\beta_1}(\DD^d) \to A^2_{\beta_2}(\DD^d)$ is not bounded. 
This holds in particular with $f(\delta)=\delta^{\alpha |I|}$ when 
$\alpha + \dfrac{\m{P_{I,\xi}}}{\m{I}}(1+\beta_2) <2 +\beta_1.$
\end{corollary}

\begin{proof}
When $\delta_j=\delta$ for all $j\in I$, we get $\omega_{I,k,\xi}(\ovd)=\delta$ for $k\in P_{I,\xi}$ and $1$ otherwise. Therefore inequality \eqref{eq:carlesonbordrestrictedxi} reads 
$$\lambda_d(\lbrace \theta \in [-\pi,\pi)^d : \m{\phi_j(\theta)-\eta_j}<\delta,\   j \in I\rbrace) \leq C\delta^{|I|(2+\beta_1)-|P_{I,\xi}|(1+\beta_2)}.$$
Hence Corollary \ref{discont_restricted} follows from Corollary \ref{cor:necconditionrestricted}.
\end{proof}

\color{black}
\subsection{A class of examples}
We finish this subsection with a class of symbols which will be used several times throughout the paper.

\begin{theorem}\label{thm:examples}
Let $\beta_1,\beta_2\geq -1.$ Let $\varphi\in\mathcal O(\DD,\DD)\cap \mathcal C^1(\overline\DD)$. Assume that, for any $z\in \overline\DD,$ $|\varphi(z)|=1$ if and only if $z=1$ and that there exist $\kappa>1$, $c>0$ such that 
$$|\varphi(e^{i\theta})|=1-c|\theta|^\kappa +o(|\theta|^\kappa).$$ 
Let $d\geq  k\geq 1$ and let us set
$$\phi(z)=(\varphi(z_1)\cdots \varphi(z_k)z_{k+1}\cdots z_{d},\dots,\varphi(z_1)\cdots \varphi(z_k)z_{k+1}\cdots z_{d},0,\dots,0)$$ where the first $q$ coordinates, $q\in\{1,\dots,d\}$, are identical. Then $C_\phi$ maps boundedly $A_{\beta_1}^2(\DD^d)$ into $A_{\beta_2}^2(\DD^d)$ if and only if 
$$ d\beta_2\geq \left(2q-d-1-\frac {\min(d-1,k)}\kappa\right)+q\beta_1 .$$
\end{theorem}
\begin{proof} 
 As $\m{\varphi(z)} = 1$ if and only if $z=1$, the only maximal $I$ for which $\phi_I(e^{i\theta_0})\in\TTI$ for some $\theta_0\in\RR^d$ is $I=\{1,\dots,q\}$. Rotating if necessary, we may assume $\varphi(1)=1$. As  $\varphi'(1)>0$ by the Julia-Carathéodory theorem, we have $P_{I,e} = \I{d}$. Thus, $\phi$ has the restricted variables property and we shall denote $P_I$ instead of $P_{I,e}$. We intend to apply Corollary \ref{cont_restricted} and Corollary \ref{discont_restricted} with $\m{I}  = q $ and $\m{P_I} = d.$
Since $\varphi(1)=1$, there exist $r_0>0$ and $\mathcal C^1$ functions $\psi,\chi:(-r_0,r_0)\to\mathbb R$
   with $\psi(0)=0$, $\psi'(0)=a=\varphi'(1)$ and $\chi(\theta)=-\log|\varphi(e^{i\theta})|=c|\theta|^\kappa+o(|\theta|^\kappa)$,
   such that $\varphi(e^{i\theta})=\exp(i\psi(\theta)-\chi(\theta))$.
We aim to construct a neighborhood $\mathcal O$ of $0$ in $\RR^d$ such that, for all $\delta\in(0,1)$ and all $j\in\{1,\dots,q\}$, 
$$\lambda_d(\{\theta\in\mathcal O:\ |\phi_j(e^{i\theta})-1|<\delta\})\simeq 
\left\{
\begin{array}{ll}
\delta^{1+\frac k\kappa}&\textrm{ if }k<d\\
\delta^{1+\frac {d-1}\kappa}&\textrm{ if }k=d.\\
\end{array}
\right.$$
Since there exists $C>0$ such that, for $|\alpha|$, $x\geq 0$ and $\delta>0$ small enough, 
$$
|e^{i\alpha-x}-1|<\delta\ \Longrightarrow\ |\alpha|\le C\delta,\ x\le C\delta,
\qquad |\alpha|\le\delta/C,\ x\le\delta/C\ \Longrightarrow\ |e^{i\alpha-x}-1|<\delta,
$$
this is clearly comparable to $\lambda_d(E_1(\delta))$ where 
$$E_1(\delta)=\{\theta\in\mathcal O:\ |\psi(\theta_1)+\cdots+\psi(\theta_k)+\theta_{k+1}+\cdots+\theta_d|<\delta,\ \chi(\theta_1)+\cdots+\chi(\theta_k)<\delta\}.$$
We consider the following change of variables:
\begin{align*}
U(\theta)&=(\theta_1,\dots,\theta_{d-1},\psi(\theta_1)+\cdots+\psi(\theta_k)+\theta_{k+1}+\cdots+\theta_d)\\
&=:(u_1,\dots,u_d).
\end{align*}
The Jacobian of $U$ at $0$ is equal to $1$ (if $k<d$) or to $a$ (if $k=d$) so that there exist two neighborhoods $\mathcal O$ and $\mathcal O'$ of $0$ in $\RR^d$ such that $\lambda_d(E_1(\delta))\simeq \lambda_d(E_2(\delta))$
where
$$E_2(\delta)=\{u\in\mathcal O':\ |u_d|<\delta,\ \chi(U_1^{-1}(u))+\cdots+\chi(U_k^{-1}(u))<\delta\}.$$
Shrinking $\mathcal O$ if necessary, we may assume that $\frac c2|t|^\kappa\leq\chi(t)\leq 2c|t|^\kappa$ for all
$t$ such that $|t|$ is less than the diameter of $\mathcal O'$.
Observe that $U_j^{-1}(u)=u_j$ for all $j=1,\dots,d-1$.
Therefore, for $k<d,$
$$\chi(U_1^{-1}(u))+\cdots+\chi(U_k^{-1}(u))=\chi(u_1)+\cdots+\chi(u_{k})$$
Therefore the bounds obtained on $\chi(t)$ above and the inequality $|u_d|<\delta$ yield
$\lambda_d(E_2(\delta))\simeq \delta^{1+\frac k\kappa}.$

For $k=d$, the proof is slightly more involved since we now have
$$\chi(U_1^{-1}(u))+\cdots+\chi(U_d^{-1}(u))=\chi(u_1)+\cdots+\chi(u_{d-1})+\chi(U_d^{-1}(u)).$$
Since $\chi(U_d^{-1}(u))\geq 0$, the same argument still yields 
$\lambda_d(E_2(\delta))\lesssim \delta^{1+\frac {d-1}\kappa}.$
Conversely, since $U^{-1}(0)=0$, there exists $M>0$ such that, for all $u\in\mathcal O',$ 
$$|U_d^{-1}(u)|\leq M(|u_1|+\cdots+|u_d|)$$
so that 
$$\chi(U_d^{-1}(u))\leq M' (|u_1|^\kappa+\cdots+|u_d|^\kappa)$$
for some $M'>0$. Hence it is easy to see that if we choose $|u_d|<\veps\delta$ and $|u_j|<\veps\delta^{1/\kappa}$ for some $\veps\in(0,1)$ sufficiently small (but independent of $\delta$), then $u\in E_2(\delta)$. Finally, $\lambda_d(E_2(\delta))\gtrsim \delta^{1+\frac {d-1}\kappa}.$
 Denoting $\ell=k$ if $k<d$ and $\ell=d-1$ if $k=d$, we may then apply Corollary \ref{discont_restricted} to get that $C_\phi$ is not bounded if 
$$\dfrac{1}{q}\left(1+\dfrac{\ell}{\kappa}\right) + \dfrac{d}{q}\left(1+\beta_2\right) < 2+\beta_1  \Leftrightarrow d\beta_2< \left(2q-d-1-\frac \ell\kappa\right)+q\beta_1.$$

Conversely, let us study the continuity of the operator.  Let $\eta \in \TT^q$  and  $\overline\delta \in (0,1)^q$. Let $j \in \I{q}$ and write $\eta_j=\exp(i \zeta_j)$.
We observe that $|\phi_j(e^{i\theta})-\eta_j|<\delta_j$ implies, in the new coordinates, up to some constant, that 
$$|u_d-\zeta_j|<\delta_j,\ \chi(U_1^{-1}(u))+\cdots+\chi(U_k^{-1}(u))<\delta_j.$$
Therefore, we still have
$$\lambda_d(\{\theta\in\mathcal O:\ |\phi_j(e^{i\theta})-\eta_j|<\delta_j\})\lesssim 
\delta_j^{1+\frac{\ell}\kappa}.$$
 Finally, we have found a neighbourhood $\mathcal O$ of $0$ in $\RR^d$ such that  $$\lambda_d(\lbrace \theta \in \mathcal O :\  \m{\phi_j(e^{i\theta})- \eta_j}<\delta_j,\ j \in I\rbrace) \lesssim \min_{j\in I}(\delta_j)^{1+\frac \ell\kappa}\leq \left(\prod_{j\in I}\delta_j\right)^{\frac 1q\left(1+\frac \ell\kappa\right)}.$$
 This has been done for $\theta_0=0$ but by rotational invariance, a similar estimate holds for all $\theta_0$  such that $\phi_I(\theta_0)\in\TT^q$ (these $\theta_0$ can be written $(0,\dots,0,\theta_{k+1},\dots,\theta_d)$). Now, by definition of $\phi$, we have that for all nonempty $J \subset I, \, \m{P_J} = d$. Thus, $\gamma_I = \min\limits_{\emptyset \neq J \subset I }\dfrac{\m{P_J}}{\m{J}} = \dfrac dq.$ 
\color{black} 
Applying Corollary \ref{cont_restricted}, we get that $C_\phi$ maps continuously $A^2_{\beta_1}(\DD^d)$ into $A^2_{\beta_2}(\DD^d)$ as soon as $$d\beta_2 \geq 2q-d-1-\frac{\ell}{\kappa}+q\beta_1.$$
\end{proof}

The tridisc case for $q=2$ and $k=1$ writes:

\begin{corollary}\label{coro}
Let $\varphi\in \mathcal O(\DD,\DD)\cap \mathcal C^1(\overline\DD)$. Assume that, for any $z\in \overline\DD,$ $|\varphi(z)|=1$  if and only if $z=1$ and that there exist $\kappa>1$, $c>0$ such that
$$|\varphi(e^{i\theta})|=1-c|\theta|^\kappa +o(|\theta|^\kappa).$$
We set $\phi(z)=(\varphi(z_1)z_2z_3,\varphi(z_1)z_2z_3,0).$ Then $C_\phi$ is bounded on $A_\beta^2(\DD^3)$ if and only if ${\beta\geq -1/\kappa.}$
\end{corollary}


\section{\texorpdfstring{$N$-th order derivatives are not enough}{N-th order derivatives are not enough}}
\label{sec:norder}
In this section, we show that the knowledge of the $N$-th order derivatives at the boundary points of the symbol is not sufficient to know if the induced composition operator is continuous on some $A^2_\beta(\DD^d)$. We first need to exhibit two functions which have the same first derivatives and which have a different behaviour. We fix $n\geq 1$ and we consider the two functions $g_n$ and $h_n$ defined by 
\begin{align*}
    g_n(z) &=  z^n-\dfrac{1}{2^n}\left( z- \left(\dfrac{1+z^2}{2}\right)\right)^n\\
    h_n(z) &= z^{2n}-\dfrac{z^n}{2^n}\left( z- \left(\dfrac{1+z^2}{2}\right)\right)^n.
\end{align*}

\noindent {\bf Fact 1.} $g_n(\overline{\DD}) \subset \overline{\DD},\ h_n(\overline{\DD})\subset \overline{\DD}$ and $\m{g_n(z)} = 1 \Leftrightarrow \m{h_n(z)}=1 \Leftrightarrow z=1$.

\begin{proof}
Let $\theta \in \mathbb{R}.$ We have 
\begin{align*}
    g_n(\eio) = e^{in\theta} - \dfrac{1}{2^n}\left(\eio - \dfrac{1+e^{2i\theta}}{2}\right)^n  = e^{in\theta}\left(1-\left(\dfrac{1-\cos(\theta)}{2}\right)^n\right). 
\end{align*}
Thus, $\m{g_n(\eio)} = \left\vert 1-\left(\dfrac{1-\cos(\theta)}{2}\right)^n\right\vert$. However, $0\leq \dfrac{1-\cos(\theta)}{2} \leq 1$ so that $\m{g_n(\eio)} \leq 1$. Moreover, 
\begin{align*}
    \m{g_n(\eio)} = 1 \Leftrightarrow \dfrac{1-\cos(\theta)}{2} = 0 \Leftrightarrow \theta = 0\ [2\pi].
\end{align*}
Finally, $g_n(\overline{\DD}) \subset \overline{\DD}$ and $\m{g_n(z)} = 1 \Leftrightarrow z=1$. We deduce the results for $h_n$ if we observe that $h_n(z)=z^n g_n(z).$
\end{proof}

\medskip

\noindent {\bf Fact 2.} For all $k\in \I{n}$, we have $g_{2n}^{(k)}(1) =h_n^{(k)}(1)=2n\cdots (2n-k+1)$. 

\begin{proof}
Setting $u_n(z)=\left(z-\left(\frac{1+z^2}2\right)\right)^n,$
this fact easily follows from Leibniz's formula if we are able to prove that $u_n^{(k)}(1)=0$ for all $k=0,\dots,n$. This last property can be proved by induction, using $u_{n+1}(z)=u_n(z)u_1(z).$ 
\end{proof}

\medskip

We now prove a variant of the theorem announced in the introduction where we fix $\beta$ and we allow the dimension to change. For a map $\phi:\overline{\mathbb D}^d\to\overline{\mathbb D}^d$ we denote by $\mathcal E(\phi)$ the set $\mathcal E(\phi):=\{z\in\overline{\DD}^d:\ \phi(z)\in\partial {\overline{\DD}^d}\}.$

\begin{theorem}\label{thm:nthderivativegeneral}
    Let $\beta \geq -1$ and $n\in\mathbb N$. There exist $d\geq 3,$ $\phi,\psi\in\mathcal O(\DD^d,\DD^d)\cap\mathcal C^\infty(\overline{\DD}^d)$ such that 
    \begin{itemize}
\item $\mathcal E(\phi)=\mathcal E(\psi)$.
\item For any $\xi\in\mathcal E(\phi),$ for any $\alpha\in\mathbb Z_+^d$ with $|\alpha|\leq n,$ $\partial^\alpha \phi(\xi)=\partial^\alpha \psi(\xi)$.
\item $C_\phi$ is continuous on $A_\beta^2(\DD^d)$ whereas $C_\psi$ is not.
\end{itemize}
\end{theorem}

\begin{proof}
Let us consider $(a,b)\in\mathbb N\times\mathbb Z$ such that 
$$0\leq \beta-\frac ba<\frac 1a$$
and $a>4n.$ We set $d=2a+b+3,$ $q=a+b+3,$ $k=4n$ and
\begin{align*}
    \phi(z)&=(h_{n}(z_1)\cdots h_{n}(z_k)z_{k+1}\cdots z_{d},\dots,h_{n}(z_1)\cdots h_{n}(z_k)z_{k+1}\cdots z_{d},0,\dots,0)\\
    \psi(z)&=(g_{2n}(z_1)\cdots g_{2n}(z_k)z_{k+1}\cdots z_{d},\dots,g_{2n}(z_1)\cdots g_{2n}(z_k)z_{k+1}\cdots z_{d},0,\dots,0)
\end{align*}
where the first $q$ coordinates are identical. Observe that $d>k$ since $b\geq -a.$
By Fact 1, $\phi$ and $\psi$ belong to $\mathcal O(\DD^d,\DD^d)\cap\mathcal C^\infty(\overline \DD^d)$ and $\mathcal E(\phi)=\mathcal E(\psi)=\{(1,\dots,1)\}\times \TT^{d-k}.$ Moreover, Fact 2 yields $\partial^\alpha\phi(\xi)  = \partial^\alpha\psi(\xi)$ for all $\m{\alpha}\leq n$ and all $\xi \in \mathcal E(\phi)$.\\
To discuss the continuity of $C_\phi$ and $C_\psi,$ we observe that
\begin{align*}
    \m{h_n(\eio)} = \left\vert 1-\left(\dfrac{1-\cos(\theta)}{2}\right)^n\right\vert = 1-\dfrac{\theta^{2n}}{4^{n}}+o(\theta^{2n})
\end{align*}
and 
\begin{align*}
    \m{g_{2n}(\eio)} = \left\vert 1-\left(\dfrac{1-\cos(\theta)}{2}\right)^{2n}\right\vert = 1-\dfrac{\theta^{4n}}{4^{2n}}+o(\theta^{4n}). 
\end{align*}
Now we get
$$(d-q)\beta -\left(2q-d-1-\frac{4n}{2n}\right)=a\beta-b\geq 0$$
 whereas 
$$(d-q)\beta-\left(2q-d-1-\frac{4n}{4n}\right)=a\beta-b-1< 0.$$
Therefore Theorem \ref{thm:examples} yields the continuity of $C_\phi$ on $A_\beta^2(\DD^d)$ and the discontinuity of $C_\psi$ on $A^2_\beta(\DD^d).$
\end{proof}

If we impose the dimension to be equal to $3$, then we can state the following theorem.

\begin{theorem}\label{thm:nthderivative}
    Let $n\geq 1$. There exist $\phi,\psi\in\mathcal O(\DD^3,\DD^3)\cap\mathcal C^\infty(\ovdt)$ and $\beta\in(-1,0)$ such that 
    \begin{itemize}
\item $\mathcal E(\phi)=\mathcal E(\psi)$.
\item For any $\xi\in\mathcal E(\phi),$ for any $\alpha\in\mathbb Z_+^3$ with $|\alpha|\leq n,$ $\partial^\alpha \phi(\xi)=\partial^\alpha \psi(\xi)$.
\item $C_\phi$ is continuous on $A_\beta^2(\DD^3)$ and $C_\psi$ is not continuous on $A_\beta^2(\DD^3)$.
\end{itemize}
\end{theorem}

\begin{proof}
The proof  follows exactly the same lines. We consider $\phi$ and $\psi$ defined  for all $z \in \DD^3$ by 
$$ \phi(z) = (h_n(z_1)z_2z_3,h_n(z_1)z_2z_3,0) \hspace{5mm} \text{ and } \hspace{5mm} \psi(z)=(g_{2n}(z_1)z_2z_3,g_{2n}(z_1)z_2z_3,0).$$ 

Applying Corollary  \ref{coro}, we get that $C_\phi$ is bounded on $\Ab$ if and only if $\beta \geq \dfrac{-1}{2n}$ and $C_\psi$ is bounded on $\Ab$ if and only if $\beta \geq \dfrac{-1}{4n}$. In particular, for $\beta \in \left[\dfrac{-1}{2n},\dfrac{-1}{4n}\right),$ we have that $C_\phi$ is bounded on $\Ab$ whereas $C_\psi$ is not. 
\end{proof}

\medskip

Theorems \ref{thm:nthderivative} and \ref{thm:nthderivativegeneral} lead to the following natural question.

\begin{question}
Let $d\geq 3$ and $n\in\mathbb N$. For which values of $\beta\geq -1$ do there exist
$\phi,\psi\in\mathcal O(\DD^d,\DD^d)\cap\mathcal C^n(\overline{\DD}^d)$ such that 
    \begin{itemize}
\item $\mathcal E(\phi)=\mathcal E(\psi)$,
\item for any $\xi\in\mathcal E(\phi),$ for any $\alpha\in\mathbb Z_+^d$ with $|\alpha|\leq n,$ $\partial^\alpha \phi(\xi)=\partial^\alpha \psi(\xi)$,
\item $C_\phi$ is continuous on $A_\beta^2(\DD^d)$ whereas $C_\psi$ is not ?
\end{itemize}
\end{question}

 We do not know the answer to this question but we can at least prove that this holds for  $$\beta \in \bigcup\limits_{q=1}^{d-1}\left[\dfrac{1}{d-q}\left(2q-1-d-\dfrac{d-1}{2n} \right) , \dfrac{1}{d-q}(2q-1-d)  \right)$$ \\
Indeed, for $n \in \NN$ and $p\geq 1$ we consider $$H_n(z) = z^{n(p+1)} - \dfrac{z^{pn}}{2^n}\left(z-\left(\dfrac{1+z^2}{2}\right)\right)^n =  z^{pn}g_n(z). $$ Then, as it was done for $h_n$, we can show that $\m{H_n(z)}\leq 1$ for all $z \in \DD$ and $\m{H_n(z)} = 1  $ if and only if $z=1.$ Moreover, in the same way as the proof of Fact 2,  we can show that $H_n^{(k)}(1) = g_{(p+1)n}^{(k)}(1) = (p+1)n\cdots ((p+1)n-k+1)$ for all $k \in \I{n}.$   Then, consider 
\begin{align*}
    \phi(z) &= (H_n(z_1)\cdots H_n(z_{d-1})z_d, \dots,H_n(z_1)\cdots H_n(z_{d-1})z_d,0,\dots,0) \\
    \psi(z)& =  (g_{(p+1)n}(z_1)\cdots g_{(p+1)n}(z_{d-1})z_d,\dots,g_{(p+1)n}(z_1)\cdots g_{(p+1)n}(z_{d-1})z_d,0,\dots,0)
\end{align*}
where the $q-$first coordinates are identical. Then, $\phi, \psi \in \mathcal{O}(\DD^d,\DD^d) \cap \mathcal{C}^n(\overline \DD^d)$ with $\mathcal{E}(\phi) = \mathcal E(\psi) $ and  $\partial^\alpha\phi(\xi) =  \partial^\alpha\psi(\xi)$ for all $\xi \in \mathcal E(\phi) $ and all $\alpha \in \ZZ_+$ such that $\m{\alpha} \leq n.$ To study the continuity of $C_\phi$ and $C_\psi$, observe that 
\begin{align*}
    \m{H_n(\eio)  } &= 1- \dfrac{\theta^{2n}}{4^n}+o(\theta^{2n}) \\
    \m{g_{(p+1)n}(\eio) } &= 1 - \dfrac{\theta^{2(p+1)n}}{4^{(p+1)n}} + o(\theta^{2(p+1)n}).
\end{align*}
Applying Theorem \ref{thm:examples}, we get that $C_\phi : \Abd \rightarrow \Abd $ is bounded if and only if 
$$\beta \geq \dfrac{1}{d-q}\left(2q-d-1-\dfrac{d-1}{2n}\right) $$
and $C_\psi : \Abd \rightarrow \Abd $ if and only if 
$$\beta \geq  \dfrac{1}{d-q}\left(2q-d-1-\dfrac{d-1}{2n(p+1)}\right). $$
Letting $p$ tend to $+\infty$, we find that for 
$\beta \in \bigcup\limits_{q=1}^{d-1}\left[\dfrac{1}{d-q}\left(2q-d-1-\dfrac{d-1}{2n}\right),\dfrac{1}{d-q}\left(2q-d-1\right) \right), $ there exist $\phi, \psi \in \mathcal{O}(\DD^d,\DD^d) \cap \mathcal{C}^n(\overline \DD^d)$ with $\mathcal{E}(\phi) = \mathcal E(\psi) $ and  $\partial^\alpha\phi(\xi) =  \partial^\alpha\psi(\xi)$ for all $\xi \in \mathcal E(\phi) $ and all $\alpha \in \ZZ_+$ with $\m{\alpha} \leq n $ such that $C_\phi$ is bounded on $\Abd$ while $C_\psi$ is not. 

\section{Two different weights}\label{sec:weights}

In this section, we shall prove the following theorem.
\begin{theorem}\label{thm:weights}
Let $d\geq 3$ and let $-1\leq \beta_1<\beta_2\leq d-3.$ Then there exist $\phi\in\mathcal O(\DD^d,\DD^d)\cap \mathcal C^1(\overline{\DD}^d)$ such that $C_\phi$ is continuous on $A^2_{\beta_2}(\DD^d)$ and $C_\phi$ is not continuous on $A^2_{\beta_1}(\DD^d).$
\end{theorem}
\begin{proof}
We proceed by induction on $d$. The initial case $d=3$ is covered by the first case below. 
Assume first that $\beta_2>d-4.$ Let $\kappa >1$ be such that $\beta_1 < d-3-\dfrac{1}{\kappa} <\beta_2$. We set $\alpha = \kappa -\lfloor \kappa \rfloor $. Let $f : \mathbb{R} \rightarrow [0,1/2]$ be such that
\begin{itemize}
\item $f$ is $\mathcal{C}^{\lfloor \kappa \rfloor, \alpha}$-smooth on $\RR$ and $2\pi$-periodic.
\item $f(\theta) = 0 \Leftrightarrow  \theta = 0\ [ 2\pi]$.
\item $f(\theta)  \sim_0 \m{\theta}^\kappa. $
\end{itemize}
Now, consider $\Gamma$ the Jordan curve defined by the following equation 
$$ r(\theta) = 1-f(\theta).$$
We have $\Gamma \subset \overline{\DD}.$ Applying Riemann's and Kellogg-Warschawski's theorems, we get that there exists a  $\mathcal{C}^1$ function  $\varphi : \overline{\DD} \rightarrow \overline{\DD}$  such that $\varphi : \DD \rightarrow\DD$ is holomorphic, $\varphi$ is a conformal map that maps $\DD$ onto $\Delta$, where $\Delta$ is the interior of $\Gamma$, and $|\varphi(z)|=1$ if and only if $z=1$. We normalize it by setting $\varphi(1)=1.$ For all $\theta \in \RR,$ we have $\varphi(\eio) \in \Gamma$ so $$1-\m{\varphi(\eio)} = f(\arg(\varphi(\eio))).$$
But, for $\theta$ close to $0,$ $\arg(\varphi(\eio)) = \arctan\left(\dfrac{\Imm(\varphi(\eio))}{\Ree(\varphi(\eio))}\right)$ and $$\varphi(\eio) = 1+\varphi'(1)(\eio-1)+o(\theta) = 1+i\varphi'(1)\theta + o(\theta).$$
Thus, $\arg(\varphi(\eio))  = \arctan(\varphi'(1)\theta+o(\theta)) = \varphi'(1)\theta+o(\theta).$ As $\varphi $ maps $\DD$ to itself, applying Julia-Carathéodory's theorem, we have $\varphi'(1)> 0. $  We deduce that 
$$1-\m{\varphi(\eio)} \sim_0 \varphi'(1)^\kappa\m{\theta}^\kappa  \Rightarrow \m{\varphi(\eio)} =1-\varphi'(1)^\kappa\m{\theta}^\kappa +o(\m{\theta}^\kappa). $$ 

Now, we consider $\phi : \DD^d \rightarrow \DD^d$ defined 
as in Theorem \ref{thm:examples} with $q=d-1,$ $k=1$ and this specific $\varphi.$ 
Since 
$$d\beta_2\geq 2q-d-1-\frac 1\kappa+q\beta_2$$
and 
$$d\beta_1< 2q-d-1-\frac 1\kappa+q\beta_1,$$
$C_\phi$ is the composition operator we were looking for. 

If $\beta_2\leq d-4,$ let $q\in[3,d-1]$ be such that $q-4<\beta_2\leq q-3$. By induction hypothesis, there exists $\psi\in \mathcal O(\DD^q,\DD^q)\cap \mathcal C^1(\overline\DD^q)$ such that $C_\psi$ is bounded on $A^2_{\beta_2}(\DD^q)$ and $C_\psi$ is not bounded on $A^2_{\beta_1}(\DD^q)$. We come back to $\DD^d$ by setting
$$\phi(z_1,\dots,z_d)=(\psi_1(z_1,\dots,z_q),\dots,\psi_q(z_1,\dots,z_q),0,\dots,0).$$
\end{proof}
In particular, for $-1\leq \beta_1<\beta_2,$ we can always distinguish continuity on $A_{\beta_1}^2(\DD^d)$ and $A_{\beta_2}^2(\DD^d)$ provided we allow the dimension to be large enough. It is interesting to know the lowest dimension $d$ for which this is possible. For $d=3,$ the bound $d-3$ on $\beta_2$ is optimal by Theorem \ref{thm:weighttridisc}. This leads to the following question.

\begin{question}
    Let $d\geq 4$ and let $\phi\in \mathcal O(\DD^d,\DD^d)\cap\mathcal C^1(\ovD)$. Is it true that if $C_\phi$ is continuous on $A_\beta^2(\DD^d)$ for some $\beta\geq d-3,$ then it is continuous on $A_\beta^2(\DD^d)$ for all $\beta\geq d-3$ ? 
\end{question}

We can observe that if $\beta_1<d-3\leq \beta_2,$ there exists a composition operator $C_\phi$ which is bounded on $A^2_{\beta_2}(\DD^d)$ and which is unbounded on $A_{\beta_1}^2(\DD^d)$ (see \cite[Corollary 6]{Ko22}), given by $\phi(z)=(z_1\cdots z_d,\dots,z_1\cdots z_d,0)$.


\section{About continuity on the tridisc} \label{sec:tridisc}
In this section, we  focus on the continuity of the composition operator $C_\phi$ on the Bergman space $\Ab$. We provide necessary and sufficient conditions for $C_\phi$ to be bounded on $\Ab$. We begin by recalling the terminology of \cite{Baytridisc}. Let $\phi \in \mathcal O(\mathbb{D}^3,\mathbb{D}^3) \cap  \mathcal{C}^3(\overline{\mathbb{D}}^3)$. Assume that there exist ${I \subset \lbrace 1,2,3 \rbrace}$ with $\vert I \vert = 2$, $I = \lbrace i_1,i_2 \rbrace$ and $\xi \in \mathbb{T}^3$ such that $\phi_I(\xi) \in \mathbb{T}^2$ and such that $\nabla\phi_{i_1}(\xi),$ $\nabla\phi_{i_2}(\xi)$ are linearly dependent. Observe that by the remark following Lemma \ref{JC} they are nonzero. \color{black} \\
We can assume that $\xi = e =  (1,1,1),\ I = \lbrace 1, 2 \rbrace$ and $\phi_I(e) = (1,1)$. For $z \in \mathbb{T}^3$, we write \\
$z = (e^{i\theta_1},e^{i\theta_2},e^{i\theta_3})$ with $\theta_i \in \mathbb{R}$, $i=1,2,3$. Thanks to a second-order Taylor expansion at $e$ for $\phi_j$, $j = 1,2$, we obtain: 
$$  \phi_j(z) = 1 + \displaystyle\sum_{k=1}^3 a_{j,k}(z_k - 1) + \displaystyle\sum_{k,l =1}^3 \beta_{j,k,l}(z_k-1)(z_l-1) + O(\vert z_1-1\vert^3 + \vert z_2-1\vert^3 + \vert z_3-1\vert^3).$$
By Julia-Carathéodory's theorem, we know  that $a_{j,k} \geq 0$ for $k =1,2, 3$ and $j= 1, 2$. At fixed $j$, all $a_{j,k}$ cannot be simultaneously equal to zero otherwise $\nabla\phi_j(\xi ) = 0 $ which is not the case. \\
Finally, by a Taylor expansion of $e^{i\tk}$,  we get: 
\begin{align*}
\text{Re}\phi_j(\theta) &=  1 - \dss_{k=1}^3 a_{j,k}\frac{\tk^2}{2}
- \dss_{k,l =1 }^3 \text{Re}(\beta_{j,k,l})\tk\tl +O\left(\vert\theta_1\vert^3+\vert\theta_2\vert^3+\vert\theta_3\vert^3\right)\\
&= 1 - Q_j(\theta) + O\left(\vert\theta_1\vert^3+\vert\theta_2\vert^3+\vert\theta_3\vert^3\right)
\end{align*}
where $Q_j$ is a quadratic form and 
\begin{align*}
\text{Im}\phi_j(\theta) &= \dss_{k=1}^3 a_{j,k}\tk
- \dss_{k,l =1 }^3 \text{Im}(\beta_{j,k,l})\tk\tl+ O\left(\vert\theta_1\vert^3+\vert\theta_2\vert^3+\vert\theta_3\vert^3\right).
\end{align*}
As $\nabla\phi_1(e)$ and $\nabla\phi_2(e)$ are linearly dependent and nonzero, we can write 
$$\text{Im}\phi_j(\theta)=
 \kappa_j L(\theta) - \dss_{k,l =1 }^3 \text{Im}(\beta_{j,k,l})\tk\tl+ O\left(\vert\theta_1\vert^3+\vert\theta_2\vert^3+\vert\theta_3\vert^3\right)
$$
where $L$ is a fixed nonzero linear form and $\kappa_j\in\mathbb R^*,$ $j=1,2.$
 Each $Q_j$ is a quadratic form of signature $(n_j,0)$ so it can be written $Q_j= \dss_{k = 1}^{n_j} L_{j,k}^2$ with ${sign(Q_1+Q_2) = (m,0)}$. Moreover, 
$$\text{span}(L_{j,k},\ j= 1,2, \ k= 1, \dots, n_j)^\perp =\text{ker}(Q_1+Q_2),$$ 
thus dim(span)$^\perp = 3-m$. We then consider $(L_1,...,L_m)$ a basis of $\text{span}(L_{j,k})$. The linear form $L$ belongs to span$(L_1,..., L_m)$. Indeed, if $L \notin \text{span}(L_{j,k}),$ then there exists $\theta \in \RR^3 $ such that $L(\theta) \neq 0$ but $Q_j(\theta) = 0$ for $j=1,2$. Therefore, for $\veps$ small enough, we have 
\begin{align*}
 \m{\phi_j(\veps\theta)}^2& = 1-2\veps^2Q_j(\theta)+ \kappa_j^2\veps^2 L(\theta)^2+ O(\veps^3) \\
 & = 1+\kappa_j^2\veps^2 L(\theta)^2+ O(\veps^3)\\
 &>1 ,  
\end{align*}
a contradiction. Then we complete $(L_1,\dots, L_m)$ to have $(L_1,L_2,L_3)$ a basis of $(\mathbb{R}^3)^*$ and  we can write 
$$ \text{Im}\phi_j = \kappa_j\left( L+ \dss_{\substack{1\leq k\leq m \\ k\leq l \leq 3}}b_{j,k,l}L_kL_l+\dss_{m+1\leq k\leq l \leq 3}b_{j,k,l}L_kL_l\right)+ O(\m{\tu}^3+\m{\td}^3+\m{\tr}^3).$$
We define the quadratic form $R$ by 
$$ R =\kappa_1\kappa_2 \left(\dss_{m+1\leq k\leq l \leq 3}b_{1,k,l}L_kL_l -  \dss_{m+1\leq k\leq l \leq 3}b_{2,k,l}L_kL_l\right).$$ 
Finally we set
\begin{center}
$s(\phi,I,\xi) = m \hspace{0.5mm} \text{ and } \hspace{0.5mm} r(\phi,I,\xi) = sign(R).$
\end{center}
It is shown in \cite{Baytridisc} that $r(\phi,I,\xi)$ does not depend on the way we define the linear forms $L_j.$
One can also notice that if we consider $\phi$ defined by $\phi = (\phi_1, \phi_1,0)$ then $R= 0$, thus $r(\phi,I,\xi) =(0,0)$.

\medskip

The notations $s(\phi,I,\xi)$ and $r(\phi,I,\xi)$ were introduced in \cite{Baytridisc} in order to give a characterization of bounded composition operators on $H^2(\DD^3)$. By Theorem \ref{thm:nthderivative} we know that second-order conditions on the derivative of $\phi$ are not enough to characterize the continuity of $C_\phi$ on all $A_\beta^2(\DD^3)$. Nevertheless, we investigate how their knowledge helps us to restrict the range of $\beta$ such that $C_\phi$ is continuous on $A^2_{\beta}(\DD^3)$ or not. 
Let us start by showing a preliminary result showing that in some cases their knowledge implies non-boundedness results.

\begin{theorem}\label{noncont}
Let $\phi \in \mathcal{O}(\DD^3,\DD^3) \cap \mathcal{C}^3(\ovdt)$. Assume that there exist $ I \subset \lbrace 1,2,3 \rbrace$, with $\m{I} = 2,$ $I=\lbrace i_1,i_2\rbrace$ and $ {\xi\in \TT^3}$ such that $ {\phi_I(\xi)\in \TT^2}$ and $\nabla\phi_{i_1}(\xi), \nabla\phi_{i_2}(\xi)$ are linearly dependent.
\begin{itemize}
\item If $\m{P_{I,\xi}} = 3$, then
\begin{enumerate}[(a)]
\item If $s(\phi, I,\xi) = 1$ and $r(\phi,I,\xi) = (0,0)$ then $C_\phi$ is not bounded on $\Ab$ for $\beta  < \dfrac{-2}{3}$,
\item  If $s(\phi, I,\xi) = 1$ and $r(\phi,I,\xi) = (1,0)$ or $(0,1)$ then $C_\phi$ is not bounded on $\Ab$ for ${\beta  < \dfrac{-5}{6}}$. 
\end{enumerate}
\item If $\m{P_{I,\xi}}  \leq 2$ then $C_\phi : \Ab \rightarrow \Ab $ is not bounded for any $\beta \geq -1$. 
\end{itemize}
\end{theorem}

\begin{proof}
Without loss of generality, we may assume that $I = \lbrace 1,2\rbrace$, $\xi=e$, $\phi_I(e)=(1,1)$ and,
since $\nabla\phi_1(e)$ and $\nabla\phi_2(e)$ are linearly dependent and nonzero,
$P_{1,e}=P_{2,e}=P_{I,e}$.

\smallskip
\emph{Case $|P_{I,e}|=3$.}
We intend to apply Corollary \ref{discont_restricted}, which is possible since $\phi$ has the restricted variables property at $\xi$ with respect to $I$.  Because it only depends on what happens on the unit polycircle, making linear changes of variables on the coordinates $(\tu,\td,\tr)$ will not change the volume estimates. Thus, in everything that follows, we will always assume that $L(\theta)  = \theta_1.$ First, suppose that $s(\phi,I,\xi) =1$ and $r(\phi,I,\xi) = (0,0)$. 
We can assume that ${Q_j(\theta) = \gamma_j\theta_1^2}$ with $\gamma_j>0$, $j = 1,2$ because $sign(Q_1(\theta)+Q_2(\theta)) =1 $ and since, $L(\theta) = \theta_1 $ and $s(\phi,I,\xi) =1$, we can choose $L_2(\theta) = \theta_2$ and $L_3(\theta) = \theta_3.$ 
 We write for $j=1,2,$
\begin{align*}
\text{Re}\phi_j(\theta) &= 1 - \gamma_j\tu^2+O\left(\vert\theta_1\vert^3+\vert\theta_2\vert^3+\vert\theta_3\vert^3\right)\\
\text{Im}\phi_j(\theta) &= \kappa_j\left( \theta_1-a_j\theta_2^2 -b_j\td\tr-c_j\tr^2\right)+O\left( \tu^2+ \vert \tu\td\vert+\vert \tu\tr\vert +\vert\theta_2\vert^3+\vert\theta_3\vert^3\right).
\end{align*}
However, as 
$$R(\td,\tr)  = \kappa_2\kappa_1\left(-a_1\theta_2^2 -b_1\td\tr-c_1\tr^2\right)- \kappa_1\kappa_2\left(-a_2\theta_2^2 -b_2\td\tr-c_2\tr^2\right) = 0$$
and $\kappa_2\kappa_1 \neq  0,$ we have 
$$a_1 = a_2 = a, \hspace{5mm} b_1 = b_2 = b \hspace{5mm} \text{ and } \hspace{5mm} c_1 = c_2 = c.$$
Thus,
\begin{align*}
&\text{Im}\phi_j(\theta) = \kappa_j (\tu -a\theta_2^2 -b\td\tr-c\tr^2)+O\left( \tu^2+ \vert \tu\td\vert+\vert \tu\tr\vert +\vert\theta_2\vert^3+\vert\theta_3\vert^3\right)\\
&\text{Re}\phi_j(\theta) = 1 - \gamma_j\tu^2+O\left(\vert\theta_1\vert^3+\vert\theta_2\vert^3+\vert\theta_3\vert^3\right).
\end{align*}
Let $\delta \in (0,1)$ and let us consider $(\td, \tr) \in [ -\delta^{1/3}, \delta^{1/3}]^2$. For those $(\td, \tr)$, we set $I_{\td,\tr} := \left\lbrace \tu :  \vert \tu - a\td^2- b\td\tr - c\tr^2\vert \leq \delta \right\rbrace.$ For $\tu \in I_{\td,\tr}$, we can easily check that $\vert \tu \vert \lesssim \delta^{2/3}$ and we have
\begin{align*}
&\vert \text{Im}\phi_j(\theta) \vert  \lesssim \delta + O(\delta^{4/3}+\delta +\delta+\delta+\delta)\lesssim \delta\\
&\vert\text{Re}\phi_j(\theta)-1\vert \lesssim \delta^{4/3}+O(\delta^2+\delta+\delta) \lesssim  \delta.
\end{align*}
Since
$$\lambda_3\left(\left\{\theta\in[-\pi,\pi)^3:\ (\theta_2,\theta_3)\in[-\delta^{1/3},\delta^{1/3}]^2,\ \theta_1\in  I_{\td,\tr}\right\}\right) \gtrsim \delta^{1+2/3}, $$
applying Corollary \ref{discont_restricted} with  $\m{I} = 2 $,  $\m{P_I} = 3$ and $\alpha = \dfrac 5 6,$ we get that $C_\phi$ is not bounded on $\Ab$ for $\beta < \dfrac{-2}{3} $. \\

\noindent Now, suppose that $s(\phi,I,\xi) =1$ and $r(\phi,I,\xi) = (1,0)$ or $(0,1)$. 
$R$ is a quadratic form with signature equal to $(1,0)$ or $(0,1)$ so, doing a linear change of variables, we may assume that $R(\td,\tr) = \pm\tr^2$. Thus, we write for $j=1,2,$ 
\begin{align*}
&\text{Re}\phi_j(\theta) = 1 - \gamma_j\tu^2+ O\left(\vert\theta_1\vert^3+\vert\theta_2\vert^3+ \m{\tr}^3\right)\\
&\text{Im}\phi_j(\theta) = \kappa_j(\theta_1-h_j\tu\td-l_j\tu\tr-a_j\theta_2^2-b_j\td\tr-c_j\tr^2) +O\left(\tu^2 +\vert\theta_2\vert^3+\vert\theta_3\vert^3\right).
\end{align*} 
 However, as 
$$R(\td,\tr) = \kappa_2\kappa_1\left(-a_1\theta_2^2 -b_1\td\tr-c_1\tr^2\right)- \kappa_1\kappa_2\left(-a_2\theta_2^2 -b_2\td\tr-c_2\tr^2\right) =\pm\tr^2$$ 
and $\kappa_2\kappa_1 \neq 0,$ we have 
$$a_1 = a_2 = a,  \hspace{5mm}  b_1 = b_2 = b \hspace{5mm} \text{ and } \hspace{5mm} c_1-c_2 \neq 0.$$
Thus,
\begin{align*}
&\text{Im}\phi_j(\theta) = \kappa_j(\theta_1-h_j\tu\td-l_j\tu\tr-a\theta_2^2-b\td\tr-c_j\tr^2) +O\left( \tu^2 +\vert\theta_2\vert^3+\vert\theta_3\vert^3\right)\\
&\text{Re}\phi_j(\theta) = 1 - \gamma_j\tu^2+O\left(\vert\theta_1\vert^3+\vert\theta_2\vert^3+\vert\theta_3\vert^3\right).
\end{align*}
Let $\delta \in (0,1)$. We aim to prove that there exist $C>0$ and $O$ a neighbourhood of $0$ in $\mathbb{R}^3$ such that
 $$\lambda_3(\lbrace \theta \in O: \m{\phi_j(\theta)-1}\leq \delta, \hspace{1mm}j=1,2\rbrace) \geq C\delta^{1 + 5/6}.$$
Let $U$ be defined on $\RR^3$ by $U(\theta) = \left(\dfrac{\text{Im}\phi_1}{\kappa_1}, \td,\tr\right)$. There exist $O_1$ and $O_2$ two neighbourhoods of $0$ in $\mathbb{R}^3$ and $C>0$ such that $U$ is a diffeomorphism from $O_1$ into $O_2$ and 
\begin{equation}\label{eq:du}
\left\{
\begin{array}{rcl}
|\det(dU)|&\leq&C\textrm{ on }O_1\\
|\det(dU^{-1})|&\leq&C\textrm{ on }O_2.
\end{array}
\right.
\end{equation}
As $dU(0) = I_3$, we have
\begin{align*}
&\text{Im}\phi_1 \circ U^{-1}(u) =\kappa_1u_1\\
&\text{Im}\phi_2 \circ U^{-1} (u)=\kappa_2(u_1+\widetilde{h_2}u_1u_2+\widetilde{l_2}u_1u_3+{c_2}u_3^2) +O(\m{u_1}^3+\m{u_2}^3+\m{u_3}^3)\\
& \Ree\phi_j\circ U^{-1}(u) = 1- \gamma_ju_1^2+O(\m{u_1}^3+\m{u_2}^3+\m{u_3}^3).
\end{align*}
Then, if $\m{u_1}\leq \delta,\ \m{u_2}\leq \delta^{1/3}$ and $\m{u_3}\leq \delta^{1/2}$, we have for $j=1,2,$
$$\m{\Imm\phi_j\circ U^{-1}(u)}\lesssim \delta \hspace{5mm} \text{ and } \hspace{5mm} \m{\Ree\phi_j\circ U^{-1}(u) - 1} \lesssim \delta.$$
Thus, because of \eqref{eq:du},
$$\lambda_3(\lbrace \theta \in O_1 : \m{\phi_j(\theta)-1}<\delta,\ j\in I\rbrace) \gtrsim  \lambda_3\left( [-\delta,\delta]\times [-\delta^{1/3},\delta^{1/3}]\times [-\delta^{1/2},\delta^{1/2}]\right)  \gtrsim \delta^{1+5/6} .$$
So, applying Corollary \ref{discont_restricted}, we get that $C_\phi$ is not bounded on $A^2_\beta(\DD^3)$ for $\beta < \dfrac{-5}{6} $. 

\smallskip

\emph{Case $|P_{I,e}|=2$}, say $P_{I,e}=\{1,2\}$. By Lemma \ref{JC}, there exist $a_1,a_2>0$ and
$\kappa_1,\kappa_2>0$ such that $\frac{\partial\phi_j}{\partial z_k}(e)=\kappa_ja_k$ for $k=1,2$, and,
since $\phi\in\mathcal C^2(\ovdt)$, Lemma \ref{JC} (d) yields $C>0$ such that, for all $z\in\ovdt$ and $j=1,2$,
$$\big|\phi_j(z)-1-\kappa_j\big(a_1(z_1-1)+a_2(z_2-1)\big)\big|\leq C\big(|z_1-1|^2+|z_2-1|^2\big).$$
For $\delta\in(0,1)$, let
$$A(\delta)=\big\{z=((1-\rho_k)e^{i\theta_k}):\ 0\leq\rho_1,\rho_2\leq\delta,\ |\theta_1|,|\theta_2|\leq\delta^{1/2},\
|a_1\theta_1+a_2\theta_2|\leq\delta\big\}.$$
For $z\in A(\delta)$ and $l=1,2$, $|z_l-1-i\theta_l|\leq\rho_l+\theta_l^2/2\leq 2\delta$ and
$|z_l-1|\leq 3\delta^{1/2}$, so that
$$|\phi_j(z)-1|\leq\kappa_j|a_1\theta_1+a_2\theta_2|+2\kappa_j(a_1+a_2)\delta+18C\delta\leq M\delta$$
for some $M>0$. Hence $A(\delta)\subset\phi^{-1}(S(e,\ovd))$ with $\ovd=(M\delta,M\delta,2)$. On the other hand,
by Lemma \ref{equiv},
$$V_\beta(A(\delta))\gtrsim\delta^{2(1+\beta)}\lambda_2\big(\{|\theta_1|,|\theta_2|\leq\delta^{1/2},\ |a_1\theta_1+a_2\theta_2|\leq\delta\}\big)
\gtrsim\delta^{2(1+\beta)+\frac32}\quad(\beta>-1),$$
and $\sigma_3(A(\delta)\cap\TT^3)\gtrsim\delta^{3/2}$. Since $V_\beta(S(e,\ovd))\simeq\delta^{2(2+\beta)}$
(also for $\beta=-1$), the necessary Carleson condition \eqref{eq:carleson} fails and $C_\phi$ is not bounded on $\Ab$.

\smallskip

\emph{Case $|P_{I,e}|=1$}, say $P_{I,e}=\{1\}$. By Lemma \ref{JC} (b) and the mean value inequality,
$|\phi_j(z)-1|=|\phi_j(z)-\phi_j(1,z_2,z_3)|\leq\|\nabla\phi_j\|_1|z_1-1|$, so that
$\{z:\ |z_1-1|<\delta/M\}\subset\phi^{-1}(S(e,(\delta,\delta,2)))$ for some $M>0$. The left-hand side has
$V_\beta$-measure $\simeq\delta^{2+\beta}$ (and $\sigma_3$-measure $\simeq\delta$ if $\beta=-1$), whereas
$V_\beta(S(e,(\delta,\delta,2)))\simeq\delta^{2(2+\beta)}$; hence $C_\phi$ is not bounded on $\Ab$.

\end{proof}

\color{black}

As another preliminary result, we show how the knowledge of $s$ and $r$ helps us to control the measures of some preimage sets.

\begin{theorem}\label{caspart}
Let $\phi \in \mathcal{O}(\DD^3,\DD^3) \cap \mathcal{C}^3( \ovdt)$. Let $\theta_0 \in \RR^3$ and $ I \subset \lbrace 1,2,3 \rbrace$ with  $\m{I} = 2$ and $ I=\lbrace i_1,i_2\rbrace$ be such that $\phi_I(\theta_0)\in \TT^2$ and $\nabla\phi_{i_1}(\theta_0), \nabla\phi_{i_2}(\theta_0)$ are linearly dependent. If $\m{P_{I,e^{i\theta_0}}}=3$, then there exist $C>0$ and $ O$ a neighbourhood of $\theta_0$ in $\RR^3$ such that for all $\eta \in \TT^2$ and all $\ovd \in (0,1)^2,$ we have 
$$\lambda_3 (\lbrace \theta \in O : \m{\phi_j(\theta)-\eta_j}<\delta_j, \ j\in I \rbrace) \leq C(\delta_{i_1}\delta_{i_2})^\alpha$$
with 
\begin{enumerate}[(a)]
\item $\alpha = \dfrac{3}{4}$ if $ s(\phi,I,e^{i\theta_0} )= 1, \hspace{1mm} r(\phi,I,e^{i\theta_0} )\in\{ (0,1),(1,0)\}$ or  $s(\phi,I,e^{i\theta_0} ) = 2, \hspace{1mm}  r(\phi,I,e^{i\theta_0} )  = (0,0)$ 
\item $0< \alpha <1$ if $ s(\phi,I,e^{i\theta_0} ) = 1, \hspace{1mm}r(\phi,I,e^{i\theta_0} )  = (1,1)$
\item $\alpha = \dfrac{1}{2}$ if $ s(\phi,I,e^{i\theta_0} )= 1, \hspace{1mm}r(\phi,I,e^{i\theta_0} ) = (0,0).$
\end{enumerate}

\end{theorem}

\begin{proof}
Without loss of generality, we can assume that $I = \lbrace 1,2\rbrace, \ \theta_0 =(0,0,0)$ and ${\phi_I(\theta_0) =(1,1).}$ First, let us consider the case $s(\phi,I,e)  = 1$ and $r(\phi,I,e) = (0,1)$ or $(1,0).$
We may argue as in the proof of Theorem \ref{noncont} to write, for $j=1,2,$
\begin{align*} 
&\text{Im}\phi_j(\theta) = \kappa_j(\tu-g_j\tu^2-h_j\tu\td-l_j\tu\tr-a\td^2-b\td\tr-c_j\tr^2)+O(\vert \tu\vert^3+\vert \td\vert^3+\vert \tr\vert^3)\\
&\text{Re}\phi_j(\theta) = 1-\gamma_j\tu^2 +O(\vert \tu\vert^3+\vert \td\vert^3+\vert \tr\vert^3),
\end{align*}
with $c_1\neq c_2.$
Let 
\begin{align*}
 f(\theta) &= \kappa_2\text{Im}\phi_1(\theta)-\kappa_1\text{Im}\phi_2(\theta) \\
 & = \alpha\tu^2+\beta\tu\td+\gamma\tu\tr +\kappa_1\kappa_2(c_2-c_1) \tr^2 +O(\vert \tu\vert^3+\vert \td\vert^3+\vert \tr\vert^3).
\end{align*}
If $\delta_1 \leq \delta_2$, let $U$ be defined on $\RR^3$ by $U(\theta) = \left(\dfrac{\text{Im}\phi_1}{\kappa_1}, \td,\tr\right)$, let $O_1$ and $O_2$ be two neighbourhoods of $0$ in $\mathbb{R}^3$ and let $C>0$ be such that $U$ is a diffeomorphism from $O_1$ into $O_2$ and \eqref{eq:du} is satisfied. As $dU(0) = I_3$,  we have 
\begin{align*}
\text{Im} \, \phi_1 \circ U^{-1}(u) &= \kappa_1 u_1 \\
f \circ U^{-1}(u) &= \widetilde{\alpha} u_1^2 + \widetilde{\beta} u_1 u_2 + \widetilde{\gamma} u_1 u_3 +\kappa_1\kappa_2(c_2-c_1) u_3^2 + O(\vert u_1 \vert^3 + \vert u_2 \vert^3 + \vert u_3 \vert^3).
\end{align*}
We observe that $f(0) = 0$, $d_{u_3}f(0) = 0$ and $d^2_{u_3}f(0) = 2\kappa_1\kappa_2(c_2-c_1) \neq 0$ as $c_1 \neq c_2.$ Hence, by the parametrized Morse's lemma, there exist a neighborhood $O_2' \subset O_2$ of $0$ in $\mathbb{R}^3$, a diffeomorphism ${V : O_2' \to V(O_2')} $ defined by  $ V(u_1, u_2, u_3) = (u_1, u_2, v_3(u_1, u_2, u_3))
$, a $\mathcal{C}^1-$map $h : \mathbb{R}^2 \to \mathbb{R}$ and $C_1>0$ such that
\[
f \circ (V \circ U)^{-1}(v) = \pm v_3^2 + h(v_1, v_2) \]
and
\begin{equation}\label{eq:dv}
\left\{
\begin{array}{rcl}
|\det(dV)|&\leq&C_1\textrm{ on }O'_2\\
|\det(dV^{-1})|&\leq&C_1\textrm{ on }V(O'_2).
\end{array}
\right.
\end{equation}

By the definition of $V$, we still have $\text{Im}\phi_1\circ (V\circ U)^{-1}(v) = \kappa_1 v_1.$
Let $ O = U^{-1}(O'_2)$. For $\theta \in O \cap \phi^{-1}_I(\Si)$, considering $v = V\circ U(\theta)$, we have 
$$\vert \kappa_1 v_1 -\Imm\eta_1 \vert \leq \delta_1 \hspace{1mm} \text{ and} \hspace{1mm } \vert \pm v_3^2 + h(v_1,v_2) - (\kappa_2\Imm\eta_1-\kappa_1\Imm\eta_2) \vert \leq \max(\delta_1,\delta_2) = \delta_2.$$ 
So, by \eqref{eq:du}, \eqref{eq:dv} and Fubini's theorem,
$$ \lambda_3(\phi^{-1}_I(\Si) \cap O ) \lesssim \lambda_3( V\circ U(O\cap \phi^{-1}_I(\Si))) \lesssim \delta_1\delta_2^{1/2}\lesssim (\delta_1\delta_2)^{3/4}.$$
If $\delta_2\leq \delta_1$, we work with ${U(\theta) = \left(\dfrac{\text{Im}\phi_2}{\kappa_2}, \td,\tr\right)}$ and in the same way, we prove that $$\lambda_3( V\circ U(O\cap \phi^{-1}_I(\Si))) \lesssim \delta_2\delta_1^{1/2}\lesssim (\delta_1\delta_2)^{3/4}.$$ 

\smallskip

Now, let us consider the case $s(\phi,I,e)  = 2$ and $r(\phi,I,e) = (0,0).$ We write 
$$\text{Re}\phi_j(\theta) = 1 -Q_j(\theta) + O(\m{\tu}^3+\m{\td}^3+\m{\tr}^3) \text{ and  } \text{Im}\phi_j(\theta) = \kappa_jL(\theta) + O(\m{\tu}^2+\m{\td}^2+\m{\tr}^2)$$
where $Q_j$ is a quadratic form and $L$ is a nonzero linear form. 

Let us consider $$f(\theta) = \text{Re}\phi_1(\theta)+\text{Re}\phi_2(\theta) = 2- Q(\theta) + O(\m{\tu}^3+\m{\td}^3+\m{\tr}^3), $$
with $Q$ a quadratic form of signature $(2,0)$. Doing a linear change of variables, we may assume 
\begin{align*}
\Imm\phi_1(\theta) = a_1\theta_1+a_2\theta_2+a_3\tr+O(\m{\theta_1}^2+\m{\theta_2}^2+\m{\theta_3}^2)\\
f(\theta) = 2 -\theta_1^2-\theta_2^2 +O(\m{\theta_1}^3+\m{\theta_2}^3+\m{\theta_3}^3)
\end{align*}
with $a_1\neq 0.$ If $\delta_1 \leq \delta_2$, let $U(\tu,\td,\tr) = (\Imm\phi_1(\theta),\td,\tr)$ and let $O_1,O_2$ and $C>0$ be as above. Doing this change of variables, we get 
$$\Imm\phi_1 \circ U^{-1}(u) = u_1 \text{ and } f\circ  U^{-1}(u) = 2- \widetilde{Q}(u_1,u_2)+O(\m{u_1}^3+\m{u_2}^3+\m{u_3}^3)$$ with $sign(\widetilde{Q}) = (2,0)$. In particular, $ \dfrac{\partial^2\widetilde{Q}(0,0)}{\partial u_2^2}>0$. Thus, applying Morse's lemma to $ f$ regarding $u_2$, we get that there exist $O_2', O_3$ two neighbourhoods of $0$ in $\RR^3$, a diffeomorphism $V : O'_2 \rightarrow O_3$ defined by ${V(u_1,u_2,u_3) = (u_1, v_2(u_1,u_2,u_3),u_3)}$, a $\mathcal{C}^1-$map $h : \mathbb{R}^2 \rightarrow \mathbb{R}$ and $C_1>0$ such that \eqref{eq:dv} is satisfied and
$$ \Imm\phi_1\circ ( V\circ U)^{-1}(v)= v_1 \text{ and } f \circ  (V\circ U)^{-1}(v) =2 -v_2^2+h(v_1,v_3).  $$
Thus, if $O = U^{-1}(O'_2)$ and $\theta \in \phi_I^{-1}(\Si)\cap O$, considering $v = V\circ U (\theta)$ we have 
$$\m{v_1 - \Imm\eta_1}\leq \delta_1 \text{ and } \m{v_2^2-(\Ree\eta_1+\Ree\eta_2)}\leq \max(\delta_1,\delta_2) = \delta_2.$$ Therefore, we have $\lambda_3( \phi_I^{-1}(\Si)\cap O) \lesssim \delta_1\delta_2^{1/2} \lesssim (\delta_1\delta_2)^{3/4}$.\\
If $\delta_2 \leq \delta_1$, we replace $U$ by $U'$ defined by ${U'(\theta_1,\theta_2,\theta_3) = (\Imm\phi_2(\theta),\theta_2,\theta_3)}$ and similarly, we get
 $$\lambda_3(\phi_I^{-1}(\Si)\cap O)  \lesssim (\delta_1\delta_2)^{3/4}.$$

\smallskip

Now, let us consider the case $s(\phi,I,e) =1$ and $r(\phi,I,e) = (1,1)$. We write for $j=1,2,$
\begin{align*}
\text{Im}\phi_j(\theta)& = \kappa_j(\tu-g_j\tu^2-h_j\tu\td-l_j\tu\tr-a_j\td^2-b_j\td\tr-c_j\tr^2)+ O\left(\m{\tu}^3+\m{\td}^3+\m{\tr}^3\right)\\
\text{Re}\phi_j(\theta) & =1 -\gamma_j\tu^2 + O\left(\m{\tu}^3+\m{\td}^3+\m{\tr}^3\right).
\end{align*}
Since 
$$R(\theta)=\kappa_1\kappa_2\left((a_2-a_1)\theta_2^2+(b_2-b_1)\theta_2\theta_3+(c_2-c_1)\theta_3^2\right)$$
has signature $(1,1)$, we may assume, doing a linear change of variables, that it is given by $R(\theta)=\kappa_1\kappa_2(\theta_2^2-\theta_3^2)$. Hence, if we set $f=\kappa_2\text{Im}(\phi_1)-\kappa_1\text{Im}(\phi_2)$, we get that we may write 
$$
f (\theta)  =  \alpha \theta_1^2+{\beta}\theta_1\theta_2+{\gamma}\theta_1\theta_3+\kappa_1\kappa_2(\theta_2^2-\theta_3^2)+O(\vert \theta_1\vert^3+\vert \theta_2\vert^3+\vert \theta_3\vert^3).
$$
If $\delta_1 \leq \delta_2,$ we do the change of variables $U(\theta) = \left( \dfrac{\text{Im}\phi_1}{\kappa_1}, \theta_2, \theta_3\right)$ where $U : O_1 \rightarrow O_2$ is a diffeomorphism, $O_1$ and $O_2$ are two neighbourhoods of $0$ in $\RR^3$ and $C>0$ is such that \eqref{eq:du} is satisfied. Then, we have 
\begin{align*}
\text{Im}\phi_1  \circ U^{-1}(u) &= \kappa_1 u_1   \\
 f  \circ  U^{-1}(u) &= \widetilde{\alpha} u_1^2+\widetilde{\beta} u_1u_2+\widetilde{\gamma} u_1u_3+\kappa_1\kappa_2(u_2^2-u_3^2)+O(\vert u_1\vert^3+\vert u_2\vert^3+\vert u_3\vert^3).
\end{align*}
We check that $f(0)=0$, $d_{u_2,u_3}f(0) = (0,0)$ and $d^2_{u_2,u_3}f(0) = \left( \begin{array}{cc}
2\kappa_1\kappa_2 & 0 \\ 
0 & -2\kappa_1\kappa_2
\end{array} \right)$ with \\
${\det(d_{u_2,u_3}f(0)) = -4(\kappa_1\kappa_2)^2 \neq 0}$. Therefore, by Morse's lemma, there exist $O_2', O_3$ two neighbourhoods of $0$ in $\RR^3$,  a diffeomorphism $V : O'_2 \rightarrow O_3$ defined by 
$$V(u_1,u_2,u_3) = (u_1, v_2(u_1,u_2,u_3),v_3(u_1,u_2,u_3)), $$
a  $\mathcal{C}^1-$map $h : \mathbb{R} \rightarrow \mathbb{R}$ and $C_1>0$ such that $f \circ (V\circ U)^{-1}(v) = v_2^2-v_3^2 +h(v_1)$ and \eqref{eq:dv} is satisfied.
By the definition of $V$, we still have $\text{Im}\phi_1 \circ (V \circ U)^{-1}(v) = \kappa_1v_1.$ We set $O =  U^{-1}(O'_2)$. Then, for $\theta \in O\cap \phi_I^{-1}(\Si)$, considering $v= V\circ U(\theta)$, we get
$$\m{\kappa_1v_1 - \Imm\eta_1  }\leq \delta_1 \hspace*{5mm} \text{ and} \hspace{5mm} \m{v_2^2-v_3^2+h(v_1)-(\kappa_2\Imm\eta_1-\kappa_1\Imm\eta_2)}\leq \max(\delta_1,\delta_2)= \delta_2 .$$ Then, by Lemma \ref{techniq}, we deduce
\begin{align*}
    \lambda_3(\lbrace \theta \in O :\  \m{\kappa_1v_1 - \Imm\eta_1 }\leq \delta_1,\ \m{v_2^2-v_3^2+h(v_1)-(\kappa_2\Imm\eta_1-\kappa_1\Imm\eta_2)}\leq \delta_2\rbrace)
    &\leq -C\delta_1\delta_2\log(\delta_2)\\
    & \lesssim (\delta_1\delta_2)^{\alpha},
\end{align*} 
for all $0<\alpha <1. $
If $\delta_2\leq \delta_1$, we work with $U'(\theta)  =  \left( \dfrac{\text{Im}\phi_2}{\kappa_2}, \td, \tr\right)$ and we get a similar estimate.

\smallskip

Finally, let us consider the case $s(\phi,I,e) =1 $ and $r(\phi, I, e)= (0,0).$ Arguing as in the proof of Theorem \ref{noncont}, we write for $j=1,2,$
\begin{align*}
\text{Im}\phi_j(\theta)& = \kappa_j(\tu-g_j\tu^2-h_j\tu\td-l_j\tu\tr-a\td^2-b\td\tr-c\tr^2)+ O\left(\m{\tu}^3+\m{\td}^3+\m{\tr}^3\right)\\
\text{Re}\phi_j(\theta) & =1 -\gamma_j\tu^2 + O\left(\m{\tu}^3+\m{\td}^3+\m{\tr}^3\right).
\end{align*}
If $\delta_1\leq \delta_2,$ we do the change of variables $U(\theta) = \left( \dfrac{\Imm\phi_1}{\kappa_1},\td,\tr\right) $  where $U : O_1 \rightarrow O_2$ is a diffeomorphism with $O_1, O_2$ two neighbourhoods of $0$ in $\RR^3$ and $C>0$ is such that \eqref{eq:du} is satisfied. We then have 
$$\Imm\phi_1 \circ U^{-1}(u) = \kappa_1u_1.$$
For $\theta \in O_1 \cap \phi_I^{-1}(\Si),$ considering $u = U(\theta)$, we have $\m{\kappa_1u_1-\Imm\eta_1 } \leq\delta_1$. Thus, 
$$\lambda_3( U(O_1\cap \phi^{-1}_I(\Si))) \lesssim \delta_1 \leq(\delta_1\delta_2)^{1/2}. $$
If $\delta_1 \geq \delta_2, $ we work with $U'(\theta) = \left( \dfrac{\Imm\phi_2}{\kappa_2},\td,\tr\right)$ and we also get
$$ \lambda_3(\lbrace \theta \in O_1 : \m{\kappa_2u_2-\Imm\eta_2}\leq \delta_2\rbrace) \lesssim (\delta_1\delta_2)^{1/2}. $$ 
\end{proof}

\begin{remark}
The last case also follows from Lemma \ref{lem:meanvalue}.
\end{remark}

Before stating the main result of this section, we need one more technical lemma that links the fact that the derivatives $\nabla\phi_{i_k}(\xi)$ are independent for $i_k \in I $ such that $\phi_I(\xi) \in \TTI$  and the measure of the sets $V_\beta(\Ei).$

\begin{lemma}\label{majovbeta}
Let $\phi \in \mathcal{O}(\DD^d,\DD^d)\cap \mathcal C^1(\ovD).$  Let $\beta\geq  -1$. Let $\xi \in \ovD,\  I=\lbrace i_1,\dots,i_p \rbrace \subset \I{d},\ I \neq \varnothing$ such that $\phi_I(\xi) \in \TTI $ and $\nabla\phi_{i_1}(\xi),\dots,\nabla\phi_{i_p}(\xi) $ are linearly independent . Then there exists $\mathcal U$ a neighbourhood of $\xi$ in $\ovD$ and $C>0 $ such that for all $\eta \in \TTI$ and all $\ovd \in \oii$ we have 
$$ V_\beta(\Ei \cap \mathcal{U}) \leq C\prod_{j\in I } \delta_j^{2+\beta}.$$
\end{lemma}
\begin{proof}
Let $\xi$ be as in our assumptions. By Lemma $12$ of \cite{Ko22}, we get that there exist $C>0$ and $\mathcal{U} $ a neighbourhood of $\xi$ in $\ovD$ such that for all $\ovd\in(0,1)^{|I|}$, 
$$ V_\beta(\phi_I^{-1}(\Si) \cap \mathcal U) \leq C\dprod_{j\in I} \delta_j^{2+\beta}. $$ 
By point a) of Lemma \ref{lem:doubleinclusion}, there exists $D>0$ such that $\Ei \subset \phi_I^{-1}(S_I(\eta, D\ovd)). $ Thus 
\begin{align*}
V_\beta(\Ei \cap \mathcal{U }) \leq V_\beta(\phi_I^{-1}(S_I(\eta, D\ovd))\cap \mathcal U) \leq C\dprod_{j\in I} \delta_j^{2+\beta}. 
\end{align*}
\end{proof}
\color{black}
Now, let us present the result of boundedness on $\Ab$. 

\begin{theorem}\label{general}
Let $\phi \in \mathcal{O}(\DD^3,\DD^3)\cap \mathcal{C}^3(\overline{\DD}^3)$ with $\phi(\overline{\DD}^3)\cap\partial \DD^3\neq\varnothing$. For $\xi \in \TT^3$, we define $J_{c}(\phi,\xi)$ and $J_{d}(\phi,\xi)$ in the following way:
\begin{enumerate}[(1)]
    \item $\phi(\xi) \in \TT^3$
    \begin{enumerate}
        \item If $d\phi(\xi)$ is invertible, then $J_{c}(\phi,\xi) = [-1,+\infty)$ and $J_{d}(\phi,\xi) = \varnothing,$
        \item If $d\phi(\xi)$ is not invertible, then $J_{c}(\phi,\xi) = \varnothing $ and $J_{d}(\phi,\xi) = [-1,+\infty).$
    \end{enumerate}
    \item $\phi(\xi) \in \TT^2\times \DD \cup \TT \times \DD \times \TT \cup \DD \times \TT^2. $ Let $1\leq i_1<i_2\leq 3$ be such that $(\phi_{i_1}(\xi),\phi_{i_2}(\xi))\in \TT^2$ and $I = \lbrace i_1,i_2 \rbrace $. 
    \begin{enumerate}
            \item If $(\nabla\phi_{i_1}(\xi),\nabla\phi_{i_2}(\xi))$ are linearly independent, then $J_{c}(\phi,\xi) = [-1,+\infty) $ and ${J_{d}(\phi,\xi) = \varnothing,}$
            \item If $(\nabla\phi_{i_1}(\xi),\nabla\phi_{i_2}(\xi))$ are linearly dependent and $\m{P_{I,\xi}} =3 $ then:
            \begin{enumerate}
                \item[$(\alpha)$]  if $s(\phi,I,\xi) = 1$ and $r(\phi,I,\xi) = (0,0)$, then $J_{c}(\phi,\xi) =[0,+\infty)$ and $J_{d}(\phi,\xi) = \left[-1,\frac{-2}{3}\right),$
                \item[$(\beta)$] if $s(\phi,I,\xi) = 1$ and $r(\phi,I,\xi) = (1,0)$ or $(0,1)$, then $J_{c}(\phi,\xi) = \left[\frac{-1}{2},+\infty\right)$ and $J_{d}(\phi,\xi) = \left[-1,\frac{-5}{6}\right)$,
                \item[$(\gamma)$] if $s(\phi,I,\xi) = 1$ and $r(\phi,I,\xi) = (1,1)$, then $J_{c}(\phi,\xi) =(-1,+\infty)$ and ${J_{d}(\phi,\xi) =\lbrace -1 \rbrace,}$
                \item[$(\delta)$] if $s(\phi,I,\xi) = 2$ and $r(\phi,I,\xi) = (0,0)$, then $J_{c}(\phi,\xi) = \left[\frac{-1}{2},+\infty \right)$ and ${J_{d}(\phi,\xi) =\lbrace -1 \rbrace,}$
                \item[$(\epsilon)$] if $s(\phi,I,\xi) = 3$ or $s(\phi,I,\xi) = 2$ and $r(\phi,I,\xi) \in \lbrace (0,1), (1,0)\rbrace$ or $s(\phi,I,\xi) = 1$ and $r(\phi,I,\xi) \in \lbrace (0,2), (2,0) \rbrace$, then $J_{c}(\phi,\xi) =[-1,+\infty)$   and $J_{d}(\phi,\xi) = \varnothing.$  
            \end{enumerate}
            \item  If $(\nabla\phi_{i_1}(\xi),\nabla\phi_{i_2}(\xi))$ are linearly dependent and $\m{P_{I,\xi}} \leq 2, $ then $J_{c}(\phi,\xi) = \varnothing$ and $J_{d}(\phi,\xi)=[-1,+\infty). $      
        \end{enumerate}
    \item $\phi(\xi) \in \TT \times\DD^2 \cup \DD\times\TT\times\DD\cup\DD^2\times\TT$, then $J_{c}(\phi,\xi) =[-1,+\infty)$ and $J_{d}(\phi,\xi) = \varnothing$.    
    \end{enumerate}
We define $J_{cont}(\phi)  =\bigcap\limits_{\xi \in \TT^3,\ \phi(\xi)\in\partial \DD^3} J_{c}(\phi,\xi) $ and  $J_{discont}(\phi) = \bigcup\limits_{\xi \in \TT^3,\ \phi(\xi)\in\partial \DD^3} J_{d}(\phi,\xi) $. Then, for $\beta \in J_{cont}(\phi)$,  $C_\phi$ is bounded on $\Ab$ and for $\beta \in J_{discont}(\phi)$, $C_\phi$ is not bounded on $\Ab$. 
\end{theorem}

\begin{proof}
Let $\beta\in J_{cont}(\phi)$. If $\beta=-1,$ the result is contained in \cite{Baytridisc} and therefore we may assume $\beta>-1.$ We intend to apply Theorem \ref{result_gen}, and to do that, we need to show that
\eqref{eq:carlesonbord} holds for all $\xi\in\overline{\DD}^3$ and all $I$ maximal such that $\phi_I(\xi)\in\TTI$.
We first consider $\xi=e^{i\theta_0}\in\TT^3$. There are several cases to consider. The first case is $\phi(\xi) \in \TT^3$, namely $I=\{1,2,3\}$. Applying Lemma \ref{majovbeta}, we get that there exist $C>0$ and a neighbourhood $\mathcal{U}$ of $\xi$ such that for all $\eta\in \TT^3$ and all $\ovd \in (0,1)^3,$ 
$$V_\beta(\Ei \cap \mathcal{U })\leq C \prod\limits_{j=1}^3\delta_j^{2+\beta}. $$
The second case is $\phi(\xi) \in \TT^2\times \DD \cup \TT \times \DD \times \TT \cup \DD \times \TT^2.$ There exists $I = \lbrace i_1,i_2 \rbrace $ such that $\phi_I(\xi) \in \TT^2$. If $\nabla\phi_{i_1}(\xi), \nabla\phi_{i_2}(\xi)$ are linearly independent then applying Lemma \ref{majovbeta}, there exist $\mathcal{U}$ a neighbourhood of $\xi$ and $C>0$ such that for all $\eta \in \TTI$ and all $\ovd \in (0,2)^{\m{I}},$  
$$V_\beta(\Ei \cap \mathcal{U}) \leq C \prod\limits_{j\in I}\delta_j^{2+\beta}. $$
Now, if $\nabla\phi_{i_1}(\xi), \nabla\phi_{i_2}(\xi)$
are linearly dependent and $\m{P_{I,\xi}} = 3$, we know that $\phi$ has the restricted variables property at $\xi$ and $\gamma_{I,\xi}=\frac{|P_{I,\xi}|}{|I|}=3/2.$  We distinguish several cases. Firstly, if $s(\phi,I,\xi) = 1 $ and $r(\phi,I,\xi) = (0,0)$ or $r(\phi,I,\xi) = (1,0) $ or $(0,1) $ or $r(\phi,I,\xi) = (1,1) $ or  if  $s(\phi,I,\xi) = 2$ and $r(\phi,I,\xi) = (0,0),$ then we apply Theorem \ref{caspart} to get $O$ a neighbourhood of $\theta_0$, $C>0$ and $\alpha\in\RR$ such that for all $\eta \in \TTI $ and all $\ovd \in (0,1)^{\m{I}}$, 
$$\lambda_3(\lbrace \theta \in O : \m{\phi_j(\theta)-\eta_j}<\delta_j,\ j\in I \rbrace) \leq C( \delta_{i_1}\delta_{i_2})^\alpha.$$
Moreover the value of $\alpha$ given by Theorem \ref{caspart} together with the choice of $ \beta \in J_{cont}(\phi) \subset J_c(\phi,\xi)$ ensures that $\alpha + \gamma_{I,\xi}(1+\beta) \geq 2+\beta$.

Secondly, if $s(\phi,I,\xi) =3$ and $r(\phi,I,\xi) = (0,0) $  or $s(\phi,I,\xi) = 2$ and $r(\phi,I,\xi) = (1,0) $ or $(0,1)$ or $s(\phi,I,\xi) = 1$ and $r(\phi,I,\xi) = (2,0) $ or $(0,2),$ then, following the proof of the fact of Theorem $5.1$ of \cite{Baytridisc}, there exist $O$ a neighbourhood of $\theta_0$ in $\RR^3$ and $C>0$ such that for all $\eta \in \TTI $ and  all $\ovd \in (0,2)^{\m{I}}, $
$$\lambda_3(\lbrace \theta \in O : \m{\phi_j(\theta)-\eta_j  } <\delta_j,\  j \in I \rbrace) \leq C \delta_{i_1}\delta_{i_2}. $$ 
Observe again that $1+\gamma_{I,\xi}(1+\beta) \geq 2+\beta.$ 

Then we may now use \eqref{eq:wikalpha} to deduce that 
$$\lambda_3(\lbrace \theta \in O : \m{\phi_j(\theta)-\eta_j}<\delta_j,\ j\in I \rbrace) \leq C \dfrac{\displaystyle\prod_{j\in I}\delta_j^{2+\beta}}{\dprod_{k=1}^3 \wikxi^{1+\beta}}.$$
Now, by Lemma \ref{lem:comparaisontwoCarlesonsets}, we deduce that there exist $\mathcal{U} $ a neighbourhood of $\xi$ in $\ovdt$ and $C >0$ such that 
$$ V_\beta(\Ei \cap \mathcal{U}) \leq C\dprod_{j\in I} \delta_j^{2+\beta}.$$

Finally, if $\phi(\xi) \in  \DD^2\times \TT \cup \DD \times \TT\times \DD\cup \TT \times \DD^2$, then there exists $i_1 \in \lbrace 1,2,3 \rbrace$ such that $\phi_{i_1}(\xi) \in \TT$. As $\nabla\phi_{i_1}(\xi) \neq 0$
, applying Lemma \ref{majovbeta}, we get that there exist $\mathcal{U}$ a neighbourhood of $\xi$ and $C>0$ such that for all $\eta \in \TT$ and all $\delta \in (0,1) $, 
$$V_\beta(E_{\{i_1\}} \cap \mathcal U ) \leq C\delta^{2+\beta}. $$ 

Let now $\xi\in\overline{\DD}^3\backslash\TT^3$ and let $I$ be maximal such that $\phi_I(\xi)\in\TTI$. For $j\in I$,
$\xi$ and $\hat\xi$ only differ on coordinates which do not belong to $P_{j,\xi}$, so that Lemma \ref{JC} (b) and (c)
yield $\phi_I(\hat\xi)=\phi_I(\xi)$ and $\nabla\phi_j(\hat\xi)=\nabla\phi_j(\xi)$. Moreover $|P_{I, \xi}|\leq 2$ since
$|\xi_k|=1$ for $k\in P_{I,\xi}$. If the vectors $\nabla\phi_j(\xi)$, $j\in I$, were linearly dependent, then
either $\phi(\hat\xi)\in\TT^3$ and $d\phi(\hat\xi)$ is not invertible, or $\hat\xi$ falls into Case (2)(c);
in both cases $J_c(\phi,\hat\xi)=\varnothing$, which contradicts $\beta\in J_{cont}(\phi)$. Hence these vectors are
linearly independent (as in Cases (1)(a), (2)(a) and (3)), and Lemma \ref{majovbeta}, applied at $\xi$, gives a neighbourhood $\mathcal U$ of $\xi$ and
$C>0$ such that $V_\beta(\Ei\cap\mathcal U)\leq C\dprod_{j\in I}\delta_j^{2+\beta}$.

Thus, we have shown that for all $\xi\in\overline{\DD}^3$ and all $I\subset\{1,2,3\}$ such that $\phi_I(\xi)\in\TTI$
and $I$ is maximal with respect to this property, there exist a neighbourhood $\mathcal U$ of $\xi$ in
$\overline{\DD}^3$ and $C>0$ such that, for all $\eta\in\TTI$ and all $\ovd\in(0,1)^{|I|}$,
$$V_\beta(\Ei\cap\mathcal U)\leq C\prod_{j\in I}\delta_j^{2+\beta}.$$
Applying Theorem \ref{result_gen}, we get that $C_\phi:\Ab\rightarrow\Ab$ is bounded.
\color{black}
\smallskip

Now, let $\beta \in J_{discont}(\phi)$. There exists $\xi \in \TT^3$ such that $\beta \in J_{d}(\phi, \xi)$. There are several cases to consider. First, if $\phi(\xi) \in \TT^3$, then $d\phi(\xi) $ is not invertible and the necessity from Theorem $1$ of \cite{KSZ08} gives that $C_\phi : \Ab \rightarrow \Ab$ is not bounded. Now, if $\phi(\xi)  \in \TT^2\times \DD \cup \TT \times \DD \times \TT \cup \DD \times \TT^2,$ let $I = \lbrace i_1,i_2\rbrace $ be such that $\phi_I(\xi) \in \TTI$ and $\nabla\phi_{i_1}(\xi), \nabla\phi_{i_2}(\xi) $ are linearly dependent. Suppose first that $\m{P_{I,\xi}} = 3. $ If $s(\phi,I,\xi)=1 $ and $r(\phi,I,\xi) =  (0,0) $ or  $r(\phi,I,\xi) =  (1,0) $ or $r(\phi,I,\xi) =  (0,1),$ then applying Theorem \ref{noncont}, we get that $C_\phi : \Ab \rightarrow \Ab$ is not bounded. If $s(\phi,I,\xi) = 1$ and $r(\phi,I,\xi) = (1,1)$ or if $s(\phi,I,\xi)=2$ and $r(\phi,I,\xi)=(0,0)$, then applying Theorem $4.1$ of \cite{Baytridisc}, we get that $C_\phi : \Ab \rightarrow \Ab$ is not bounded. 
Now, if $\m{P_{I,\xi}} \leq 2$, Theorem \ref{noncont} shows that $C_\phi : \Ab \rightarrow \Ab$ is never bounded.  
\end{proof} 
\color{black}


\section{Examples} \label{sec:examples}
In Theorem \ref{general} the sets $J_c(\phi,\xi)$ and $J_d(\phi,\xi)$ have to be distinct. One can ask whether they are optimal (that is, the largest possible). We are going to give examples of the cases $2)b)\alpha), \beta), \gamma), \delta)$ and $\veps)$ of Theorem \ref{general} and to discuss the optimality of $J_c(\phi,\cdot)$ and $J_d(\phi,\cdot)$ in these cases. 

\medskip

\subsection{Case 1: $s(\phi,I,\xi) = 1$ and $r(\phi,I,\xi) = (0,0)$.}

\begin{example}
Let $\phi : \DD^3 \rightarrow \DD^3 $ be defined by $\phi(z_1,z_2,z_3) = (z_1z_2z_3, z_1z_2z_3, 0).$  $C_\phi$ is bounded on $A^2_\beta(\DD^3) $ if and only if $\beta \geq 0$.
\end{example}
This follows from Proposition $5$ of \cite{Ko22}. The set of contact points of $\phi$ is $\TT^3$ and it is easy to show that for all contact points $\xi$, $P_{1,\xi}=P_{2,\xi}=\{1,2,3\},$ that $\nabla\phi_1(\xi)$ and $\nabla\phi_2(\xi)$ are not independent, and that $s(\phi,I,\xi)=1$ whereas $r(\phi,I,\xi)=(0,0)$ with $I=\{1,2\}$. In particular, the interval $[0,+\infty)$ is optimal for $J_c(\phi,\xi)$ in this case.

\medskip

Let now $\veps\in(0,1)$, let $c\in(0,2^{-4-\veps}]$ and let
$f(\theta)=c\,|e^{i\theta}-1|^{3+\veps}$, $\theta\in\RR$. Then $f$ is $2\pi$-periodic, $\mathcal C^{3,\veps}$-smooth, takes
its values in $[0,\frac12]$ and vanishes only at $\theta\equiv 0\ [2\pi]$. Let $\Gamma$ be the Jordan curve
$\theta\mapsto(1-f(\theta))e^{i\theta}$, which is $\mathcal C^{3,\veps}$-smooth, contained in $\overline\DD$, and satisfies
$\Gamma\cap\TT=\{1\}$. Let $\varphi$ be the Riemann map from $\DD$ onto the interior of $\Gamma$ such that $\varphi(1)=1$.
By the Kellogg--Warschawski theorem, $\varphi\in\mathcal C^{3}(\overline\DD)$. Moreover $\varphi(\DD)\subset\DD$,
$|\varphi(z)|=1$ if and only if $z=1$, and $a:=\varphi'(1)>0$ by the Julia--Carathéodory theorem. Since
$\varphi(e^{i\theta})\in\Gamma$ and $\arg\varphi(e^{i\theta})=a\theta+o(\theta)$, we get, arguing as in Section \ref{sec:weights},
$$|\varphi(e^{i\theta})|=1-f\big(\arg\varphi(e^{i\theta})\big)=1-c\,a^{3+\veps}|\theta|^{3+\veps}+o(|\theta|^{3+\veps}).$$

\begin{example} Let  
\begin{center}
\begin{tabular}{ccccc}
&$\phi : $& $\DD^3$& $\rightarrow$& $\DD^3$ \\
&&$(z_1,z_2,z_3)$&$\mapsto$ & $(\varphi(z_1)\varphi(z_2)\varphi(z_3),\varphi(z_1)\varphi(z_2)\varphi(z_3),0).$
\end{tabular}
\end{center}
Then $C_\phi$ is continuous on $A^2_\beta(\DD^3)$ if and only if $\beta\geq-\frac{2}{3+\veps}.$
\end{example}
\begin{proof}
The only contact point of $\phi_I$, $I=\{1,2\}$, is $e$ and we have
$P_{1,e}=P_{2,e}=\{1,2,3\}.$  Let us check that we are in Case 1. Near $0$, we may write
$\varphi(e^{i\theta})=\exp(i\psi(\theta)-\chi(\theta))$ with $\psi\in\mathcal C^3$, $\psi(0)=0$, $\psi'(0)=a$, and
$\chi(\theta)=-\log|\varphi(e^{i\theta})|=O(|\theta|^{3+\veps})$. Setting $L(\theta)=\theta_1+\theta_2+\theta_3$, we get, for $j=1,2$,
\begin{align*}
\Ree\phi_j(\theta)&=1-\frac{a^2}2L(\theta)^2+O(|\theta_1|^3+|\theta_2|^3+|\theta_3|^3),\\
\Imm\phi_j(\theta)&=aL(\theta)+\frac{\psi''(0)}2(\theta_1^2+\theta_2^2+\theta_3^2)+O(|\theta_1|^3+|\theta_2|^3+|\theta_3|^3).
\end{align*}
Hence $s(\phi,I,e)=1$ and, since $\phi_1=\phi_2$, $R=0$ so that $r(\phi,I,e)=(0,0)$.
Finally, Theorem \ref{thm:examples} applied with $d=k=3$, $q=2$ and $\kappa=3+\veps$ shows that $C_\phi$ is bounded on $\Ab$
if and only if
$$3\beta\geq\left(4-3-1-\frac{2}{3+\veps}\right)+2\beta\iff\beta\geq-\frac{2}{3+\veps}.$$
\end{proof}
\color{black}

It should be observed that these previous examples illustrate that the knowledge of $s$ and $r$ is not sufficient enough to study the continuity on any $\Ab$, contrary to what happens for the Hardy space. We have also proved that in the case $s(\phi,I,\xi)=1$ and $r(\phi,I,\xi)=(0,0)$, the best (=largest) value for $J_d(\phi,\xi)$ that we can expect is $[-1,-2/3].$ Hence, Theorem \ref{general} is almost optimal in Case 1: the only remaining question is whether $-\frac23$
belongs to the optimal set $J_d(\phi,\xi)$.

\subsection{Case 2: $s(\phi,I,\xi) =1$ and $r(\phi,I,\xi) = (1,0)$ or $(0,1)$.\\}

\smallskip

For $\veps\geq 0,$  we define the map $F_\veps$ on $\CC$ by
$$F_\veps(z)=\frac{3+6z-z^2}8+2i\veps(z-1)^2-i\veps(z-1)^3.$$
It is shown in \cite{Baytridisc} that there exists $\veps_1>0$ such that $F_\veps\in\mathcal O(\DD,\DD)$ provided $\veps\in[0,\veps_1)$ and $\m{F_\veps(z)} = 1 $ if and only if $z=1$.
\begin{example}
Let
\begin{center}
\begin{tabular}{ccccc}
&$\phi : $& $\DD^3$& $\rightarrow$& $\DD^3$ \\
&&$(z_1,z_2,z_3)$&$\mapsto$ & $(F_0(z_1)F_b(z_2)F_0(z_3),F_0(z_1)F_b(z_2)F_c(z_3),0)$
\end{tabular}
\end{center}
with $b,c\in(0,\veps_1)$. Then $C_\phi$ is bounded on $A_\beta^2(\DD^3)$ if and only if $\beta\geq -3/4$.
\end{example}
\begin{proof}
We write 
$c_j = \left\lbrace\begin{array}{ll}
0 \text{ if } j = 1\\
c \text{ if } j = 2 
\end{array}\right. .$  Doing a third-order Taylor expansion  and then doing the linear change of variables  ${u_1 = \tu+\td+\tr}$, $u_2= \td$ and $u_3 = \tr$, we get 
\begin{align*}
\text{Re}\phi_j\circ U^{-1}(u) &= 1-\dfrac{1}{8}u_1^2+ bu_2^2u_1+c_ju_3^2u_1+O(u_1^4+u_2^4+u_3^4)\\
\text{Im}\phi_j\circ U^{-1}(u) &= \dfrac{1}{2}u_1-2(bu_2^2+c_ju_3^2)+\dfrac{1}{24}u_1^3-\dfrac{3}{16}(u_1^2u_2+u_1^2u_3-u_1u_2^2-u_1u_3^2+u_2^2u_3^2)\\
&\quad+\dfrac{3}{8}u_1u_2u_3+O(u_1^4+u_2^4+u_3^4). 
\end{align*}
Hence, with $I=\{1,2\}$ again, $s(\phi,I,e) = 1$ and $R(u_2,u_3) = (-2bu_2^2)-(-2bu_2^2-2cu_3^2) = {2}cu_3^2$ so $r(\phi,I,e) =(1,0)$ and we are indeed in Case $2$.\\
We consider $u_2 \in [-\delta^{1/4},\delta^{1/4}]$ and  $u_3 \in [-\delta^{1/2},\delta^{1/2}]$. For those $u_2, u_3$, we choose \\
$u_1 \in I_{u_2,u_3} := \left\lbrace u_1 : \left\vert \dfrac{u_1}{2}-2bu_2^2-\dfrac{3}{16}(u_2^2u_3+u_3^2u_2)\right\vert \leq \delta \right\rbrace$. It is easy to check that $\vert u_1\vert \lesssim \delta^{1/2}$. So, 
\begin{align*}
\vert \text{Re}\phi_1 \circ U^{-1}(u) - 1\vert &\lesssim \delta+\delta+ O\left(\delta^{2}+\delta+\delta^{2}\right)\lesssim \delta\\
\vert \text{Re}\phi_2\circ U^{-1}(u) - 1\vert &\lesssim \delta+\delta
+\delta^{3/2}+ O(\delta^{2}+\delta+\delta^{2})\lesssim \delta\\
\vert \text{Im}\phi_1\circ U^{-1}(u) \vert &\lesssim \delta+\delta^{3/2}+\delta^{1+1/4}+\delta^{1+1/2}+\delta+\delta^{1+1/2}+\delta^{1+1/2+1/4}+ O(\delta)\lesssim \delta\\
\vert \text{Im}\phi_2 \circ U^{-1}(u) \vert &\lesssim \delta+\delta+O(\delta)\lesssim \delta.\\
\end{align*}
Thus, 
    $$\lambda_3\left(\lbrace u\in [-\pi, \pi)^3 : \m{\phi_j(u)-1}\lesssim \delta,\ j=1,2\rbrace\right) \gtrsim \delta^{1+3/4}.$$
Applying Corollary \ref{discont_restricted}, we get that $C_\phi$ is not bounded on $A^2_\beta(\DD^3) $ for $\beta < \dfrac{-3}{4}$. \\
Now, about the continuity. The only maximal $I$ for which $\phi_I(e^{i\theta_0})\in\TTI$ for some $\theta_0\in\RR^d$ is $I=\{1,2\}$ and this only holds for $\theta_0 =(0,0,0)$ (see Lemma $3.12$ of \cite{Baytridisc}). As before, we will be able to conclude if we prove that there exist $C>0$ and a neighbourhood $O$ of $0$ in $\mathbb{R}^3$ such that for all $\bar{\delta}\in (0,1)^2$ and all $\eta \in \mathbb{T}^2$, we have 
$$\lambda_3(\lbrace \theta \in O : \vert \phi_j(\theta)-\eta_j\vert \leq \delta_j,\ j=1,2\rbrace ) \leq C(\delta_1\delta_2)^{7/8}.$$
Note that 
\begin{align*}
\vert \phi_j(z) - \eta_j\vert < \delta_j \Rightarrow \left\lbrace\begin{array}{rclll}
\m{F_0(z_1)}^2&\geq 1-2\delta_j\\
\m{F_b(z_2)}^2&\geq 1-2\delta_j\\
\m{F_c(z_3)}^2&\geq 1-2\delta_j
\end{array}\right. .
\end{align*}
But, doing a fourth-order Taylor expansion of $\m{F_0}^2, \m{F_b}^2$  and $\m{F_c}^2$ gives 
\begin{align*}
\m{F_0(e^{i\tu})}^2 &= 1- \dfrac{3}{64}\tu^4 + O(\m{\tu}^5)\\
 \m{F_b(e^{i\td})}^2 &= 1+\left(4b^2- \dfrac{3}{64}\right)\td^4 + O(\m{\td}^5)\\
\m{F_c(e^{i\tr})}^2 &= 1+\left(4c^2- \dfrac{3}{64}\right)\tr^4 + O(\m{\tr}^5) .  
\end{align*}
So, $\m{\theta_i} \lesssim \delta_j^{1/4}$ if $\m{\phi_j(\theta)-\eta_j} \leq \delta_j$. Then, we define $f(\theta)  =\kappa_2\Imm\phi_1(\theta)-\kappa_1\Imm\phi_2(\theta)$. If $\delta_1 \leq \delta_2$,  we argue as above by considering the diffeomorphism $U : O_1 \rightarrow O_2$ defined by $U(\theta) = \left( \dfrac{\Imm\phi_1}{\kappa_1}, \td,\tr\right)$. As $dU(0) = I_3$, we get 
\begin{align*}
f\circ U^{-1}(u) &= cu_3^2+ O(\m{u_1}^3+\m{u_2}^3+\m{u_3}^3).
\end{align*} 
As $c\neq 0$, $\dfrac{\partial^2f(0)}{\partial u_3^2} \neq 0$ and applying Morse's lemma to $f$ regarding $u_3$, we get  $V : O_2'\subset O_2 \rightarrow O_3$  a diffeomorphism defined by ${V(u_1,u_2,u_3) = (u_1,u_2,v_3(u_1,u_2,u_3))}$ and a  $\mathcal C^1-$map $h : \mathbb{R}^2 \rightarrow \mathbb{R}$  such that $f \circ (V \circ U)^{-1}(v) = cv_3^2 + h(v_1,v_2)$. Now, set $O = U^{-1}(O'_2)$. For $\theta \in O \cap \phi^{-1}_I(\Si)$, considering $v = V\circ U(\theta)$, we have
$$\vert \kappa_1 v_1 -\Imm\eta_1 \vert \leq \delta_1,\ \vert v_3^2 \pm h(v_1,v_2) - (\kappa_2\Imm\eta_1-\kappa_1\Imm\eta_2) \vert \leq \delta_2 \hspace{1mm} \text{ and} \hspace{1mm } \m{v_2}\lesssim \delta_1^{1/4}.$$ 
So, 
$$ \lambda_3( V\circ U(O\cap \phi^{-1}_I(\Si))) 
\lesssim \delta_1^{5/4}\delta_2^{1/2} 
\leq (\delta_1\delta_2)^{7/8}.$$ 
If $\delta_2\leq \delta_1$, we replace $U$ by $U'$ defined by $U'(\theta) = \left( \dfrac{\Imm\phi_2}{\kappa_2}, \td,\tr\right)$ to get the same result. 
\end{proof}

In particular, in this case, one cannot expect a better set than $[-3/4,+\infty)$ for $J_c(\phi,\cdot)$ and a better set than $[-1,-3/4)$ for $J_d(\phi,\cdot)$.

\subsection{Case 3: $s(\phi,I,\xi ) =1 $ and $r(\phi,I,\xi) = (1,1). $}

\begin{example}
Let \begin{center}
\begin{tabular}{ccccc}
&$\phi : $& $\DD^3$& $\rightarrow$& $\DD^3$ \\
&&$(z_1,z_2,z_3)$&$\mapsto$ & $(F_0(z_1)F_b(z_2)F_0(z_3),F_0(z_1)F_0(z_2)F_c(z_3),0)$
\end{tabular}
\end{center}
with $b,c \in(0,\veps_1)$. Then $C_\phi : \Ab\rightarrow \Ab$ is bounded if and only if $\beta >-1. $
\end{example}
\begin{proof}
By Theorem \ref{general}, we just have to verify that we are in Case 3. The only contact point is $e.$
Doing the change of variables $u_1 = \tu+\td+\tr, \ u_2 =\td,$ and $ u_3 = \tr, $ we get: 
\begin{align*}
    \Ree\phi_j \circ U^{-1}(u)&= 1-\dfrac{1}{8}u_1^2+O(\m{u_1}^3+\m{u_2}^3+\m{u_3}^3)\\
    \Imm\phi_1\circ U^{-1}(u) &= \dfrac 12 u_1 -2bu_2^2 +O(\m{u_1}^3+\m{u_2}^3+\m{u_3}^3)\\
    \Imm\phi_2\circ U^{-1}(u) & = \dfrac12 u_1- 2cu_3^2 +O(\m{u_1}^3+\m{u_2}^3+\m{u_3}^3). 
\end{align*}
Thus, $s(\phi,I,e) = 1 $ and  $R(u_2,u_3) = -2bu_2^2+2cu_3^2 $ so $r(\phi,I,e) = (1,1).$ 
\end{proof}

\medskip

\subsection{Case 4: $s(\phi, I,\xi) = 2$ and $r(\phi,I,\xi) = (0,0)$.}

\begin{example}
Let $\phi : \begin{array}{ccc}
\DD^3 & \rightarrow & \DD^3\\
(z_1,z_2,z_3) & \longmapsto & \left(\dfrac{1}{2}\left(z_1z_2+\dfrac{1+z_3}{2}\right),\dfrac{1}{2}\left(z_1z_2+\dfrac{1+z_3}{2}\right),0\right)
\end{array}$.
Then $C_\phi$ is bounded on $\Ab$ if and only if $\beta \geq\dfrac{-1}{2}$.
\end{example}
\begin{proof}
The set of contact points is $\TT^2\times\{1\}$. By rotational invariance, we may choose $\xi=e.$
 Let $\delta\in (0,1)$. We start by doing the change of variables ${u_1 = \tu+\td},\ {u_2 = \td}$ and $u_3 = \tr$. We then have 
$$\phi(e^{iu}) = \left(\dfrac{1}{2}\left(e^{iu_1}+\dfrac{1+e^{iu_3}}{2}\right),\dfrac{1}{2}\left(e^{iu_1}+\dfrac{1+e^{iu_3}}{2}\right),0\right).$$
By doing a third-order Taylor expansion and the linear change of variables $v_1 = u_1+u_3/2$ and $v_3 = u_3$, we get: 
\begin{align*}
&\Ree\phi_j\circ (V\circ U)^{-1}(v) = 1 - \dfrac{1}{4}\left(v_1-\dfrac{1}{2}v_3\right)^2-\dfrac{2}{16}v_3^2 +O(v_1^4+v_3^4)\\
&\Imm\phi_j\circ (V\circ U)^{-1}(v) = \dfrac{1}{2}v_1+\dfrac{1}{8}v_1^2v_3-\dfrac{1}{16}v_1v_3^2-\dfrac{1}{32}v_3^3+O(v_1^4+v_3^4).
\end{align*}
Hence, $s(\phi,I,e)=2$ and $R(v_2,v_3) = 0$ so $r(\phi,I,e)=(0,0)$ and we are indeed in Case $4$.\\
Considering $\m{v_1}\leq \delta$ and $\m{v_3}\leq \delta^{1/2}$, we have 
\begin{align*}
\m{\Ree\phi_j\circ   (V\circ U)^{-1}(v)-1} &\lesssim \delta^2+\delta^{3/2}+\delta+O(\delta^4+\delta^2)\lesssim \delta\\
\m{\Imm\phi_j\circ   (V\circ U)^{-1}(v)}&\lesssim \delta+\delta^{2+1/2}+\delta^{2}+\delta^{3/2}+O(\delta^4+\delta^2)\lesssim\delta
\end{align*}
Thus, 
$$[-\delta,\delta]\times [-\pi,\pi)\times [-\delta^{1/2},\delta^{1/2}] \subset \lbrace u\in [-\pi, \pi)^3 : \m{\phi_j(u)-1}\lesssim \delta,\ j=1,2\rbrace$$
\text{ and } 
    $$  \lambda_3\left([-\delta,\delta]\times [-\pi,\pi)\times [-\delta^{1/2},\delta^{1/2}] \right) \gtrsim \delta^{1+1/2}.$$
Then, applying Corollary \ref{discont_restricted}, we get that $C_\phi$ is not bounded for $\beta <\dfrac{-1}{2}$. Finally, by Theorem \ref{general}, $C_\phi$ is bounded on $\Ab$ if and only if $\beta \geq\dfrac{-1}{2}$. 
\end{proof}
In particular, the interval $[-\frac{1}{2}, +\infty) $ is optimal for $J_c(\phi,\cdot)$ in this case
and we cannot expect something better than $[-1,-1/2)$ for $J_d(\phi,\cdot).$

\subsection{Case 5: $s(\phi, I, \xi) = 3  $ and $r(\phi,I,\xi) = (0,0). $}

\begin{example}
    Let $\phi : \begin{array}{ccc}
\DD^3 & \rightarrow & \DD^3\\
(z_1,z_2,z_3) & \longmapsto & \left(\dfrac{z_1+z_2+z_3}{3},\dfrac{z_1+z_2+z_3}{3},0\right)
\end{array}.$  
Then $C_\phi : \Ab \rightarrow \Ab$ is bounded for all $\beta \geq -1. $
\end{example}
\begin{proof}
    The only contact point is $e$.
Doing a Taylor expansion at $e$ gives 
$$\phi_j(\theta) = 1+\dfrac i 3 (\tu+\td+\tr)- \dfrac 1 6(\tu^2+\td^2+\tr^2) +O(\m{\tu}^3+ \m{\td}^3+\m{\tr}^3). $$
Thus, $s(\phi, I, e)= 3 $ and $r(\phi, I, e) = (0,0) $ and, we are in fact in Case $5$. An application of Theorem \ref{general} yields the result.
\end{proof}

\medskip

\subsection{Case 6: $s(\phi, I, \xi) = 2 $ and $r(\phi, I, \xi) = (0,1) $ or $(1,0)$. }

\begin{example} Let $\veps>0$ be sufficiently small and let
$$\phi : \begin{array}{ccc}
\DD^3 & \rightarrow & \DD^3\\
(z_1,z_2,z_3) & \longmapsto & \left(\dfrac{1}{2}\left(z_1z_2+\dfrac{1+z_3}{2}+i\veps(z_3-1)^3\right),\dfrac{1}{2}\left(z_1z_2+\dfrac{1+z_3}{2}\right),0\right)
\end{array}.$$
Then $C_\phi : \Ab \rightarrow \Ab$ is bounded for any $\beta \geq -1.$
\end{example}
\begin{proof}
 By \cite[Example 7.2]{Baytridisc}, we know that $\phi(\DD^3) \subset \DD^3$ and that the set of contact points is $\TT^2\times\{1\}$. Again we may assume $\xi=e.$ Doing a Taylor expansion at $e$ we obtain
 \begin{align*}
\Imm\phi_1(\theta) &= \dfrac 1 4 (\theta_1+\td+\tr) -\veps \dfrac{\tr^2}{2}+O(\m{\tu}^3 +\m{\td}^3+\m{\tr}^3)\\
\Imm\phi_2(\theta) &= \dfrac 1 4 (\theta_1+\td+\tr) +O(\m{\tu}^3 +\m{\td}^3+\m{\tr}^3)\\
\Ree\phi_1(\theta) &=1- \dfrac 1 4 (\theta_1+\td)^2 -\dfrac 1  8 \tr^2+O(\m{\tu}^3 +\m{\td}^3+\m{\tr}^3)\\
 \end{align*}
Thus, $s(\phi, I, e) = 2 $ and $R(\td,\tr) = -2\veps\tr^2$ so $r(\phi, I,e) = (0,1) $ or $(1,0)$. By Theorem \ref{general}, $C_\phi : \Ab \rightarrow \Ab$ is bounded for any $\beta \geq -1.$
\end{proof}

\medskip

\subsection{Case 7:  $s(\phi,I, \xi) = 1$ and $r(\phi, I, \xi ) = (0,2)$ or $(2,0).$ }

\begin{example}
    Let us consider \begin{center}
\begin{tabular}{ccccc}
&$\phi : $& $\DD^3$& $\rightarrow$& $\DD^3$ \\
&&$(z_1,z_2,z_3)$&$\mapsto$ & $(F_0(z_1)F_a(z_2)F_a(z_3),F_0(z_1)F_0(z_2)F_0(z_3),0)$
\end{tabular}
\end{center}
with $a \neq 0$ small enough.
Then $C_\phi : \Ab\rightarrow \Ab$ is bounded for any $\beta \geq-1. $
\end{example}
\begin{proof}
The only contact point is $\xi=e$.
Doing the change of variables $u_1 = \tu+\td+\tr,\ u_2 =\td,$ and $ u_3 = \tr $ we get: 
\begin{align*}
    \Ree\phi_j \circ U^{-1}(u)&= 1-\dfrac{1}{8}u_1^2+O(\m{u_1}^3+\m{u_2}^3+\m{u_3}^3)\\
    \Imm\phi_1\circ U^{-1}(u) &= \dfrac 12 u_1 -2a(u_2^2+u_3^2) +O(\m{u_1}^3+\m{u_2}^3+\m{u_3}^3)\\
    \Imm\phi_2\circ U^{-1}(u) & = \dfrac12 u_1 +O(\m{u_1}^3+\m{u_2}^3+\m{u_3}^3). 
\end{align*}
Thus, $s(\phi,I,e) = 1 $ and  $R(u_2,u_3) = -4a(u_2^2+u_3^2) $ so $r(\phi,I, e) = (0,2)$ or $(2,0)$. We are in Case 7. Applying Theorem \ref{general}, we get that $C_\phi : \Ab\rightarrow \Ab$ is bounded for any $\beta \geq-1. $
    
\end{proof}

\subsection{Examples of symbols without the restricted variables property}
We finish this section of examples by giving several examples of symbols which do not satisfy the restricted variables property.
\begin{example}
Let 
$$\phi : \begin{array}{ccc}
     \DD^3 &\rightarrow & \DD^3 \\
     (z_1,z_2,z_3) & \mapsto & \left(\dfrac{(1+z_1)^2+z_2(1-z_1)^2}{4},z_1z_2z_3 ,0 \right).  
\end{array}$$
Then $\phi\in\mathcal O(\DD^3,\DD^3)$, it does not satisfy the restricted variables property and $C_\phi$ is bounded on $\Ab$ for any $\beta\geq -1.$
\end{example}
\begin{proof}
The function $\phi$ maps $\DD^3 $ into itself since $$\left\vert\dfrac{(1+z_1)^2+z_2(1-z_1)^2}{4}\right\vert \leq \dfrac{|1+z_1|^2+|1-z_1|^2}{4} \leq \dfrac{1+|z_1|^2}{2}<1. $$
The set of contact points is $C =  \{1\} \times \TT^2 \cup \{-1\} \times \TT^2\cup \TT \times \{-1\} \times \TT .$ Let $I=\{1,2\} $ and $\xi\in C.$
Then $\frac{\partial\phi_1}{\partial z_3}(\xi)=0$, $\frac{\partial \phi_2}{\partial z_3}(\xi)\neq 0$ and $\nabla\phi_1(\xi)\neq 0$, so that $\nabla\phi_1(\xi) $ and $\nabla\phi_2(\xi)$ are linearly independent. Applying Theorem \ref{general}, we get that $C_\phi$ is bounded on $\Ab$ for any $\beta\geq -1.$ Finally since $P_{1,e}=\{1\}$ and $\phi_1$ depends on $z_1$ and $z_2$, $\phi$ does not have the restricted variables property at $e$ with respect to $\{1\}.$
\end{proof}

\begin{example}Let $f : \DD^3 \rightarrow \DD$ be defined by $f(z) = \dfrac{(1+z_1)^2+z_2z_3(1-z_1)^2}{4}$ for all $z\in \DD^3$ and consider
$$ \phi : \begin{array}{ccc}
    \DD^3 & \rightarrow & \DD^3   \\
     (z_1,z_2,z_3) & \mapsto & (f(z),f(z), 0 ).  
\end{array}
$$ 
Then $\phi\in\mathcal O(\DD^3,\DD^3)$, it does not satisfy the restricted variables property and $C_\phi$ is never bounded on $\Ab$ for any $\beta\geq -1.$
\end{example}
\begin{proof}
We first observe that $f$ maps $\DD^3$ into $\DD$ since, 
$$|f(z)| \leq \dfrac{|1+z_1|^2+|z_2||z_3| |1-z_1|^2}{4} \leq \dfrac{1+|z_1|^2}{2}< 1.$$
Considering $\xi = e$ and $I= \{1,2\}$, we have $\phi_I(\xi)  = (1,1) \in \TT^2. $ For $j\in I$, we have $\dfrac{\partial \phi_j(\xi)}{\partial z_1} = 1 $ and $\dfrac{\partial\phi_j(\xi)}{\partial z_2} = \dfrac{\partial\phi_j(\xi)}{\partial z_3} = 0$ so $P_{I,\xi} = \{1\}. $ However, for $j\in I$, $\phi_j$ depends on its three variables. By case 2c) of Theorem \ref{general}, we get that $C_\phi $ is not bounded on any $\Ab. $
\end{proof}


\section{Applications to composition operators on the bidisc}
\label{sec:bidisc}
\subsection{On the range of composition operators}

In this final section, we use our results to get new information on composition operators on the bidisc. Fix $\phi:\DD^2\to\DD^2$ a holomorphic map which belongs to $\mathcal C^1(\overline{\DD}^2)$. It is proved in \cite{SZ06b} that $C_\phi$ always maps $A^2_\beta(\DD^2)$ into $A^2_{2\beta+2}(\DD^2)$ (see also Theorem \ref{thm:automaticcontinuity}). Moreover, a jump phenomenon occurs at both sides of the interval $[\beta,2\beta+2]$:
\begin{itemize}
    \item in \cite{KSZ08}, it is shown that if $C_\phi$ does not map $A^2_\beta(\DD^2)$ into itself (namely if there exists $\xi\in\TT^2$ such that $\phi(\xi)\in\TT^2$ and $d\phi(\xi)$ is not invertible), then $C_\phi$ does not map $A^2_{\beta}(\DD^2)$ into $A^2_{\beta+\frac 14-\veps}(\DD^2)$ for any $\veps>0$ provided $\phi\in\mathcal C^2(\overline\DD^2)$.
    \item in \cite{KL25}, it is proved that if $C_\phi$ maps $A^2_\beta(\DD^2)$ into $A^2_{2\beta+2-\veps}(\DD^2)$ for some $\veps>0,$ then it automatically maps $A^2_{\beta}(\DD^2)$ into $A^2_{\beta+\frac 12}(\DD^2)$.    
\end{itemize}

Hence, this leads to introduce, for $\beta_1\geq -1,$ the set $\Lambda_{\beta_1}$ of real numbers $\beta_2\geq\beta_1$ such that there exists $\phi\in\mathcal O(\DD^2,\DD^2)\cap \mathcal C^1(\overline\DD^2)$ such that $C_\phi$ maps $A^2_{\beta_1}(\DD^2)$ into $A^2_{\beta}(\DD^2)$ if and only if $\beta\geq\beta_2.$ 
The above results show that $\beta_1+\frac 14\in \Lambda_{\beta_1}$ and that $((\beta_1+\frac 12,2\beta_1+2)\cup(2\beta_1+2,+\infty))\cap \Lambda_{\beta_1}=\varnothing.$
In \cite{KL25}, it is shown that $\beta_1+3/8\in \Lambda_{\beta_1}$ and the following conjecture is made: for any $m\in\NN$, $\beta_1+\frac 12-\frac 1{4m}$ belongs to $\Lambda_{\beta_1}$.

We are able to completely describe $\Lambda_{\beta_1}$ and thus to prove this conjecture.

\begin{theorem}\label{thm:conjectureKL}
Let $\beta_1\geq -1.$ Then $\Lambda_{\beta_1}=[\beta_1,\beta_1+1/2]\cup\{2\beta_1+2\}.$ 
\end{theorem}

\begin{proof}
For $\beta_2=2\beta_1+2$ (resp. $\beta_2=\beta_1+1/2$), the map $\phi(z)=(z_1,z_1)$ (resp. $\phi(z)=(z_1z_2,z_1z_2)$) does the job (see \cite[Examples 7.1]{KL25}). For $\beta_2=\beta_1$, the map $\phi(z)=((z_1+z_2)/2,(2z_1+z_2)/3)$ works, as any map $\phi$ such that $d\phi(\xi)$ is invertible at all contact points $\xi\in\TT^2$. For $\beta_2\in(\beta_1,\beta_1+1/2),$ write it $\beta_2=\beta_1+\frac 12-\frac1{2\kappa}$ with $\kappa>1.$
Let $\varphi\in\mathcal O(\DD,\DD)\cap \mathcal C^1(\overline{\DD})$ be such that for any $z\in\TT,$ $|\varphi(z)|=1$ if and only if $z=1$ with $\varphi'(1) > 0 $ and satisfying 
$$|\varphi(e^{i\theta})|=1-c|\theta|^{\kappa}+o(|\theta|^{\kappa})$$
for some $c>0.$ Define $\phi(z)=(z_1\varphi(z_2),z_1\varphi(z_2)).$ Then by Theorem \ref{thm:examples}, $C_\phi$ is continuous from $A_{\beta_1}^2(\DD^2)$ into $A_{\beta_2}^2(\DD^2)$ if and only if $\beta\geq\beta_2$.

For the converse inclusion, we use the quoted results of \cite{SZ06b} if $\beta_2>2\beta_1+2$ and of \cite{KL25} if $\beta_2\in(\beta_1+1/2,2\beta_1+2)$.
\end{proof}
\begin{remark}
    If $\beta_2=\beta_1+\frac 12-\frac 1{4n}$ for some $n\in\NN$, then we can choose $\phi\in\mathcal C^\infty(\overline{\DD}^2)$ by taking for $\varphi$ the function $h_n$ of Section \ref{sec:norder}.
\end{remark}

\subsection{A characterization of the extreme cases}

A natural problem arises: for $\phi\in\mathcal C^\infty(\overline{\DD}^2)$ and $n\in\NN\cup\{\infty\}$, is it possible to characterize when $C_\phi$ maps $A^2_\beta(\DD^2)$ into $A^2_{\beta+\frac 12-\frac 1{4n}}(\DD^2)$? In view of the examples of Section \ref{sec:norder}, this seems a very difficult task. Indeed, if we apply Theorem \ref{thm:examples} with $d=2$, $q=2$ and respectively $\varphi=g_{2n}$ and $\varphi=h_n$, and if we call $\phi$ and $\psi$ the induced self-maps of $\DD^2,$ then we observe that for any $\xi\in\partial\DD^2,$ $\phi(\xi)\in\partial \DD^2$ if and only if $\psi(\xi)\in\partial \DD^2$ and that at these points $\phi$ and $\psi$ share the same same derivatives up to order 
$n$. However, $C_\phi$ (resp. $C_\psi$) maps $A^2_{\beta_1}(\DD^2)$ into $A^2_{\beta_2}(\DD^2)$ if and only if $\beta_2\geq \beta_1+\frac 12-\frac1{8n}$ (resp. $\beta_2\geq \beta_1+\frac 12-\frac 1{4n}$).

However, for the extreme cases, this is doable and this was done in \cite{KL25}. In this section, in the spirit of the work of this paper and of \cite{Baytridisc}, we give an alternative characterization of symbols $\phi$ such that $C_\phi$ maps $A_\beta^2(\DD^2)$ into $A_{\beta+\frac 14}^2(\DD^2)$. Our characterization is simpler to verify.

Let $\phi:\DD^2\to\DD^2$ be holomorphic which extends to a $\mathcal C^3$-smooth function on $\ovdd$. Let $\xi\in\TT^2$ be such that $\phi(\xi)\in\TT^2$ and assume that $\nabla \phi_1(\xi)$ and $\nabla\phi_2(\xi)$
are linearly dependent (namely $J(\phi)(\xi)=0$). We assume that $\xi=e$, $\phi(\xi)=e$ and we write, for $j=1,2,$
\begin{align*}
 \Imm \phi_j(\theta)&=\kappa_j L(\theta)+o(|\theta_1|+|\theta_2|)  \\
\Ree \phi_j(\theta)&=1-Q_j(\theta)+O\left(|\theta_1|^3+|\theta_2|^3\right) 
\end{align*}
where $L$ is a nonzero linear form, $\kappa_j$ are nonzero real numbers and $Q_j$ are quadratic forms.  As for the tridisc, the signature of each $Q_j$ should be $(n_j,0)$ with $n_j\in\{0,1,2\}$ and there exists $s=s(\phi,\xi)\in\{1,2\}$ such that the signature of $Q_1+Q_2$ is equal to $(s,0)$. Taking $L_1,\dots,L_s$ a basis of $L_{j,k}$, where each $Q_j$ is written $Q_j=\dss_{k=1}^{n_j}L_{j,k}^2$, we know that $L$ belongs to $\textrm{span}(L_1,\dots,L_s)$. If $s=1$, we complete $(L_1)$ into a basis $(L_1,L_2)$ of $(\mathbb R^2)^*$ and we write 
$$\Imm \phi_j(\theta)=\kappa_j(L(\theta)+a_j L_1^2(\theta)+b_jL_1(\theta)L_2(\theta)+c_jL_2^2(\theta)\big)+O(|\theta_1|^3+|\theta_2|^3).$$
We define $r(\phi,\xi)=1$ if $c_1\neq c_2$, $r(\phi,\xi)=0$ if $c_1=c_2$.

\medskip

The values of $|P_{I,\xi}|,$ $s(\phi,\xi)$ and $r(\phi,\xi)$ are linked by the following lemma.
\begin{lemma}\label{lem:psr}
    Let $\phi\in\mathcal O(\DD^2,\DD^2)\cap\mathcal C^2(\overline{\DD}^2)$ and $\xi\in\TT^2$ such that $\phi(\xi)\in\TT^2$. If $|P_{I,\xi}|=1$ with $I=\{1,2\}$, then $s(\phi,\xi)=1$ and $r(\phi,\xi)=0.$
\end{lemma}
\begin{proof}
    We may assume $\xi=\phi(\xi)=e$ and $P_{I,e}=\{1\}$. Let $j\in\{1,2\}$. By Lemma \ref{JC} (d), 
    we may write
    $$\phi_j(z)=1+a_j(z_1-1)+b_j(z_1-1)^2+o(|z_1-1|^2+|z_2-1|^2).$$
    A Taylor expansion yields
    \begin{align*}
    \Ree\phi_j(e^{i\theta})&=1-\left(\frac{a_j}2+\Ree(b_j)\right)\theta_1^2+o(\theta_1^2+\theta_2^2)\\
    \Imm\phi_j(e^{i\theta})&=a_j\theta_1-\Imm(b_j)\theta_1^2+o(\theta_1^2+\theta_2^2).
    \end{align*}
    We immediately get the result from the definition of $s(\phi,e)$ and $r(\phi,e)$.
\end{proof}

\begin{theorem}\label{thm:beta1/4}
    Let $\phi\in \mathcal O(\DD^2,\DD^2) \cap \mathcal{C}^3(\overline \DD^2)$ and $\beta>-1.$ Then $C_\phi$ maps $A_\beta^2(\DD^2)$ into $A^2_{\beta+\frac 14}(\DD^2)$ if and only if, for all $\xi\in\TT^2$ such that $\phi(\xi)\in\TT^2$, 
    \begin{itemize}
        \item either $J(\phi)(\xi)\neq 0$;
        \item or $J(\phi)(\xi)=0$ and $s(\phi,\xi)=2$;
        \item or $J(\phi)(\xi)=0$ and $s(\phi,\xi)=r(\phi,\xi)=1$.
    \end{itemize}
\end{theorem}
\begin{proof}
We first show that if there exists $\xi \in \TT^2$ such that $\phi(\xi) \in \TT^2$ with  $\m{P_{I,e}} = 1$ then $C_\phi : A_\beta^2(\DD^2) \rightarrow A_{\beta+1/4}^2(\DD^2)$ is not bounded. Indeed, assume that $\phi(e)= e $ and  $P_{I,e} =\{1\}$. By $b)$ of Lemma \ref{JC} and the mean value theorem, we get $|\phi_j(z) - 1 | = |\phi_j(z)-\phi_j(1,z_2)| \leq \norm{\nabla\phi_j}|z_1-1|.$ Thus $\{ z\in\DD^2: |z_1-1|<\delta /M \} \subset \phi^{-1}(S(e,(\delta,\delta))). $ Moreover, $V_{\beta+1/4}(\{ z_1: |z_1-1|<\delta /M \}) \simeq \delta^{2+\beta+1/4} $ and $V_\beta(S(e,(\delta,\delta))) \simeq \delta^{2(2+\beta)}. $ 
 Thus, $C_\phi : A_\beta^2(\DD^2) \rightarrow A_{\beta+1/4}^2(\DD^2)$ is not bounded.

 We then show that $C_\phi : A_\beta^2(\DD^2) \rightarrow A_{\beta+1/4}^2(\DD^2)$ is not bounded if there exists $\xi \in \TT^2$ such that $\phi(\xi) \in \TT^2$ with $J(\phi)(\xi)=0$, $|P_{I,\xi}| =2,$ $s(\phi,\xi)=1$ and $r(\phi,\xi)=0$. Once again, we can assume that $\phi(e) = e.$ Since $J\phi(e)=0,$ $P_{1,\xi}=P_{2,\xi}=\{1,2\}$, so that $\phi$ has the restricted variables property at $\xi$ with respect to I. Doing a linear change of variables, we may assume that $L(\theta)=\theta_1$ and we obtain 
\begin{align*}
    \Imm \phi_j(\theta)&=\kappa_j(\theta_1+b_j\theta_1\theta_2+c\theta_2^2)+O(\theta_1^2+|\theta_2|^3)\\
    \Ree \phi_j(\theta)&=1+O(|\theta_1|^2+|\theta_2|^3).
\end{align*}
Let $\delta>0$, set $\overline\delta=(\delta,\delta)$ and let us choose $\theta_2\in[-\delta^{1/3},\delta^{1/3}]$
and $\theta_1$ such that $|\theta_1+c\theta_2^2|\leq \delta.$ It is obvious that $|\theta_1|\lesssim \delta^{2/3}$ so that, after an easy computation, 
$$|\Imm \phi_j(\theta)|\lesssim\delta\textrm{ and }|\Ree\phi_j(\theta)-1|\lesssim \delta.$$
Hence, 
$$\lambda_2\left(\left\{\theta\in[-\pi,\pi)^2:\ \phi(\theta)\in S(e,\overline\delta)\right\}\right)\gtrsim \delta^{1+\frac 13}.$$
By Corollary \ref{discont_restricted}, this prevents $C_\phi$ from mapping $A_\beta^2(\DD^2)$ into $A_{\beta'}^2(\DD^2)$ when $\beta'<\beta+\frac 13.$\\

Assume now that for all $\xi\in\TT^2$ with $\phi(\xi)\in\TT^2$, the conditions of Theorem \ref{thm:beta1/4} are satisfied. Let us consider $\xi\in\ovdd$ with $\phi(\xi)\in\partial \DD^2$. Let $I\subset\{1,2\}$ be maximal with $\phi_I(\xi)\in\TTI.$ We aim to show 
that there exist $\mathcal{U} $ a neighbourhood of $\xi$ in $\ovdd$  and $C>0$ such that for all $\eta \in \TT^{|I|}$, for all $\ovd\in(0,1)^{|I|},$
\begin{equation}\label{eq:beta1/4}
V_{\beta+1/4} (\Ei \cap \mathcal{U}) \leq C\prod_{j\in I}\delta_j^{2+\beta}. 
\end{equation}
If $\phi(\xi)\in \partial \DD^2\backslash \TT^2$, let $I=\{j\}$ be such
that $\phi_j(\xi)\in \TT.$ Then by Lemma \ref{majovbeta}, there exist $\mathcal{U} $ a neighbourhood of $\xi$ in $\ovdd$  and $C>0$ such that for all $\eta \in \TT$ and all $\ovd\in(0,1)$, 
$$V_{\beta+1/4}(\Ei\cap\mathcal U)\lesssim V_\beta(\Ei \cap \mathcal{U}) \leq C\delta_j^{2+\beta}. $$ 
We may appeal to the same lemma if $\phi(\xi)\in \TT^2$ with $J(\phi)(\xi)\neq 0.$

Suppose now that $J(\phi)(\xi)=0$. If $\xi\notin\TT^2$, then $|P_{I,\xi}|=1$ since $|\xi_k|=1$ for $k\in P_{I,\xi}$. By Lemma \ref{JC} (b) and (c), $\phi(\hat\xi)=\phi(\xi)$ and $\nabla\phi_j(\hat\xi)=\nabla\phi_j(\xi)$, so that
$J(\phi)(\hat\xi)=0$ and $|P_{I,\hat\xi}|=1$. By Lemma \ref{lem:psr}, $s(\phi,\hat\xi)=1$ and $r(\phi,\hat\xi)=0$,
which contradicts our assumption at $\hat\xi$. Hence $\xi\in\TT^2$, and the same lemma shows that $|P_{I,\xi}|=2$,
so that $P_{1,\xi}=P_{2,\xi}=\{1,2\}$ and $\phi$ has the restricted variables property at $\xi$ with respect to $I$.
We assume that $\xi =\phi(\xi)  =e.$  We intend to show that there exist a neighbourhood $\mathcal{O}$ of $\xi$ in $\TT^2$ and $C>0$ such that, for all $\eta\in\TT^2$ and all $\ovd\in(0,1)^2,$
\begin{equation*}\label{eq:beta1/4lambda}
\lambda_2(O\cap \phi^{-1}(S(\eta,\overline\delta)))\leq C\delta_1^{\frac 34}\delta_2^{\frac 34}
\end{equation*}
with $\dfrac{3}{4} + (1+\beta+1/4) \geq 2+\beta$ (actually, here we have $\gamma_{\{1,2\},e} = \frac{|P_{\{1,2\},e}|}{2} = 1)$.
Then we use \eqref{eq:wikalpha} and Lemma \ref{lem:comparaisontwoCarlesonsets} to get \eqref{eq:beta1/4}.

Without loss of generality, we assume $\delta_1\leq\delta_2.$ First, assume that $s(\phi,e)=2$. Doing a linear change of variables, we may assume that $L(\theta)=\theta_1$ and we write 
$$\Imm \phi_j(\theta)=\kappa_j\left(\theta_1+G_j(\theta)\right)$$
with $G_j(0)=0,$ $\nabla G_j(0)=0.$ We set $U(\theta)=(\theta_1+G_1(\theta),\theta_2)$ and observe that there exist two
neighbourhoods $ O_1$ and $ O_2$ of $0$ and $C_1>0$ such that
$U$ is a diffeomorphism from $ O_1$ onto $ O_2$ and 
$$\left\{
\begin{array}{rcll}
|\det(\mathrm dU)|&\leq& C_1&\textrm{on } O_1\\
|\det(\mathrm dU^{-1})|&\leq& C_1&\textrm{on } O_2.
\end{array}\right.$$
We then write
$$\Ree \phi_j(\theta)=1-Q_j(\theta)+H_j(\theta)$$
with $H_j(\theta)=O\left(|\theta_1|^3+|\theta_2|^3\right).$
Because $\mathrm dU(0)=I_2$, for $u\in  O_2,$
$$\Ree \phi_j\circ U^{-1}(u)=1-Q_j(u)+K_j(u)$$
with $K_j(u)=O\left(|u_1|^3+|u_2|^3\right).$
 Let us now define
$$E(\ovd)=\left\{\theta\in\RR^2:\ |\phi_j(\theta)-\eta_j|\leq\delta_j,\ j=1,2\right\}$$
and let us write $\eta_j=x_j+iy_j.$ Provided $u\in U(E(\ovd)\cap O_1),$
\begin{eqnarray}
\nonumber&&|-Q_1(u)-Q_2(u)+K_1(u)+K_2(u)+2-x_1-x_2|\\
&&\quad\quad\quad\quad\leq|\Ree \phi_1\circ U^{-1}(u)-x_1|+|\Ree \phi_2\circ U^{-1}(u)-x_2| \nonumber \\
&&\quad\quad\quad\quad\leq2\delta_2. \label{eq:volumelinearlydependent2}
\end{eqnarray}
We set $F(u)=-Q_1(u)-Q_2(u)+K_1(u)+K_2(u)$. Since the signature of $Q_1+Q_2$ is $(2,0)$, ${\frac{\partial^2 F}{\partial u_2^2}(0)<0}$.
Hence we may apply the parametrized Morse lemma: there exist two neighbourhoods $ O_2'\subset O_2$ 
and $ O_3$ of $0$, a constant $C_2>0$, a diffeomorphism $V: O'_2\to O_3,$ ${V(u_1,u_2)=(u_1,V_2(u_1,u_2))}$ and a $\mathcal C^1$-map
$h:\mathbb R\to\mathbb R$ such that, for all $v\in O_3,$ 
$${F\circ V^{-1}(v)=-v_2^2+h(v_1)}\textrm{ and } 
\left\{
\begin{array}{rcll}
|\det(\mathrm dV)|&\leq C_2&\textrm{on } O'_2\\
|\det(\mathrm dV^{-1})|&\leq C_2&\textrm{on } O_3.
\end{array}\right.$$
We finally set $ O=U^{-1}( O'_2)\subset O_1.$ Therefore,
provided $v\in V\circ U(E(\ovd)\cap O),$ we infer from $|\Imm \phi_1-\Imm \eta_1|\leq\delta_1$ on $E(\delta)$ and from \eqref{eq:volumelinearlydependent2} that
$$\left\{
\begin{array}{rcl}
|v_1-y_1|&\leq&\delta_1/\kappa_1\\
|-v_2^2+h(v_1)+2-x_1-x_2|&\leq&2\delta_2.
\end{array}\right.$$
It then follows from Fubini's theorem that 
$$\lambda_2\big(V\circ U(E(\ovd)\cap O)\big)\lesssim\delta_1\delta_2^{1/2}\leq \delta_1^{3/4}\delta_2^{3/4}.$$
Assume finally $s(\phi,e)=1$ and $r(\phi,e)=1$. We begin by arguing as above and we write 
\begin{align*}
    \Imm\phi_j(\theta)=\kappa_j(\theta_1+a_j\theta_1^2+b_j\theta_1\theta_2+c_j\theta_2^2)+O(|\theta_1|^3+|\theta_2|^3)
\end{align*}
with $c_1\neq c_2$. Doing the same change of variables and using $dU(0)=I_2$, we get
\begin{align*}
\Imm \phi_1\circ U^{-1}(u)&=\kappa_1 u_1\\
\Imm \phi_2\circ U^{-1}(u)&=\kappa_2(u_1+a_2'u_1^2+(b_2-b_1)u_1u_2+(c_2-c_1)u_2^2)+K(u_1,u_2)
\end{align*}
with $K(u_1,u_2)=O(|u_1|^3+|u_2|^3)$ and $a'_2=a_2-a_1$. We apply the parametrized Morse lemma to $\Imm \phi_2\circ U^{-1}$ to get a diffeomorphism $V: O'_2\to O_3,$ $V(u_1,u_2)=(u_1,V_2(u_1,u_2))$ and a $\mathcal C^1$-map
$h:\mathbb R\to\mathbb R$ such that, for all $v\in O_3,$ 
\begin{align*}
    \Imm \phi_2\circ(V\circ U)^{-1}(v)&=\kappa_2(\pm v_2^2+h(v_1))\\
    \Imm \phi_1\circ(V\circ U)^{-1}(v)&=\kappa_1 v_1.
\end{align*}
We then conclude as above, using only the imaginary parts $\Imm\phi_1$ and $\Imm\phi_2$, that
$$\lambda_2\big(V\circ U(E(\ovd)\cap O)\big)\lesssim\delta_1\delta_2^{1/2}\leq \delta_1^{3/4}\delta_2^{3/4}.$$

We have shown that \eqref{eq:carlesonbord} holds with $\beta_1=\beta$ and $\beta_2=\beta+\frac14$ for all
$\xi\in\overline{\DD}^2$ and all maximal $I$, so that Theorem \ref{result_gen} yields that $C_\phi$ maps
$A^2_\beta(\DD^2)$ into $A^2_{\beta+\frac14}(\DD^2)$.
\end{proof}

Let us finish this subsection with two examples that illustrate our results.

\begin{example}
Let $\phi(z)=\big((z_1+z_2)/2,(z_1+z_2)/2\big).$ Then for all $\beta> -1,$ $C_\phi$ maps $A_\beta^2(\DD^2)$ into $A_{\beta+\frac14}^2(\DD^2).$
\end{example}
\begin{proof}
    Indeed, let $\xi\in\TT^2$ be such that $\phi(\xi)\in\TT^2.$ By symmetry we may assume $\xi=e.$ Then for $j=1,2,$
\begin{align*}
    \Imm(\phi_j)&=\frac12(\theta_1+\theta_2)+O(\theta_1^2+\theta_2^2)\\
    \Ree(\phi_j)&=1-\frac 14(\theta_1^2+\theta_2^2)+O(|\theta_1|^3+|\theta_2|^3).
\end{align*}
Therefore, $s(\phi,e)=2$ and the result follows from Theorem \ref{thm:beta1/4}.
\end{proof}

For our second example, we recall that, for $\veps\geq 0,$ the map $F_\veps$ defined on $\CC$ by
$$F_\veps(z)=\frac{3+6z-z^2}8+2i\veps(z-1)^2-i\veps(z-1)^3$$
is a holomorphic self-map of $\DD$ provided $\veps\in[0,\veps_1)$ for
some $\veps_1>0.$

\begin{example}
Let $\beta> -1,$ $a_1,a_2\in[0,\veps_1)$ and let 
$$\phi_{a_1,a_2}(z)=(F_{a_1}(z_1)F_{a_1}(z_2),F_{a_2}(z_1)F_{a_2}(z_2)).$$
Then $C_{\phi_{a_1,a_2}}$ maps $A_\beta^2(\DD^2)$ into $A_{\beta+\frac 14}^2(\DD^2)$
if and only if $a_1\neq a_2.$
\end{example}
\begin{proof}
Observe first that for $\xi\in\TT^2,$ $\phi_{a_1,a_2}(\xi)\in\TT^2\iff \xi=e.$
We may write, for $j\in\{1,2\},$
\begin{align*}
    \Imm \phi_j(\theta)&=\frac12(\theta_1+\theta_2)-a_j(\theta_1^2+\theta_2^2)+O(|\theta_1|^3+|\theta_2|^3)\\
    \Ree \phi_j(\theta)&=1-\frac 18(\theta_1+\theta_2)^2+O(|\theta_1|^3+|\theta_2|^3).
\end{align*}
Therefore, $s(\phi_{a_1,a_2},e)=1$ and it is easy to check that $r(\phi_{a_1,a_2},e)=1$ if and only if $a_1\neq a_2.$
\end{proof}

\providecommand{\bysame}{\leavevmode\hbox to3em{\hrulefill}\thinspace}
\providecommand{\MR}{\relax\ifhmode\unskip\space\fi MR }
\providecommand{\MRhref}[2]{%
  \href{http://www.ams.org/mathscinet-getitem?mr=#1}{#2}
}
\providecommand{\href}[2]{#2}


\begin{thebibliography}{10}

\bibitem{BAYPOLY}
F.~Bayart, \emph{Composition operators on the polydisk induced by affine maps}, J. Funct. Anal. \textbf{260} (2011), 1969--2003.

\bibitem{Baytridisc}
\bysame, \emph{{Composition operators on the Hardy space of the tridisc}}, Indiana U. Math. J. \textbf{75} (2026), 105--139.

\bibitem{Bes25b}
A.~Beslikas, \emph{{Composition operators and Rational Inner Functions on the bidisc}}, Proc. Am. Math. Soc. \textbf{153} (2025), 3491--3502.

\bibitem{Bes26}
\bysame, \emph{{Composition operators and Rational Inner Functions on the bidisc II: boundedness between two different Bergman spaces}}, preprint (2025), arXiv:2509.04366.

\bibitem{Bes25a}
\bysame, \emph{A note on composition operators on the disc and bidisc}, Canadian Math. Bull. \textbf{68} (2025), 1163--1176.

\bibitem{BG92}
J.~W. Bruce and P.~J. Giblin, \emph{Curves and {S}ingularities: a geometrical introduction to singularity theory}, Cambridge University Press, 1992.

\bibitem{CCS25}
T.G. Clos, Ž. Čučković, and S~Şahutoğlu, \emph{{Compactness of composition operators on the Bergman space of the bidisc}}, Integr. Equ. Oper. Theory \textbf{97} (2025), 21.

\bibitem{Jaf90}
F.~Jafari, \emph{On bounded and compact composition operators in polydiscs}, Can. J. Math \textbf{42} (1990), 869--889.

\bibitem{Jaf91}
\bysame, \emph{{Carleson measures in Hardy and weighted Bergman spaces of polydiscs}}, Proc. Amer. Math. Soc. \textbf{112} (1991), 771--781.

\bibitem{KL25}
H.~Koo and S-Y. Li, \emph{Composition operators on the bidisc}, J. Math. Anal. Appl. \textbf{554} (2025), 129902.

\bibitem{KSZ08}
H.~Koo, M.~Stessin, and K.~Zhu, \emph{Composition operators on the polydisc induced by smooth symbols}, J. Funct. Anal. \textbf{254} (2008), 2911--2925.

\bibitem{Ko22}
L.~Kosiński, \emph{{Composition operators on the polydisc}}, J. Funct. Anal. \textbf{284} (2023), 109801.

\bibitem{LQ26}
P.~Lefèvre and H.~Queffélec, \emph{A primer, and beyond, on composition operators on the unit disk}, New mathematical monographs, vol.~53, Cambridge University Press, 2026.

\bibitem{SZ06b}
M.~Stessin and K.~Zhu, \emph{Composition operators induced by symbols defined on a polydisk}, J. Math. Anal. Appl. \textbf{319} (2006), 815--829.

\bibitem{SZ06a}
\bysame, \emph{{Composition operators on embedded disks}}, J. Op. Th. \textbf{56} (2006), 423--449.

\end{thebibliography}
\end{document}